\documentclass[a4paper,11pt]{amsart}
\usepackage{amsmath,amssymb,amsthm}
\usepackage{graphics,color}
\usepackage{epsfig}

\graphicspath{{fig/}}

\newtheorem{dfn}{Definition}[section]
\newtheorem{thm}[dfn]{Theorem}

\newtheorem{lem}[dfn]{Lemma}
\newtheorem{cor}[dfn]{Corollary}
\newtheorem{conj}[dfn]{Conjecture}
\newtheorem{pro}[dfn]{Problem}

\theoremstyle{definition}
\newtheorem{exl}[dfn]{Example}
\newtheorem{rem}[dfn]{Remark}

\def\R{{\mathbb R}}

\def\Sph{{\mathbb S}}
\def\H{{\mathbb H}}

\def\cN{\mathcal{N}}
\def\cF{\mathcal{F}}
\def\cC{\mathcal{K}}
\def\epsilon{\varepsilon}

\def\a{\operatorname{area}}

\def\vol{\operatorname{vol}}
\def\area{\operatorname{area}}

\def\im{\operatorname{im}}
\def\phi{\varphi}
\def\emptyset{\varnothing}

\def\pos{\operatorname{pos}}
\def\relint{\operatorname{relint}}

\def\span{\operatorname{span}}
\def\ind{\operatorname{ind}}
\def\ir{\operatorname{ir}}
\def\co{\operatorname{co}}
\def\int{\operatorname{int}}
\def\aff{\operatorname{aff}}
\def\rank{\operatorname{rk}}
\def\conv{\operatorname{conv}}
\def\clir{\operatorname{clir}}
\def\one{\mathbf 1}
\def\zero{\mathbf 0}
\def\Ch{\operatorname{Ch}}
\def\GL{\operatorname{GL}}
\def\per{\operatorname{per}}
\def\dist{\operatorname{dist}}
\def\supp{\operatorname{supp}}
\def\grad{\operatorname{grad}}
\def\core{\operatorname{core}}
\def\Vert{\operatorname{Vert}}

\def\wT{\widetilde{T}}
\def\wir{\widetilde{\ir}}
\def\wclir{\widetilde{\clir}}
\def\wco{\widetilde{\co}}
\def\cl{\operatorname{cl}}

\title{Shapes of polyhedra, mixed volumes, and hyperbolic geometry}
\author{Fran\c{c}ois Fillastre}
\address{University of Cergy-Pontoise\\ UMR CNRS 8088\\ departement of mathematics\\
F-95000 Cergy-Pontoise\\ FRANCE}
\email{francois.fillastre@u-cergy.fr}

\author{Ivan Izmestiev}
\thanks{Supported by the European Research Council under the European Union's Seventh Framework Programme (FP7/2007-2013)/\allowbreak ERC Grant agreement no.~247029-SDModels}
\address{Institut f\"ur Mathematik \\
Freie Universit\"at Berlin \\
Arnimallee 2 \\
D-14195 Berlin \\
 GERMANY}
\email{izmestiev@math.fu-berlin.de}

\date{\today}

\begin{document}

\begin{abstract}
We are generalizing to higher dimensions the Bavard-Ghys construction of the hyperbolic metric on the space of polygons with fixed directions of edges.

The space of convex $d$-dimensional polyhedra with fixed directions of facet normals has a decomposition into type cones that correspond to different combinatorial types of polyhedra. This decomposition is a subfan of the secondary fan of a vector configuration and can be analyzed with the help of Gale diagrams.

We construct a family of quadratic forms on each of the type cones using the theory of mixed volumes. The Alexandrov-Fenchel inequalities ensure that these forms have exactly one positive eigenvalue. This introduces a piecewise hyperbolic structure on the space of similarity classes of polyhedra with fixed directions of facet normals.
We show that some of the dihedral angles on the boundary of the resulting cone-manifold are equal to $\frac{\pi}2$.
\end{abstract}

\maketitle

\setcounter{tocdepth}{2}
\tableofcontents

\section*{Introduction}
\subsection*{Motivation}
In \cite{Thu98}, Thurston put a complex hyperbolic structure on the space $C(\alpha)$ of Euclidean metrics on the sphere with fixed cone angles $\alpha = (\alpha_1, \ldots, \alpha_n)$ by showing that the area of the metric is a Hermitian form with respect to some natural choices of coordinates on the space of metrics. This provided a more elementary approach to the Deligne-Mostow examples of non-arithmetic complex Coxeter orbifolds.

Bavard and Ghys \cite{BG92} adapted Thurston's construction to the planar case by considering the space of convex polygons with fixed angles or, equivalently, with fixed directions of edges. On this space, the area of a polygon turns out to be a real quadratic form (with respect to the edge lengths) of the signature $(+, -, \ldots, -)$. This turns the space of polygons with fixed edge directions into a hyperbolic polyhedron. By computing the dihedral angles of this polyhedron, Bavard and Ghys were able to obtain all hyperbolic Coxeter orthoschemes from the list previously drawn by Im Hof \cite{ImHof1}. For more details, see \cite{FillEM}. The Bavard-Ghys polyhedron can be viewed as a subset of Thurston's space, since gluing together two copies of a polygon along the boundary yields a Euclidean cone-metric on the sphere.

In the present paper, we are generalizing the construction of Bavard and Ghys to higher dimensions. Namely, we consider the space of $d$-dimensional polyhedra with fixed directions $V$ of the facet normals and exhibit a family of quadratic forms that makes it to a hyperbolic cone-manifold $M(V)$ with polyhedral boundary. We also obtain partial information about the cone angles in the interior and dihedral angles on the boundary of $M(V)$.

\subsection*{Outline of the paper}
The study of the space of $d$-dimensional polyhedra with fixed facet normals poses some problems that are missing in the case of polygons $d=2$. First, the facet normals don't determine the combinatorial structure of the polyhedron anymore. In order to analyze the space of polyhedra, we employ machinery from discrete geometry, namely Gale diagrams and secondary polyhedra \cite{GKZ94}. This constitutes Section \ref{sec:Shapes} of our paper.

Second, one needs to introduce a quadratic form and show that it has a hyperbolic signature. For $d=3$ one can still take the surface area, but for $d > 3$ a new construction is needed. This is provided by mixed volumes, which yield a whole family of quadratic forms of hyperbolic signature, even for $d=3$. The signature of the quadratic form is ensured by the Alexandrov-Fenchel inequalities. This aspect is discussed in Section \ref{sec:MixVol}.

Finally, the dihedral angles are now more difficult to compute. Positive and negative results in this direction are contained in Section \ref{sec:HypGeom}.

\begin{figure}[ht]
\centering
\includegraphics[width=\textwidth]{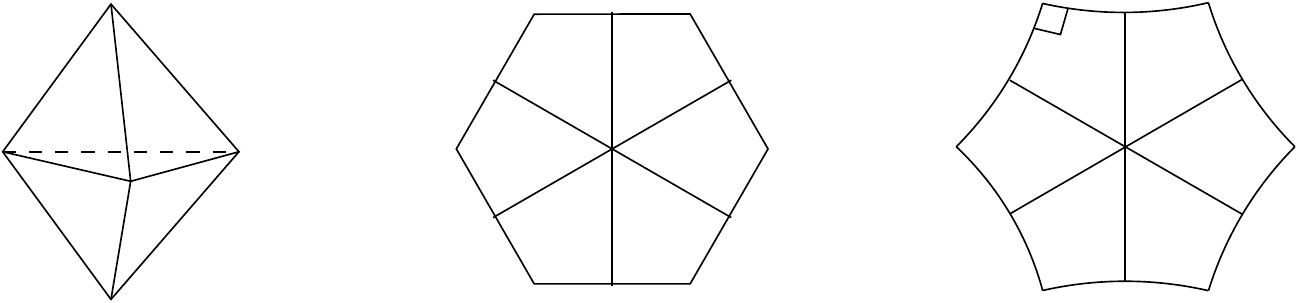}
\caption{The space of polyhedra with face normals parallel to those of a triangular bipyramid is a right-angled hyperbolic hexagon.}
\label{fig:Example}
\end{figure}

All along the paper we are analyzing several examples. One of them is given by the space of $3$-dimensional polyhedra with $6$ faces whose normals are parallel to those of a triangular bipyramid, Figure \ref{fig:Example} left. Translating each face independently yields $6$ different combinatorial types, and the space of all polyhedra up to similarity forms a $2$-dimensional polyhedral complex on Figure \ref{fig:Example} middle. The surface area is a quadratic form of signature $(+,-,-)$ on the space of polyhedra, which puts a hyperbolic metric on the complex making it a right-angled hyperbolic hexagon, Figure \ref{fig:Example} right.

\subsection*{Related work}
A different generalization (and dualization) of the Bavard-Ghys construction was given by Kapovich and Millson \cite{KM95,KM96} who considered the space of polygonal lines in $\R^3$ with fixed edge lengths.

By Alexandrov's theorem \cite{Al42}, every Euclidean metric with cone angles $\alpha_i < 2\pi$ is the intrinsic metric on the boundary of a unique convex polyhedron in $\R^3$. Thus Thurston's space is the space of all convex polyhedra with fixed solid exterior angles at the vertices. A link between our construction and that of Thurston is discussed in Section \ref{sec:RelWork} that contains also a number of other connections to discrete and hyperbolic geometry.


\subsection*{Acknowledgements}
The authors thank Haruko Nishi, Arnau Padrol, Francisco Santos, and Raman Sanyal for interesting discussions.
Part of this work was done during the visit of the first author to FU Berlin and the visit of the second author to the university of Cergy-Pontoise. We thank both institutions for hospitality.


\section{Shapes of polyhedra}
\label{sec:Shapes}
\subsection{Normally equivalent polyhedra and type cones}
\label{sec:TypeCones}
\subsubsection{Convex polyhedra, polytopes, cones}
A \emph{convex polyhedron} $P \subset \R^d$ is an intersection of finitely many closed half-spaces. A bounded convex polyhedron is called \emph{polytope}; equivalently, a polytope is the convex hull of finitely many points.

By $\aff(P)$ we denote the affine hull of $P$, that is the smallest affine subspace of $\R^d$ containing $P$. The \emph{dimension} of a convex polyhedron is the dimension of its affine hull.

A hyperplane $H$ is called a \emph{supporting hyperplane} of a convex polyhedron $P$ if $H \cap P \ne \emptyset$ while $P$ lies in one of the closed half-spaces bounded by $H$. The intersection $P \cap H$ is called a \emph{face} of $P$ (sometimes it makes sense to consider $\emptyset$ and $P$ as faces of $P$, too). Faces of dimensions $0$, $1$, and $\dim P - 1$ are called vertices, edges, and \emph{facets}, respectively. A $d$-dimensional polytope is called \emph{simple}, if each of its vertices belongs to exactly $d$ facets (equivalently, to exactly $d$ edges).

A \emph{convex polyhedral cone} is the intersection of finitely many closed half-spaces whose boundary hyperplanes pass all through the origin. Equivalently, it is the positive hull
$$
\pos\{w_1, w_2, \ldots, w_k\} := \left\{\sum_{i=1}^k \lambda_i w_i \mid \lambda_i \ge 0,\, i = 1, \ldots, k\right\}
$$
of finitely many vectors. A cone is called \emph{pointed}, if it contains no linear subspaces except for $\{0\}$; equivalently, if it has $\{0\}$ as a face. A cone $C$ is called \emph{simplicial} if it is the positive hull of $k$ linearly independent vectors. In this case the intersection of $C$ with an appropriately chosen hyperplane is a $(k-1)$-simplex.

See \cite{Ewald96} or \cite{Zie95} for more details on polyhedra, cones, and fans.

\subsubsection{Normal fans and normally equivalent polyhedra}
\begin{dfn}
\label{dfn:Fan}
A \emph{fan} $\Delta$ in $\R^d$ is a collection of convex polyhedral cones such
that
\begin{enumerate}
\item
if $C \in \Delta$, and $C'$ is a face of $C$, then $C' \in \Delta$;
\item
if $C_1, C_2 \in \Delta$, then $C_1 \cap C_2$ is a face of both $C_1$ and $C_2$.
\end{enumerate}
\end{dfn}
By $\Delta^{(k)}$ we denote the collection of $k$-dimensional cones of the fan $\Delta$. The \emph{support} of a fan is the union of all of its cones:
$$
\supp(\Delta) := \bigcup_{\sigma \in \Delta} \sigma
$$
A fan $\Delta$ is called \emph{complete} if $\supp(\Delta) = \R^d$, and it is called \emph{pointed}, respectively \emph{simplicial} if all of its cones are pointed, respectively simplicial. Complete simplicial fans in $\R^d$ are in 1-1 correspondence with 
geodesic triangulations of $\Sph^{d-1}$.

For every proper face $F$ of a convex polyhedron $P$ there is a supporting hyperplane through $F$. The set of outward normals to all such hyperplanes spans a convex polyhedral cone, the \emph{normal cone} at $F$. Below is a more formal definition.

\begin{dfn}
\label{dfn:NormFan}
Let $P \subset \R^d$ be a convex polyhedron, and let $F$ be a non-empty face of $P$. The \emph{normal cone} $N_F(P)$ of $P$ at $F$ is defined as
$$
N_F(P) := \{v \in \R^d \mid \max_{x \in P} \langle v, x \rangle = \langle v, p \rangle \ \forall p \in F\},
$$
where $\langle \cdot, \cdot \rangle$ is the standard scalar product in $\R^d$.
The \emph{normal fan} of $P$ is the collection of all normal cones of $P$:
$$
\cN(P) := \{N_F(P) \mid F \text{ a proper face of }P\}
$$
\end{dfn}
Note that Definition \ref{dfn:NormFan} makes sense also when $\dim P < d$. In this case every $N_F(P)$ contains the linear subspace $\aff(P)^\perp$.
The following are some simple facts about the normal fan.
\begin{itemize}
\item
$\supp(\cN(P))$ is a convex polyhedral cone in $\R^n$ positively spanned by the normals to the facets of $P$.
\item
If $F$ is a face of $G$, which is a face of $P$, then $N_G(P)$ is a face of $N_F(P)$.
\item
$\dim N_F(P) = d - \dim F$
\item
$P$ is a polytope $\Leftrightarrow \cN(P)$ is complete
\item
$\dim P = d \Leftrightarrow \cN(P)$ is pointed
\item
$P$ is a simple $d$-polytope $\Leftrightarrow \cN(P)$ is complete and simplicial
\end{itemize}

%

\begin{dfn}
Two convex polyhedra $P, Q \subset \R^d$ are called \emph{normally equivalent} if they have the same normal fan:
$$
P \simeq Q \Leftrightarrow \cN(P) = \cN(Q)
$$
For a fan $\Delta$, we denote by $\wT(\Delta)$ the set of all convex polyhedra with the normal fan $\Delta$:
$$
\wT(\Delta) := \{P \mid \cN(P) = \Delta\},
$$
and by $T(\Delta)$ the set of all such polyhedra modulo translation:
$$
T(\Delta) := \wT(\Delta)/\sim, \quad \text{where }P \sim P+x \ \forall x \in \R^d
$$
The set $T(\Delta)$ is called the \emph{type cone} of $\Delta$.
\end{dfn}
The set $T(\Delta)$ is a cone in the sense that scaling a convex polyhedron does not change its normal fan. Later we will see that the closure of $T(\Delta)$ is in fact a convex polyhedral cone.

Not every fan $\Delta$ is the normal fan of some convex polyhedron, examples are shown on Figure \ref{fig:notfan}.
If it is, that is $T(\Delta) \ne \emptyset$, then $\Delta$ is called \emph{polytopal}.

Normally equivalent polyhedra are also called ``analogous''~\cite{Ale37} or ``strongly isomorphic'' \cite{McM73, McM96}. The term ``normally equivalent'' is used in \cite{LRS10}.

\subsubsection{Support numbers}
\label{sec:SuppNumb}
Let $P \subset \R^d$ be a $d$-dimensional convex polyhedron. Then $1$-dimensional cones of $\cN(P)$ correspond to facets of $P$. Denote by $V := (v_1, \ldots, v_n)$
the collection of unit vectors that generate $\cN(P)^{(1)}$, and by $F_i$ the facet of $P$ with the outward unit normal $v_i$. Then $P$ is the solution set of a system of linear inequalities $\langle v_i, x \rangle \le h_i$, $i = 1, \ldots, n$ for some $h \in \R^n$. We will express this as $P = P(V,h)$, where
\begin{equation}
\label{eqn:SuppNumb}
P(V,h) := \{x \in \R^d \mid Vx \le h\}, \quad V \in \R^{n \times d},\, h \in \R^n
\end{equation}
Here, by abuse of notation, $V$ denotes the $n \times d$-matrix whose $i$-th row is~$v_i$. By a repeated abuse of notation, we will also write $V = \cN(P)^{(1)}$. Thus $V$ stands for any of the following objects:
\begin{itemize}
\item
a collection of $n$ rays in $\R^d$ starting at the origin;
\item
a collection of $n$ unit vectors in $\R^d$;
\item
an $n \times d$-matrix with rows of norm $1$.
\end{itemize}

From $\|v_i\| = 1$ it follows that $h_i$ is the signed distance from the coordinate origin to the affine hull of $F_i$, see Figure \ref{fig:SuppNumb}. The numbers $h_i$ are called the \emph{support numbers} of the polyhedron $P$, and $h = (h_1, \ldots, h_n) \in \R^n$ the \emph{support vector} of $P$.

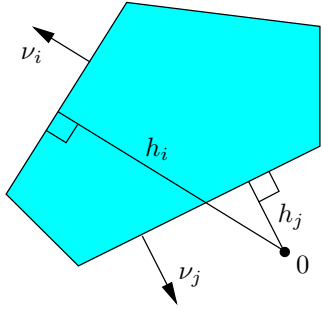
\begin{figure}[ht]
\centering
\begin{picture}(0,0)%
\includegraphics{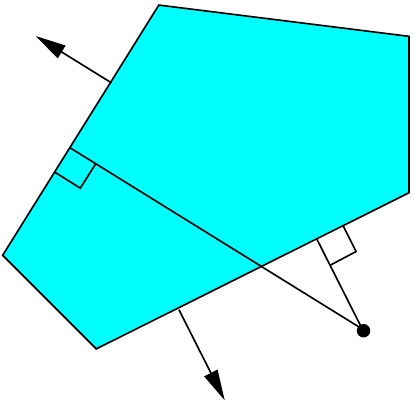}%
\end{picture}%
\setlength{\unitlength}{4144sp}%
\begingroup\makeatletter\ifx\SetFigFont\undefined%
\gdef\SetFigFont#1#2#3#4#5{%
  \reset@font\fontsize{#1}{#2pt}%
  \fontfamily{#3}\fontseries{#4}\fontshape{#5}%
  \selectfont}%
\fi\endgroup%
\begin{picture}(1882,1831)(79,-1154)
\put(1699,-644){\makebox(0,0)[lb]{\smash{{\SetFigFont{10}{12.0}{\familydefault}{\mddefault}{\updefault}{\color[rgb]{0,0,0}$h_j$}%
}}}}
\put(915,-254){\makebox(0,0)[lb]{\smash{{\SetFigFont{10}{12.0}{\familydefault}{\mddefault}{\updefault}{\color[rgb]{0,0,0}$h_i$}%
}}}}
\put(1111,-988){\makebox(0,0)[lb]{\smash{{\SetFigFont{10}{12.0}{\familydefault}{\mddefault}{\updefault}{\color[rgb]{0,0,0}$\nu_j$}%
}}}}
\put(179,325){\makebox(0,0)[lb]{\smash{{\SetFigFont{10}{12.0}{\familydefault}{\mddefault}{\updefault}{\color[rgb]{0,0,0}$\nu_i$}%
}}}}
\put(1809,-951){\makebox(0,0)[lb]{\smash{{\SetFigFont{10}{12.0}{\familydefault}{\mddefault}{\updefault}{\color[rgb]{0,0,0}$0$}%
}}}}
\end{picture}%
\caption{Support numbers: $h_i > 0$, $h_j < 0$.}
\label{fig:SuppNumb}
\end{figure}

For any pointed polytopal fan $\Delta$, the support vector determines an embedding
$$
\wT(\Delta) \to \R^n, \quad P(V,h) \mapsto h,
$$
where $V = \Delta^{(1)}$. Due to
\begin{equation}
\label{eqn:PTransl}
P(V, h) + t = P(V, h + Vt)
\end{equation}
the equivalence classes modulo translation correspond to points in $\R^n/\im V$. Thus we have
$$
T(\Delta) \subset \R^n/\im V,
$$
and $\wT(\Delta) = \pi^{-1}(T(\Delta))$, where
$$
\pi \colon \R^n \to \R^n/\im V
$$
is the canonical projection.

\subsubsection{Support function}
\begin{dfn}
Let $P \subset \R^d$ be a convex polyhedron. Its \emph{support function} is defined as
\begin{gather*}
h_P \colon \supp(\cN(P)) \to \R\\
h_P(v) := \max_{x \in P} \langle v, x \rangle
\end{gather*}
\end{dfn}
In particular, if $P$ is a polytope, then $h_P$ is defined on the whole $\R^d$.

The support function is positively homogeneous and convex, that is $h_P(\lambda v) = \lambda h_P(v)$ for $\lambda > 0$, and
$$
h_P(\lambda v + (1-\lambda) w) \le \lambda h_P(v) + (1-\lambda) h_P(w)
$$
for all $\lambda \in [0,1]$, $v,w \in \supp(\cN(P))$. The support function can be defined for any closed convex set, and there is a 1-1 correspondence between closed convex sets and positively homogeneous convex functions defined on a convex cone in $\R^d$.

\begin{lem}
\label{lem:SuppFunc}
The support function of a convex polyhedron $P$ is given by
$$
h_P(v) = \langle v, x \rangle \quad \text{for } v \in N_F(P) \text{ and } x \in F
$$
Thus the support function of a polyhedron is linear on every normal cone.
\end{lem}
The formula in Lemma \ref{lem:SuppFunc} implies that the linearity domains of $h_P$ are the normal cones at the vertices of $P$; that is, the nonlinearity locus of $h_P$ is the union of the normal cones at the edges of $P$.

\begin{dfn}
\label{dfn:PLExt}
Let $V = (v_1, \ldots, v_n)$ be a vector configuration in $\R^d$, and $\Delta$ be a fan with $\Delta^{(1)} = V$. For every $h \in \R^n$ denote by
$$
\widetilde{h_\Delta} \colon \supp(\Delta) \to \R^d
$$
the piecewise linear function obtained by extending the map $v_i \mapsto h_i$ linearly to each cone of $\Delta$.
\end{dfn}
Note that if $\Delta$ is not simplicial, then $\widetilde{h_\Delta}$ is defined not for all $h$.

Let $P = P(V,h)$ be a $d$-dimensional convex polyhedron such that $\cN(P)^{(1)} = V$. Since $h_P(v_i) = h_i$, we have
$$
h_P = \widetilde{h_{\cN(P)}}
$$

\begin{cor}
\label{cor:PolFanChar}
Let $\Delta$ be a pointed fan with a convex support. Then $h \in \wT(\Delta)$ if and only if the piecewise linear function $\widetilde{h_\Delta}$ is defined, convex, and has $\Delta^{(d-1)}$ as its non-linearity locus.

Consequently, a pointed fan with convex support is polytopal if and only if there exists an $h$ with the above properties.
\end{cor}

\begin{rem}
\label{rem:RegSubdiv}
Corollary \ref{cor:PolFanChar} says that a polytopal fan is the same as a \emph{regular subdivision} of a vector configuration, see \cite[Section 9.5.1]{LRS10}. In a special case, when $\supp(\Delta)$ is contained in an open half-space, the rays of $\Delta$ can be replaced by their intersection points with an affine hyperplane $A$. The convex hull of these points is subdivided by the cones of $\Delta$ into convex polytopes that are the linearity domains of a convex piecewise affine function on $A$. Such a subdivision is called a \emph{regular subdivision} of a point configuration, see \cite[Definition 5.3]{Zie95} and \cite{LRS10}.
\end{rem}

\subsubsection{Fans refinement and the Minkowski sum}
\begin{dfn}
We say that a fan $\Delta'$ \emph{refines} a fan $\Delta$ (or $\Delta$ \emph{coarsens} $\Delta'$) and write $\Delta' \preccurlyeq \Delta$ if every cone of $\Delta$ is a union of cones of $\Delta'$.
\end{dfn}

\begin{lem}
\label{lem:FanRefine}
Every non-simplicial polytopal fan $\Delta$ can be refined to a simplicial polytopal fan $\Delta'$. Besides, for every such $\Delta'$ we have $T(\Delta) \subset \cl T(\Delta')$.
\end{lem}
Indeed, let $h \in \wT(\Delta) \subset \R^n$. Choose a generic $y \in \R^n$ and consider $h' = h + ty$ with $t$ sufficiently small. The extension $\widetilde{h'_\Delta}$ is not defined, but it becomes so if every $\sigma \in \Delta$ is appropirately subdivided (by a convex hull construction). This yields a simplicial fan $\Delta'$ such that $\widetilde{h'_{\Delta'}}$ is convex, (the convexity across $\Delta^{(d-1)}$ is preserved for small $t$). As for the second statement of the lemma, by choosing $h \in \wT(\Delta)$ and $y \in \wT(\Delta')$ we see in a similar way that $h + ty \in \wT(\Delta')$ for all $t > 0$.

The geometric picture behind this argument is that translating facets of a non-simple polyhedron generically and by small amounts makes the polyhedron simple without destroying any of its faces.

\begin{dfn}
The Minkowski sum of two sets $K,L \subset \R^d$ is defined as
$$
K+L := \{x+y \mid x \in K, y \in L\}
$$
\end{dfn}

The additive structure on the space of support vectors is related to the Minkowski addition. However, one should be careful when the summands are not normally equivalent.

\begin{dfn}
For two fans $\Delta_1$ and $\Delta_2$ with the same support denote by $\Delta_1 \wedge \Delta_2$ their coarsest common refinement.
Explicitely,
$$
\Delta_1 \wedge \Delta_2 = \{\sigma_1 \cap \sigma_2 \mid \sigma_i \in \Delta_i, i = 1,2\}
$$
\end{dfn}

\begin{lem}
\label{lem:FanMink}
Let $P$ and $Q$ be two convex polyhedra such that $\supp(\cN(P)) = \supp(\cN(Q))$. Then we have:
$$
\cN(P+Q) = \cN(P) \wedge \cN(Q)
$$
\end{lem}

Fix $V$ and denote $P(h) := P(V,h)$.
Lemma \ref{lem:FanMink} implies that in general
$$
P(h+h') \ne P(h) + P(h'),
$$
because $P(h) + P(h')$ can have more facets than $P(h+h')$. See Figure~\ref{fig:SuppAdd} showing fragments of 3-dimensional polyhedra. (However, $P(h+h') \supset P(h) + P(h')$ holds always.)

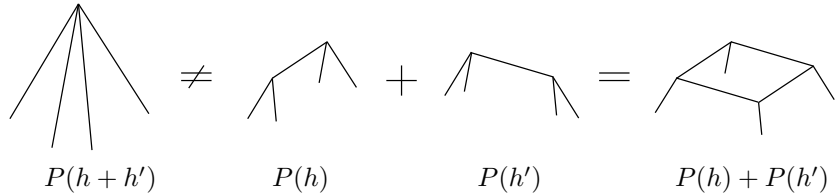
\begin{figure}[ht]
\centering
\begin{picture}(0,0)%
\includegraphics{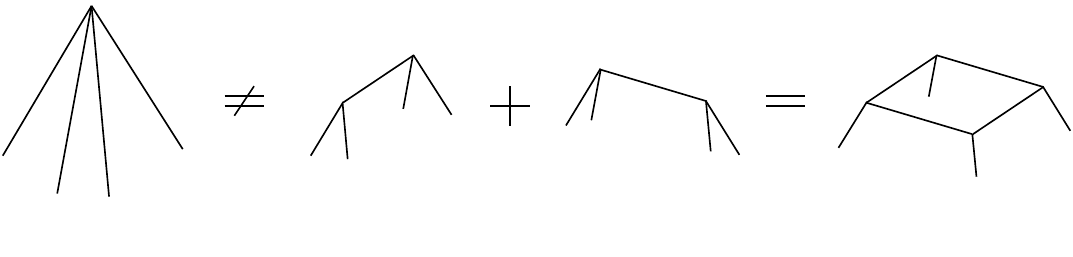}%
\end{picture}%
\setlength{\unitlength}{4144sp}%
\begingroup\makeatletter\ifx\SetFigFont\undefined%
\gdef\SetFigFont#1#2#3#4#5{%
  \reset@font\fontsize{#1}{#2pt}%
  \fontfamily{#3}\fontseries{#4}\fontshape{#5}%
  \selectfont}%
\fi\endgroup%
\begin{picture}(4905,1159)(235,-1651)
\put(451,-1591){\makebox(0,0)[lb]{\smash{{\SetFigFont{10}{12.0}{\rmdefault}{\mddefault}{\updefault}{\color[rgb]{0,0,0}$P(h+h')$}%
}}}}
\put(1801,-1591){\makebox(0,0)[lb]{\smash{{\SetFigFont{10}{12.0}{\rmdefault}{\mddefault}{\updefault}{\color[rgb]{0,0,0}$P(h)$}%
}}}}
\put(3016,-1591){\makebox(0,0)[lb]{\smash{{\SetFigFont{10}{12.0}{\rmdefault}{\mddefault}{\updefault}{\color[rgb]{0,0,0}$P(h')$}%
}}}}
\put(4186,-1591){\makebox(0,0)[lb]{\smash{{\SetFigFont{10}{12.0}{\rmdefault}{\mddefault}{\updefault}{\color[rgb]{0,0,0}$P(h)+P(h')$}%
}}}}
\end{picture}%
\caption{Minkowski addition does not always correspond to the addition of support numbers.}
\label{fig:SuppAdd}
\end{figure}

Instead, we have the following.
\begin{lem}
\label{lem:MinkSum}
The support function of the Minkowski sum of two convex bodies is the sum of their support functions:
$$
h_{K+L} = h_K + h_L
$$
\end{lem}

\begin{cor}
\label{cor:MinkLin}
If $\cN(P(h)) \preccurlyeq \cN(P(h'))$, then $P(h) + P(h') = P(h + h')$.
\end{cor}
We also clearly have $h_{\lambda K} = \lambda h_K$ and $P(\lambda h) = \lambda P(h)$ for $\lambda > 0$. This is false for $\lambda < 0$. Taking differences of the support functions embeds the space of convex polytopes in the vector space of the so called virtual polytopes, see \cite{McM89}, \cite{PK92}.

A subset of a euclidean space is called \emph{relatively open}, if it is open as a subset of its affine hull. For example, an open half-line in $\R^2$ is relatively open.
\begin{lem}
\label{lem:LinTypeCone}
For every polytopal fan $\Delta$, the type cone $T(\Delta)$ is convex and relatively open.
\end{lem}

The convexity follows from
$$
P(\lambda h + (1-\lambda) h') = \lambda P(h) + (1-\lambda) P(h')
$$
for $\cN(P(h)) \simeq \cN(P(h'))$ and $\lambda \in [0,1]$,
and the relative openness from
$$
h, h' \in \wT(\Delta) \Rightarrow h + \epsilon h' \in \wT(\Delta),
$$
for $\epsilon$ sufficiently small, whether positive or negative.

\subsubsection{Linear inequalities describing a type cone}
\label{sec:ContrEdge}
Let $P \subset \R^d$ be a $d$-polytope with the normal fan $\Delta$. Then $P = P(h)$ for $h \in \wT(\Delta)$ (we omit $V = \Delta^{(1)}$ from the notation $P(V,h)$). For every cone $\sigma \in \Delta$ let $F_\sigma(h)$ be the face of $P(h)$ with the normal cone $\sigma$.

If $\sigma \in \Delta^{(d-1)}$, then $F_\sigma(h)$ is an edge; denote its length by
$$
\ell^\Delta_\sigma(h) := \vol_1(F_\sigma(h))
$$
Of course, all functions $\ell^\Delta_\sigma \colon \wT(\Delta) \to \R$ descend through $\pi \colon \R^n \to \R^n/\im V$ to functions on $T(\Delta)$. However, the proofs of the lemmas in this section are better written in terms of $h$ than $\pi(h)$. Due to $T(\Delta) = \pi(\wT(\Delta))$ the statements of the lemmas are easily translated in terms of $T(\Delta)$.

\begin{lem}
\label{lem:EllLin}
\begin{enumerate}
\item
For every $\sigma \in \Delta^{(d-1)}$ and $h \in \wT(\Delta)$, we have
$$
\ell^\Delta_\sigma(h) = \|\grad (h_P|_{\rho_1}) - \grad (h_P|_{\rho_2})\|,
$$
where $\rho_1, \rho_2 \in \Delta^{(d)}$ are the two full-dimensional cones having $\sigma$ as their facet, and $h_P$ is the support function of the polytope $P = P(h)$.
\item
$\ell^\Delta_\sigma \colon \wT(\Delta) \to \R$ extends to a linear function on $\span(\wT(\Delta))$.
\end{enumerate}
\end{lem}
\begin{proof}
Denote by $e_\sigma$ the unit vector orthogonal to $\aff(\sigma)$ and directed from $\rho_{i_2}$ into~$\rho_{i_1}$. Then we have
\begin{equation}
\label{eqn:EllSigma}
p_1 - p_2 = \ell^\Delta_\sigma e_\sigma,
\end{equation}
where $p_1$ and $p_2$ are vertices of~$P$ with normal cones $\rho_1$ and $\rho_2$, respectively. By Lemma \ref{lem:SuppFunc} we have
$$
h_P(v) = \langle v, p_\alpha \rangle \ \forall v \in \rho_\alpha, \ \alpha = 1,2
$$
and therefore $p_\alpha = \grad(h_P|_{\rho_\alpha})$. Substituting this into \eqref{eqn:EllSigma} proves the first part of the lemma.

For the second part, choose among the unit length generators of $\Delta^{(1)}$ linearly independent vectors
$$
v_{j_1}, \ldots, v_{j_{d-1}} \in \sigma
$$
Choose also $v_{i_1} \in \rho_1 \setminus \sigma$ and $v_{i_2} \in \rho_2 \setminus \sigma$.
Then the function $h_P|_{\rho_\alpha}$, $\alpha = 1,2$, is uniquely determined by its values $h_{j_1}, \ldots, h_{j_{d-1}}, h_{i_\alpha}$.

As every set of $d+1$ vectors in $\R^d$ is linearly dependent, we have
\begin{equation}
\label{eqn:Circuit}
\lambda_1 v_{i_1} + \lambda_2 v_{i_1} + \mu_1 v_{j_1} + \cdots + \mu_{d-1} v_{j_{d-1}} = 0
\end{equation}
for some $\lambda_\alpha, \mu_\beta \in \R$. Besides, $\lambda_1, \lambda_2 \ne 0$. By definition of the support function we have
$$
h_{i_1} = \langle v_{i_1}, p_1 \rangle, \quad h_{j_\beta} = \langle v_{j_\beta}, p_1 \rangle = \langle v_{j_\beta}, p_2 \rangle, \quad h_{i_2} = \langle v_{i_2}, p_2 \rangle
$$
It follows that
\begin{multline*}
\lambda_1 h_{i_1} + \lambda_2 h_{i_1} + \mu_1 h_{j_1} + \cdots + \mu_{d-1} h_{j_{d-1}}\\
= \lambda_1 \langle v_{i_1}, p_1 \rangle + \lambda_2 \langle v_{i_2}, p_2 \rangle + \sum_{\beta=1}^{d-1} \mu_{j_\beta} \langle v_{i_2}, p_2 \rangle\\
= \lambda_1 \langle v_{i_1}, p_1 - p_2 \rangle
\end{multline*}

By substituting \eqref{eqn:EllSigma} we obtain
$$
\ell^\Delta_\sigma = \frac1{\lambda_1 \langle v_{i_1}, e_\sigma \rangle} (\lambda_1 h_{i_1} + \lambda_2 h_{i_1} + \mu_1 h_{j_1} + \cdots + \mu_{d-1} h_{j_{d-1}}),
$$
which shows that $\ell^\Delta_\sigma \colon \wT(\Delta) \to \R$ is a restriction of a linear function.
\end{proof}

\begin{rem}
The function $\ell^\Delta_\sigma$ can be computed as follows. Choose a flag of faces
$$
P \supset F_1 \supset \ldots \supset F_{d-1} = F_\sigma,
$$
where $\dim F_i = d-i$.
Project $0$ orthogonally to $\aff(F_1)$. Then the support numbers of $F_1$, with respect to the projection of $0$, can be written as linear functions of the support numbers of $P$ (the coefficients are trigonometric functions of the dihedral angles between $F_1$ and adjacent facets). Repeat this by expressing the support numbers of $F_2$ as linear functions of the support numbers of $F_1$, and so on. At the end we obtain the support numbers of $F_{d-1}$ (there are two of them, as $F_{d-1}$ is a segment) as linear functions of the support numbers of $P$. But the length of $F_{d-1}$ is just the sum of these two numbers.

For $d=3$, this computation is done in Appendix \ref{sec:App2}.
\end{rem}

Note that if $\Delta$ is not simplicial, then we might have several different choices for $v_{i_\alpha}$, $v_{j_\beta}$ in the proof of Lemma \ref{lem:EllLin}, and hence several formulas for $\ell^\Delta_\sigma$ in terms of $h_i$. This is due to the fact that, for a non-simplicial $\Delta$, the set $\wT(\Delta)$ does not span the space $\R^n$ of support vectors. Thus different linear functions on $\R^n$ have the same restrictions to $\span(\wT(\Delta))$.

\begin{lem}
\label{lem:AffT}
Let $\Delta$ be a complete pointed polytopal fan with $|\Delta^{(1)}| = n$.
\begin{enumerate}
\item
If $\Delta$ is simplicial, then we have
$$
\span(\wT(\Delta)) = \R^{\Delta^{(1)}},
$$
\item
If $\Delta$ is not simplicial, then
\begin{multline*}
\span(\wT(\Delta)) = \{h \in \R^{\Delta^{(1)}} \mid \ell^{\Delta'}_{\sigma'}(h) = 0 \text{ for all } \sigma' \in \Delta'^{(d-1)} \text{ such that }\\
\relint \sigma' \subset \relint \rho \text{ for some } \rho \in \Delta^{(d)}\},
\end{multline*}
where $\Delta'$ is any simplicial refinement of $\Delta$.
\end{enumerate}
\end{lem}
\begin{proof}
If $h \in \wT(\Delta)$ for a simplicial $\Delta$, then for every $y \in \R^n$ the piecewise linear extension of $h + ty$ with respect to $\Delta$ is convex and has $\Delta^{(d-1)}$ as its non-linearity locus, provided that $t$ is sufficiently small. This proves the first part of the lemma.

If $\Delta$ is not simplicial, then let us first show
\begin{equation}
\label{eqn:TKernel}
\wT(\Delta) \subset \{h \in \R^n \mid \ell^{\Delta'}_{\sigma'}(h) = 0\},
\end{equation}
with $\Delta'$ and $\sigma'$ as in the statement of the lemma. Let $h \in \wT(\Delta)$. Then the support function $h_P = \widetilde{h_\Delta}$ is linear on $\rho$, since $\rho$ is a normal cone of $P(h)$. On the other hand, let $\rho'_1, \rho'_2 \in \Delta'$ be the $d$-cones adjacent to $\sigma'$. As $\rho'_1, \rho'_2 \subset \rho$, we have $\grad(h_P|_{\rho'_1}) = \grad(h_P|_{\rho'_2})$. As $h \in \cl \wT(\Delta')$ by Lemma \ref{lem:FanRefine}, the value $\ell^{\Delta'}_{\sigma'}(h)$ for $h \in \wT(\Delta)$ is given by the same formula as for $h \in T(\Delta')$. Thus by the first part of Lemma \ref{lem:EllLin} we have $\ell^{\Delta'}_{\sigma'}(h) = 0$.

In the other direction, let $h \in \wT(\Delta)$ and $y \in \ker \ell^{\Delta'}_{\sigma'}$ for all $\sigma'$ as in the lemma. We claim that then $h + ty \in T(\Delta)$ for $t$ sufficiently small. This would mean that
$$
\{h \in \R^n \mid \ell^{\Delta'}_{\sigma'}(h) = 0\} \subset \span(T(\Delta)),
$$
and thus, together with \eqref{eqn:TKernel}, imply the second part of the lemma. So, take any $\rho \in \Delta^{(d)}$ and consider the piecewise linear extension of $y$ with respect to the subdivision of $\rho$ induced by $\Delta'$. This extension is in fact linear, since it is linear across all $(d-1)$-cones of $\Delta'$ whose relative interiors lie in $\rho$. This shows that the extension $\widetilde{y_\Delta}$ exists. For $t$ small, the piecewise linear extension of $h + ty$ with respect to $\Delta$ is convex and has the same non-linearity locus as $\widetilde{h_\Delta}$. Hence $h + ty \in \wT(\Delta)$, and we are done.
\end{proof}

\begin{rem}
For $d = 3$, the condition on $\sigma'$ in the second part of Lemma~\ref{lem:AffT} can be replaced by $\sigma' \notin \Delta$. This is not so for $d > 3$: if the cone $\sigma \in \Delta^{(d-1)}$ is not simplicial, then it gets subdivided into simplicial cones $\sigma'_i \in \Delta'^{(d-1)}$. We thus have $\sigma'_i \notin \Delta$. However, the corresponding edge lengths are not zero. In fact, the edges $F_{\sigma'_i}(h)$, which are different for $h \in \wT(\Delta')$, become one edge $F_\sigma(h)$ for $h \in \wT(\Delta)$. This means that for $h \in \wT(\Delta)$ we have
$$
\ell^{\Delta'}_{\sigma'_i}(h) = \ell^{\Delta'}_{\sigma'_j}(h)
$$
for any $\sigma'_i, \sigma'_j \subset \sigma$. Thus these linear equations hold also on $\span(\wT(\Delta))$. However, they don't enter the description of $\span(\wT(\Delta))$ given in the second part of Lemma \ref{lem:AffT}, as they follow from the other ones (this simply follows from the statement of the lemma).
\end{rem}

\begin{lem}
\label{lem:TypeConeIneq}
Let $\Delta$ be a complete pointed fan in $\R^d$.

If $\Delta$ is simplicial, then $\wT(\Delta)$ is the solution set of the following system of linear inequalities:
$$
\wT(\Delta) = \{h \in \R^n \mid \ell^\Delta_\sigma(h) > 0 \text{ for all } \sigma \in \Delta^{(d-1)}\}
$$

If $\Delta$ is not simplicial, then $\wT(\Delta)$ is the solution set of the following system of linear equations and inequalities:
\begin{multline*}
\wT(\Delta) = \{h \in \R^n \mid \ell^\Delta_\sigma(h) > 0 \text{ for all } \sigma \in \Delta^{(d-1)}, \text{ and}\\
\ell^{\Delta'}_{\sigma'}(h) = 0 \text{ for all } \sigma' \in \Delta'^{(d-1)} \text{ such that }\\
\relint \sigma' \subset \relint \rho \text{ for some } \rho \in \Delta^{(d)}\},
\end{multline*}
where $\Delta'$ is any simplicial fan that refines $\Delta$.

In particular, the closure $\cl T(\Delta)$ of the type cone is a convex polyhedral cone in $\R^n/\im V$.
\end{lem}
Note that the function $\ell^\Delta_\sigma(h)$ is well-defined only on $\span(\wT(\Delta))$. Due to Lemma \ref{lem:AffT}, the above description of $\wT(\Delta)$ in the non-simplicial case is unambiguous.
\begin{proof}
First, let $\Delta$ be simplicial. Then the inequality $\ell_\sigma^\Delta(h) > 0$ just means that the piecewise linear function $\widetilde{h_\Delta}$ is strictly convex across the $(d-1)$-cone $\sigma$. It follows that these inequalities are necessary and sufficient for $h$ to belong to $\wT(\Delta)$.

Second, let $\Delta$ be non-simplicial. Then, similarly to the proof of Lemma \ref{lem:AffT}, the equations $\ell^{\Delta'}_{\sigma'}(h) = 0$ ensure that the linear extension $\widetilde{h_\Delta}$ exists. And, as in the simplicial case, $\ell^\Delta_\sigma(h) > 0$ is equivalent to the strict convexity of this function across $\sigma$.
The lemma is proved.
\end{proof}

\subsection{Secondary fans}
Everywhere in this section $V = (v_1, \ldots, v_n)$ is a linearly spanning vector configuration in $\R^d$. As usual, $V$ will also stand for the $n \times d$ matrix with rows $v_i$.

\subsubsection{The compatibility and the irredundancy domains}
As in Section \ref{sec:SuppNumb}, for an $h \in \R^n$ denote by $P(h)$ the solution set of the system $Vx \le h$. It can happen that the set $P(h)$ is empty; and even when it is non-empty, it can happen that some of the inequalities $\langle v_i, x \rangle \le h_i$ can be removed without changing $P(h)$. This leads us to consider the following two subsets of $\R^n$.
$$
\wco(V) := \{h \in \R^n \mid Vx \le h \text{ is compatible}\},
$$
that is all those $h$ for which $P(h) \ne \emptyset$.
$$
\wir(V) := \{h \in \R^n \mid Vx \le h \text{ is compatible and irredundant}\},
$$
where \emph{irredundant} means that removing any inequality from the system makes the solution set bigger.

As in the case with a type cone and a lifted type cone, equation \eqref{eqn:PTransl} implies that both $\wco(V)$ and $\wir(V)$ are invariant under translation by $Vt$ for any $t \in \R^d$. Therefore it suffices to study their quotients under the map
\begin{equation}
\label{eqn:Pi}
\pi \colon \R^n \to \R^n/\im V
\end{equation}

\begin{dfn}
\label{dfn:CoIr}
Let $V \in \R^{n \times d}$ be of rank $d$. The set
$$
\co(V) := \wco(V)/\im V
$$
is called the \emph{compatibility domain}, and the set
$$
\ir(V) := \wir(V)/\im V
$$
the \emph{irredundancy domain} for $V$, respectively. Further, we denote by $\clir(V)$ the closure of $\ir(V)$.
\end{dfn}

For every polytopal fan $\Delta$ with $\Delta^{(1)} = V$ we have $T(\Delta) \subset \ir(V)$, and we will later see that type cones decompose the interior of $\ir(V)$. Moreover, this can be extended to a subdivision of $\co(V)$, the so called \emph{secondary fan}.

Most of the material presented here can be found in \cite{McM73, McM79, BFS90, BGS93, LRS10}. The notation $\ir$ was introduced in \cite{McM73}, meaning ``\emph{i}nner \emph{r}egion'', but in some later works of the same author the terminology was changed to the ``\emph{ir}redundancy domain''.

\subsubsection{Gale duality, linear dependencies and evaluations}
For the moment, we don't require all $v_i$ to have norm $1$ (in fact, we even allow $v_i = 0$ or $v_i = v_j$ for $i \ne j$).

\begin{dfn}
\label{dfn:Gale}
Choose a linear isomorphism $\R^n/\im V \cong \R^{n-d}$ and denote by $\bar V^\top$ the matrix representation of the projection \eqref{eqn:Pi}:
$$
\R^d \stackrel{V}{\longrightarrow} \R^n \stackrel{\bar V^\top}{\longrightarrow} \R^{n-d}
$$
Then the \emph{Gale transform} or \emph{Gale diagram} of $(v_1, \ldots, v_n)$ is the collection of $n$ vectors $(\bar v_1, \ldots, \bar v_n)$ that form the rows of the $n \times (n-d)$-matrix $\bar V$.
\end{dfn}

\begin{lem}
The vector configuration $(\bar v_1, \ldots, \bar v_n)$ is well-defined up to a linear transformation of $\R^{n-d}$. Besides, if $(\bar v_1, \ldots, \bar v_n)$ is a Gale diagram for $(v_1, \ldots, v_n)$, then also $(v_1, \ldots, v_n)$ is a Gale diagram of $(\bar v_1, \ldots, \bar v_n)$.
\end{lem}
Indeed, the uniqueness up to a linear transformation follows from the freedom in the choice of the isomorphism $\R^n/\im V \cong \R^{n-d}$. The involutivity of the Gale transform follows by transposing the short exact sequence in Definition \ref{dfn:Gale}. It allows also to call the vector configuration $(\bar v_i)$ \emph{Gale dual} to $(v_i)$.

\begin{rem}
If $V$ contains the vectors of the standard basis of $\R^d$, then its Gale dual is very easy to compute:
$$
V =
\begin{pmatrix}
E_d\\
A
\end{pmatrix}
\Rightarrow
\bar V =
\begin{pmatrix}
-A^\top\\
E_{n-d}
\end{pmatrix},
$$
where $E_k$ denotes the $k \times k$-unit matrix.
\end{rem}

As an immediate consequence of Definition \ref{dfn:Gale} we have
\begin{equation}
\label{eqn:ValDep}
\ker V^\top = \im \bar V
\end{equation}
The elements of $\ker V^\top$ are the \emph{linear dependencies} between the vectors $(v_1, \ldots, v_n)$:
$$
\lambda \in \ker V^\top \Leftrightarrow \sum_{i=1}^n \lambda_i v_i = 0,
$$
while the elements of $\im \bar V$ are the \emph{evaluations of linear functionals} on $(\bar v_1, \ldots, \bar v_n)$:
$$
\lambda \in \im \bar V \Leftrightarrow \lambda_i = \langle \mu, \bar v_i \rangle \text{ for some }\mu \in \R^{n-d}
$$
Therefore \eqref{eqn:ValDep} can be phrased as ``dependencies of $V$ equal evaluations on~$\bar V$''. Of course, the same holds with $V$ and $\bar V$ exchanged.


In the next lemma, the rank of a vector configuration means the dimension of its linear span, so that a rank $d$ configuration is the same as a linearly spanning configuration. We also use the notation $[n] := \{1, \ldots, n\}$, and for any subset $I \subset [n]$ we denote by $V_I$ the vector configuration $(v_i \mid i \in I)$.

\begin{lem}
\label{lem:GaleProp}
Let $V = (v_1, \ldots, v_n)$ be a rank $d$ configuration of $n$ vectors in $\R^d$, and let $\bar V$ be its Gale dual.
\begin{enumerate}
\item
The vector configuration $V$ is positively spanning if and only if there is $\mu \in \R^{n-d}$ such that $\langle \bar \mu, v_i \rangle > 0$ for all $i \in [n]$.
\item
Subconfiguration $V_I$ is linearly independent if and only if the subconfiguration $\bar V_{[n] \setminus I}$ has full rank. In particular, $V_I$ is a basis of $\R^d$ if and only if $\bar V_{[n] \setminus I}$ is a basis of $\R^{n-d}$.
\item
Configuration $\bar V$ contains a zero vector $\bar v_i = 0$ if and only if the subconfiguration $V \setminus \{v_i\}$ has rank $d-1$. In particular, if $V$ is positively spanning, then $\bar v_i \ne 0$ for all $i$.
\item
Configuration $\bar V$ contains two collinear vectors $\bar v_i = c \bar v_j$ (one or both of which may be zero) if and only if the subconfiguration $V \setminus \{v_i, v_j\}$ has rank at most $d-1$.
\end{enumerate}
\end{lem}
\begin{proof}
A configuration $V$ is positively spanning if and only if it has rank $d$ and is positively dependent:
\begin{equation}
\label{eqn:PosDep}
\rank V = d, \quad \sum_{i=1}^n \lambda_i v_i = 0 \text{ so that } \lambda_i > 0 \text{ for all }i
\end{equation}
Indeed, if \eqref{eqn:PosDep} holds, then every linear combination of $v_i$ can be made positive by adding a positive multiple of $\sum_i \lambda_i v_i$, so that $\rank V = d$ implies that $V$ is positively spanning. For the inverse implication, express $-v_1$ as a positive linear combination of $v_i$. On the other hand, by the dependence-evaluation duality \eqref{eqn:ValDep} a positive dependence of $V$ corresponds to a linear functional positively evaluating on $\bar V$. This proves the first part of the lemma.

The subconfiguration $V_I$ is linearly dependent if and only if there is $\lambda \in \R^n$, $\lambda \ne 0$ such that
$$
\sum_{i=1}^n \lambda_i v_i = 0 \text{ and } \lambda_i = 0 \text{ for } i \notin I
$$
By the dependence-evaluation duality this is equivalent to the existence of a non-zero linear functional on $\R^{n-d}$ that vanishes on all $\bar v_i$ for $i \ne I$, which is equivalent to $\rank \bar V_{[n]\setminus I} < n-d$. This proves the second part of the lemma.

The third part follows from the second: $\bar v_i = 0$ is equivalent to the set $\{\bar v_i\}$ being linearly dependent, and if $\rank V = d$, then $\rank (V \setminus \{v_i\}) \ge d-1$. Also, if $\rank (V \setminus \{v_i\}) = d-1$, then $-v_i \notin \pos(V)$.

The fourth part is a direct consequence of the second.
\end{proof}

Let $V \in \R^{n \times d}$ be a vector configuration such that its Gale diagram $(\bar v_1, \ldots, \bar v_n)$ contains no zero vectors (a necessary and sufficient condition for this is given in part 3 of Lemma~\ref{lem:GaleProp}). The \emph{affine Gale diagram} is constructed by choosing an affine hyperplane $A \subset \R^{n-d}$ non-parallel to each of $\bar v_i$ and scaling each $\bar v_i$ so that
$$
\bar p_i = \alpha_i v_i \in A
$$
Besides, a point $\bar p_i$ is colored black if $\alpha_i > 0$, and white if $\alpha_i < 0$.

An affine Gale diagram determines the vector configuration $V$ up to independent positive scalings of $v_i$. Indeed, the point $\bar p_i$ together with its color determines $\bar v_i$ up to a positive scaling; and positive scaling $\bar v_i \mapsto \beta_i \bar v_i$ of $\bar V$ corresponds to positive scaling $v_i \mapsto \beta_i^{-1} v_i$ of $V$.

\begin{rem}
\label{rem:AffGale}
By part 1 of Lemma \ref{lem:GaleProp}, a positively spanning vector configuration has an affine Gale diagram consisting of black points only.

By the dependencies-evaluations duality \eqref{eqn:ValDep}, if $\sum_{i=1}^n v_i = 0$, then the vectors $\bar v_i$ lie already in an affine hyperplane.
\end{rem}

\begin{exl}
Let $V = (e_1, e_2, e_3, -e_1, -e_2, -e_3)$ be a configuration of six vectors in $\R^3$. Its Gale dual is $\bar V = (e_1, e_2, e_3, e_1, e_2, e_3)$. The corresponding affine Gale diagram is shown on Figure \ref{fig:GaleExl}.

\begin{figure}[ht]
\centering
\begin{picture}(0,0)%
\includegraphics{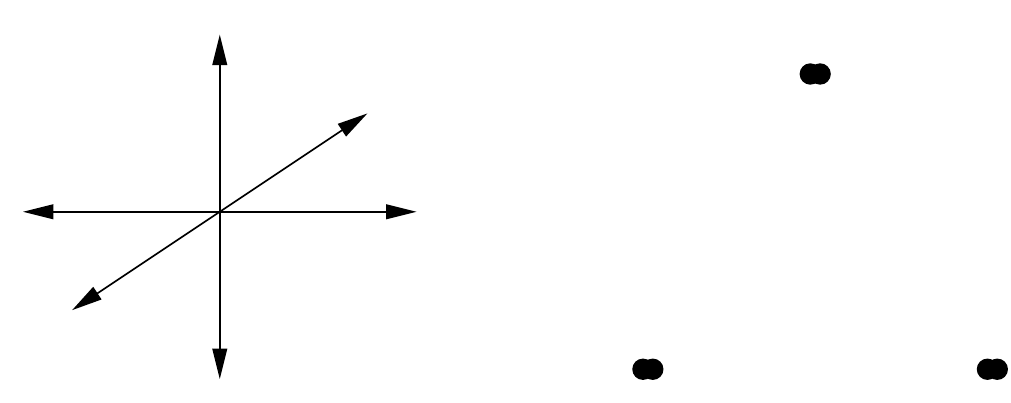}%
\end{picture}%
\setlength{\unitlength}{4144sp}%
\begingroup\makeatletter\ifx\SetFigFont\undefined%
\gdef\SetFigFont#1#2#3#4#5{%
  \reset@font\fontsize{#1}{#2pt}%
  \fontfamily{#3}\fontseries{#4}\fontshape{#5}%
  \selectfont}%
\fi\endgroup%
\begin{picture}(4613,1831)(-104,-935)
\put(1756, 29){\makebox(0,0)[lb]{\smash{{\SetFigFont{10}{12.0}{\rmdefault}{\mddefault}{\updefault}{\color[rgb]{0,0,0}$v_1$}%
}}}}
\put(946,749){\makebox(0,0)[lb]{\smash{{\SetFigFont{10}{12.0}{\rmdefault}{\mddefault}{\updefault}{\color[rgb]{0,0,0}$v_2$}%
}}}}
\put(181,-646){\makebox(0,0)[lb]{\smash{{\SetFigFont{10}{12.0}{\rmdefault}{\mddefault}{\updefault}{\color[rgb]{0,0,0}$v_3$}%
}}}}
\put(-89, 29){\makebox(0,0)[lb]{\smash{{\SetFigFont{10}{12.0}{\rmdefault}{\mddefault}{\updefault}{\color[rgb]{0,0,0}$v_4$}%
}}}}
\put(946,-871){\makebox(0,0)[lb]{\smash{{\SetFigFont{10}{12.0}{\rmdefault}{\mddefault}{\updefault}{\color[rgb]{0,0,0}$v_5$}%
}}}}
\put(1441,434){\makebox(0,0)[lb]{\smash{{\SetFigFont{10}{12.0}{\rmdefault}{\mddefault}{\updefault}{\color[rgb]{0,0,0}$v_6$}%
}}}}
\put(4186,-646){\makebox(0,0)[lb]{\smash{{\SetFigFont{10}{12.0}{\rmdefault}{\mddefault}{\updefault}{\color[rgb]{0,0,0}$\bar v_1 = \bar v_4$}%
}}}}
\put(3376,659){\makebox(0,0)[lb]{\smash{{\SetFigFont{10}{12.0}{\rmdefault}{\mddefault}{\updefault}{\color[rgb]{0,0,0}$\bar v_2 = \bar v_5$}%
}}}}
\put(2611,-646){\makebox(0,0)[lb]{\smash{{\SetFigFont{10}{12.0}{\rmdefault}{\mddefault}{\updefault}{\color[rgb]{0,0,0}$\bar v_3 = \bar v_6$}%
}}}}
\end{picture}%
\caption{A vector configuration and its affine Gale diagram.}
\label{fig:GaleExl}
\end{figure}
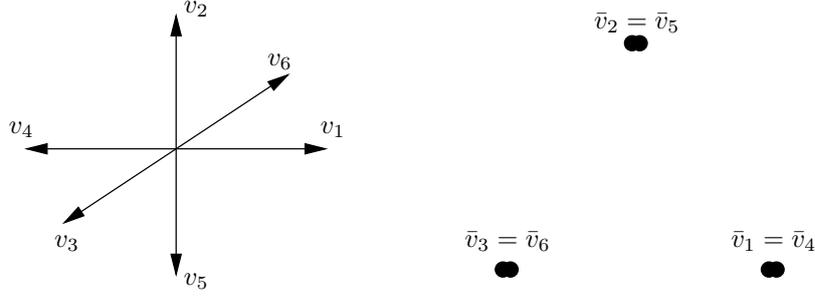
\end{exl}

For more details on Gale duality, see \cite{Ewald96,LRS10,Zie95}.

\subsubsection{Positive circuits and $\co(V)$}
\label{sec:PosCirc}
\begin{dfn}
A subset $C \subset [n]$ is called a \emph{circuit} of a vector configuration $V = (v_1, \ldots, v_n)$, if $V_C = (v_i \mid i \in C)$ is an inclusion-minimal linearly dependent subset of $V$.
\end{dfn}

An inclusion-minimal linearly dependent set is one that is linearly dependent, while each of its proper subsets is not.

The minimality condition implies that a linear dependence $\lambda^C \subset \R^C$ between vectors of a circuit $C$ is unique up to a scaling, and that all coefficients $\lambda^C_i$ are nonzero. In particular, it makes sense to speak about the signs of the coefficients (up to a simultaneous change of all signs).

A circuit $C$ is called \emph{positive}, if all coefficients of the corresponding linear dependence can be chosen positive:
$$
\sum_{i \in C} \lambda^C_i v_i = 0, \quad \text{where }\lambda^C_i > 0 \ \forall i \in C
$$
By the dependence-evaluation duality \eqref{eqn:ValDep}, to a positive circuit $C$ there corresponds a linear functional $\mu^C \in \R^{n-d}$ such that
$$
\langle \bar \mu^C, \bar v_i \rangle
\begin{cases}
> 0, &\text{if } i \in C,\\
= 0, &\text{if } i \notin C
\end{cases}
$$
Namely, we have $\lambda^C = \bar V \mu^C$.
Since $\bar V$ has full rank, $\mu^C$ is well-defined up to a scaling.
Besides, the minimality of $C$ implies that $\bar V_{[n] \setminus C}$ is a \emph{maximal non-spanning subconfiguration} of $\bar V$. Such subconfigurations are called \emph{cocircuits}. We thus have
$$
C \text{ is a circuit for } V \Leftrightarrow [n] \setminus C \text{ is a cocircuit for }\bar V
$$
(See also the second part of Lemma \ref{lem:GaleProp}.)

\begin{thm}
\label{thm:CoV}
Let $(v_1, \ldots, v_n)$ be a linearly spanning vector configuration in $\R^d$. Then the compatibility domain $\co(V)$ is a convex polyhedral cone in $\R^{n-d}$. The compatibility domain and its lift $\wco(V)$ to $\R^n$ can be described in the following ways.
\begin{subequations}
\begin{equation}
\label{eqn:CoV1}
\wco(V) = \R_+^n + \im V, \quad \text{where } \R_+^n = \{(h_1, \ldots, h_n) \mid h_i \ge 0 \ \forall i\}
\end{equation}
\begin{equation}
\label{eqn:CoV2}
\co(V) = \pos\{\bar v_1, \ldots, \bar v_n\}, \quad \text{where } \bar V \text{ is the Gale dual of }V
\end{equation}
\begin{equation}
\label{eqn:CoV3}
\wco(V) = \left\{h \in \R^n \mid \langle \lambda^C, h \rangle \ge 0 \text{ for all positive circuits }C\right\}
\end{equation}
\begin{equation}
\label{eqn:CoV4}
\co(V) = \left\{y \in \R^{n-d} \mid \langle \mu^C, y \rangle \ge 0 \text{ for all positive circuits }C \right\}
\end{equation}
\end{subequations}
Here $\lambda^C \in \R^C_+$ are the coefficients of the linear dependence, and $\mu^C \in \R^{n-d}$ is the linear functional associated with the circuit $C$.

Moreover, $V$ is positively spanning if and only if $\co(V)$ is pointed.
\end{thm}
\begin{proof}
Equation \eqref{eqn:CoV1} follows from
$$
x \in P(h) \Leftrightarrow 0 \in P(h-Vx) \Leftrightarrow h - Vx \in \R^n_+
$$
Thus $P(h) \ne \emptyset$ if and only if $h \in \R^n_+ + \im V$.

\eqref{eqn:CoV1} $\Rightarrow$ \eqref{eqn:CoV2}: By definition of $\bar V$ we have $\pi(e_i) = \bar v_i$, where $e_i$ is a standard basis vector of $\R^n$. Hence
$$
\co(V) = \pi(\wco(V)) = \pi(\R^n_+) = \pos(\bar v_1, \ldots, \bar v_n)
$$

\eqref{eqn:CoV2} $\Rightarrow$ \eqref{eqn:CoV4}: Represent the cone $\pos(\bar v_1, \ldots, \bar v_n)$ as an intersection of half-spaces whose boundary hyperplanes pass through the origin. It suffices to take only those half-spaces whose boundaries are spanned by a subconfiguration of $\bar V$. But every such half-space corresponds to a linear functional vanishing on a maximal non-spanning subset of $\bar V$ and positive otherwise. This yields the representation \eqref{eqn:CoV4}.

\eqref{eqn:CoV4} $\Rightarrow$ \eqref{eqn:CoV3}: Due to $\wco(V) = \pi^{-1}(\co(V))$ we have
$$
\wco(V) = \{h \in \R^n \mid \langle \mu^C, \pi(h) \rangle \ge 0\}
$$
Since $\bar V \colon \R^{n-d} \to \R^n$ is the adjoint of $\pi$ and $\lambda^C = \bar V \mu^C$, this corresponds to the description \eqref{eqn:CoV3}.


Finally, by Lemma \ref{lem:GaleProp} $V$ is positively spanning if and only if there exists a linear functional on $\R^{n-d}$ taking on $\bar v_i$ only positive values, which is equivalent to $\pos(\bar V)$ being pointed.

%
%
\end{proof}

Note that if $\pos(V)$ is pointed, then $V$ has no positive circuits, and hence the description \eqref{eqn:CoV4} yields $\co(V) = \R^{n-d}$. This is in full accordance with the dualization of the last statement of the theorem: $\bar V$ is positively spanning if and only if $\pos(v_1, \ldots, v_n)$ is pointed.

\begin{rem}
Parts \eqref{eqn:CoV1} and \eqref{eqn:CoV2} of Theorem \ref{thm:CoV} are due to McMullen, \cite{McM73}. See also \cite[Theorem 4.1.39]{LRS10}.

Equation \eqref{eqn:CoV3} can be seen as a version of Farkas Lemma, \cite[Chapter~1]{Zie95}.

Still another interpretation of \eqref{eqn:CoV3} is in terms of the support function of $P(h)$. We have $P(h) \ne \emptyset$ if and only if there exists a convex positively homogeneous function $\widetilde{h} \colon \R^d \to \R$ such that $\widetilde{h}(v_i) \le h_i$. The epigraph of $\widetilde{h}$ is a convex cone containing all points $(v_i,h_i)$. If we have $\lambda(h) \ge 0$ for all positive circuits, then $\conv\{(v_i,h_i)\}$ ``lies above'' $0$, and therefore $\pos\{(v_i,h_i)\}$ is the epigraph of a convex function.
\end{rem}

\subsubsection{Hyperbolic circuits and $\ir(V)$}
From now on we assume $v_i \ne 0$ and $v_i \ne \lambda v_j$ for $\lambda > 0$.

Our main goal here is to prove an analog of Theorem \ref{thm:CoV} for the space $\ir(V)$. For this, we need some preliminary work.

Recall that $P(h) = \{x \in \R^d \mid Vx \le h\}$ and that $h \in \wco(V) \Leftrightarrow P(h) \ne \emptyset$. Denote
$$
F_i(h) := \{x \in P(h) \mid \langle v_i, x \rangle = h_i\}
$$

\begin{lem}
\label{lem:FiNE}
If $h \in \wir(V)$, then $F_i(h) \ne \emptyset$.
\end{lem}
\begin{proof}
By definition, the $i$-th inequality in $Vx \le h$ is irredundant if and only if there exist $x_{out} \in \R^d$ such that
$$
\langle v_i, x_{out} \rangle > h_i, \quad \langle v_j, x_{out} \rangle \le h_j \ \forall j \ne i
$$
Pick $x_{in} \in P(h)$. Then for an appropriate convex combination of $x_{in}$ and $x_{out}$ we have $\langle v_i, x \rangle = h_i$ and $\langle v_j, x \rangle \le h_j$ for all $j \ne i$. Thus $F_i(h) \ne \emptyset$ and the lemma is proved.
\end{proof}

\begin{rem}
The inverse of Lemma \eqref{lem:FiNE} does not hold. If $F_i(h) \ne \emptyset$ and $\dim F_i(h) < \dim P(h) - 1$, then the $i$-th inequality is still redundant. Even worse, the $i$-th inequality can be redundant also when $\dim F_i(h) = \dim P(h) - 1$, although for this $\dim P(h) < d$ is necessary. A concrete example is the polytope in $\R^2$ given by
$$
y \le 0,\quad -y \le 0,\quad -x \le 0,\quad x+y \le 1,\quad x-y \le 1
$$
This is a segment with the endpoints $(0,0)$ and $(1,0)$, and we have $F_4 = F_5 = \{(1,0)\}$, which is a facet of $P$. Nevertheless, both the fourth and the fifth inequalities are redundant.

Note that in the last example removing both redundant inequalities at once makes the solution set larger. Thus an inclusion-minimal irredundant subsystem is in general not unique.
\end{rem}

The life becomes easier, if we restrict our attention either to the interior $\int\ir(V)$ or to the closure $\clir(V)$ of the irredundancy domain. Note that they are the images under the map \eqref{eqn:Pi} of the interior and of the closure of $\wir(V)$, respectively.

\begin{lem}
\label{lem:IntCoIr}
We have
\begin{subequations}
\begin{equation}
\label{eqn:IntCo}
\int\co(V) = \{\pi(h) \mid \dim P(h) = d\}
\end{equation}
\begin{equation}
\label{eqn:IntIr}
\int\ir(V) = \{\pi(h) \mid \dim P(h) = d,\, \dim F_i(h) = d-1 \ \forall i\}
\end{equation}
\begin{equation}
\label{eqn:ClIr}
\clir(V) = \{\pi(h) \mid F_i(h) \ne \emptyset \ \forall i\}
\end{equation}
\end{subequations}
\end{lem}
\begin{proof}
From \eqref{eqn:CoV1} we have $\int \wco(V) = \int(\R_+^n + \im V)$. It follows that
$$
h \in \int\co(V) \Leftrightarrow \exists x \colon h - Vx \in \int\R_+^n \Leftrightarrow \{x \mid Vx < h\} \ne \emptyset
$$
Since $\{x \mid Vx < h\} = \int P(h)$, and $\int P(h) \ne \emptyset \Leftrightarrow \dim P(h) = d$, we have \eqref{eqn:IntCo}.

Let $h \in \int\wir(V)$. Since $\wir(V) \subset \wco(V)$, we have $\int\wir(V) \subset \int\wco(V)$, and thus $\dim P(h) = d$. Since every $d$-dimensional convex polyhedron in $\R^d$ is the intersection of the half-spaces determined by its facets \cite[Chapter 2]{Zie95}, the $i$-th inequality is irredundant only if $\dim F_i(h) = d-1$. Thus the left hand side in \eqref{eqn:IntIr} is a subset of the right hand side.

Let us prove that the right hand side of \eqref{eqn:IntIr} is a subset of the left hand side. First, the right hand side is a subset of $\ir(V)$. Indeed, if $\dim F_i(h) = d-1$, then in a neigborhood of $x \in \relint F_i(h)$ there are points for which all inequalities in $Vx \le h$ hold except the $i$-th, that is the $i$-th inequality is irredundant. Further, any small change of $h$ preserves the properties $\dim P(h) = d$ and $\dim F_i(h) = d-1$, and thus the right hand side is a subset of $\int \ir(V)$.

By Lemma \ref{lem:FiNE}, $\ir(V)$ is a subset of the right hand side of \eqref{eqn:ClIr}. Since the right hand side is closed, it contains also $\clir(V)$.
To prove the inverse inclusion, let us show that any $h$ such that $F_i(h) \ne \emptyset$ lies in the closure of $\int\ir(V)$. Indeed, for any $\epsilon > 0$ put $h' = h + \epsilon \one$, where $\one = (1, \ldots, 1)$. Then for any $x \in F_i(h)$ we have
\begin{gather*}
\langle v_i, x + \epsilon v_i \rangle = \langle v_i, x \rangle + \epsilon = h_i + \epsilon\\
\langle v_j, x + \epsilon v_i \rangle = \langle v_j, x \rangle + \epsilon \langle v_i, v_j \rangle < h_j + \epsilon
\end{gather*}
(where we assume $\|v_i\| = 1$ for all $i$).
It follows that $x + \epsilon v_i \in F_i(h')$ and that $\dim F_i(h') = d-1$ for all $i$. Thus $h' \in \int\ir(V)$. This proves \eqref{eqn:ClIr}.
\end{proof}

\begin{lem}
\label{lem:1inIr}
If $(v_1, \ldots, v_n)$ are different unit vectors, then $\one \in \int\ir(V)$. In particular, $\int\ir(V)$ is non-empty.
\end{lem}
The proof is similar to that of the last part of Lemma \ref{lem:IntCoIr}: we have $v_i \in F_i(\one)$ and, since $\langle v_i, v_j \rangle < 1$, 
we have $\dim F_i(\one) = d-1$. The polytope $P(\one)$ is circumscribed about the unit sphere. Its normal fan is related to the Delaunay tessellation, see Section \ref{sec:Variations}.

Note that the assumption $\|v_i\| = 1$ is not too restrictive: scaling all vectors $v_i$ by positive factors scales the support numbers $h_i$ correspondingly. In the context of Gale duality, if $v_i$ is replaced by $\lambda_i v_i$, then it suffices to replace $\bar v_i$ by $\lambda_i^{-1} \bar v_i$.

Two more definitions will be needed.
\begin{dfn}
\label{dfn:Core}
Let $W = (w_1, \ldots, w_n)$ be a vector configuration in $\R^m$, where we allow $w_i = w_j$ for $i \ne j$. The \emph{$k$-core} $\core_k(W)$ of $W$ consists of all $y \in \R^m$ such that any linear functional that takes a non-negative value on $y$ takes a non-negative value on at least $k$ entries of $W$.
\end{dfn}
(The $k$-core is also called the set of vectors of \emph{depth} $k$, see \cite{Wag08}.)

If $w_i \ne 0$ for all $i$, then the $1$-core is the positive hull:
$$
\core_1(W) = \pos(W)
$$
Similarly, the $2$-core can be expressed as
$$
\core_2(W) = \bigcap_{i=1}^n \pos(W \setminus \{w_i\})
$$
For example, if for every $w_i$ there is a $j \ne i$ such that $w_j = w_i$, then $\core_2(W) = \core_1(W) = \pos(W)$.

\begin{dfn}
A circuit $C$ of a vector configuration $V$ is called \emph{hyperbolic}, if one of the coefficients of the corresponding linear dependence is positive, while the rest are negative. Equivalently, a hyperbolic circuit is an index subset $C = \{p(C)\} \cup C^-$ such that
\begin{equation}
\label{eqn:HypCircuit}
v_{p(C)} = \sum_{i \in C^-} \lambda^{C^-}_i v_i, \quad \text{where } \lambda^{C^-}_i > 0 \ \forall i \in C^-
\end{equation}
and every proper subset of $V_C = \{v_i \mid i \in C\}$ is linearly independent.
\end{dfn}
A hyperbolic circuit of cardinality $2$ consists of two non-zero vectors $v_i$ and $v_j$ such that $v_i = \lambda v_j$ for $\lambda > 0$. In this case any of the indices $i$ and $j$ can be declared to be $p(C)$. As we assumed at the beginning of this section $v_i \ne \lambda v_j$ for $\lambda > 0$, every hyperbolic circuit has cardinality at least $3$, and thus the positive index $p(C)$ is well-defined.

Equation \eqref{eqn:HypCircuit} means that we are scaling the coefficients of every hyperbolic circuit $C$ so that
$$
\lambda^C_{p(C)} = 1, \quad \lambda^C_i =
\begin{cases}
-\lambda_i^{C^-}, &\text{if } i \in C^-\\
0, &\text{if } i \notin C
\end{cases}
$$
Due to the dependence-evaluation duality (see \eqref{eqn:ValProp} and the two equations following it), to every hyperbolic circuit there corresponds a unique vector $\mu^C \in \R^{n-d}$ such that
$$
\langle \mu^C, \bar v_{p(C)} \rangle = 1, \quad \langle \mu^C, \bar v_i \rangle =
\begin{cases}
- \lambda_i^{C^-}, &\text{if } i \in C^-\\
0, &\text{if } i \notin \{p(C)\} \cup C^-
\end{cases}
$$

\begin{thm}
\label{thm:IrV}
Let $(v_1, \ldots, v_n)$ be a linearly spanning vector configuration in $\R^d$ such that
$$
v_i \ne 0 \ \forall i \quad \text{and} \quad v_i \ne \lambda v_j \ \forall i \ne j
$$
Then the closure $\clir(V)$ of the irredundancy domain is a convex polyhedral cone in $\R^n$. The set $\clir(V)$ and its lift $\wclir(V)$ to $\R^n$ can be described in the following ways.
\begin{subequations}
\begin{equation}
\label{eqn:IrV1}
\wclir(V) = \bigcap_{i=1}^n (\R_+^{[n] \setminus i} + \im V),
\end{equation}
where $\R_+^{[n] \setminus i} = \{(h_1, \ldots, h_n) \mid h_i = 0,\, h_j \ge 0 \ \forall j \ne i\}$.
\begin{equation}
\label{eqn:IrV2}
\clir(V) = \bigcap_{i=1}^n \pos (\bar V \setminus \{\bar v_i\}) = \core_2(\bar V)
\end{equation}
\begin{multline}
\label{eqn:IrV3}
\wclir(V) = \{ h \in \R^n \mid \langle \lambda^C, h \rangle \ge 0 \text{ for all positive circuits }C, \text{ and}\\
h_{p(C)} \le \langle \lambda^{C^-}, h \rangle \text{ for all hyperbolic circuits } C\}
\end{multline}
\begin{multline}
\label{eqn:IrV4}
\clir(V) = \{ y \in \R^{n-d} \mid \langle \mu^C, y \rangle \ge 0 \text{ for all positive circuits }C, \text{ and}\\
\langle \mu^C, y \rangle  \le 0 \text{ for all hyperbolic circuits } C\}
\end{multline}
\end{subequations}
Moreover, if $V$ is positively spanning, then $\ir(V)$ is pointed.
\end{thm}
\begin{proof}
From the definition of $F_i(h)$ it follows that
$$
x \in F_i(h) \Leftrightarrow h - Vx \in \R_+^{[n] \setminus i}
$$
Therefore
$$
F_i(h) \ne \emptyset \ \forall i \Leftrightarrow h \in \bigcap_{i=1}^n (\R_+^{[n] \setminus i} + \im V)
$$
Since $\wclir(V) = \{h \mid F_i(h) \ne \emptyset \ \forall i\}$, we have \eqref{eqn:IrV1}.

\eqref{eqn:IrV1} $\Rightarrow$ \eqref{eqn:IrV2}: Since $\R_+^{[n] \setminus i}$ is the positive hull of all basis vectors except $e_i$, we have $\pi(\R_+^{[n] \setminus i}) = \pos(\bar V \setminus \{\bar v_i\})$ and hence
$$
\clir(V) = \bigcap_{i=1}^n \pi(\R_+^{[n] \setminus i}) = \bigcap_{i=1}^n \pos (\bar V \setminus \{\bar v_i\}),
$$
(here $\cap$ and $\pi$ commute because $X + \im V$ is a full preimage under $\pi$).

\eqref{eqn:IrV2} $\Rightarrow$ \eqref{eqn:IrV3}: By definition of irredundancy, $h \in \wir(V)$ if and only if along with the system $Vx \le h$ the following $n$ systems:
\begin{equation}
\label{eqn:NewSyst}
\langle v_i, x \rangle > h_i, \ \langle v_j, x \rangle \le h_j \ \forall j \ne i,
\end{equation}
$i = 1, \ldots, n$, are compatible. Consequently, $h \in \wclir(V)$ if and only if all systems \eqref{eqn:NewSyst} are compatible when $>$ is replaced by $\ge$.
The system \eqref{eqn:NewSyst} can be rewritten as $V^{(i)}x \le h^{(i)}$, where
$$
V^{(i)} = (v_1, \ldots, -v_i, \ldots, v_n), \quad h^{(i)} = (h_1, \ldots, -h_i, \ldots, h_n)
$$
According to the criterion \eqref{eqn:CoV3} of compatibility, we have to consider all positive circuits of $V^{(i)}$. Each of them is either a positive circuit of $V \setminus \{v_i\}$ or a hyperbolic circuit of $V$ with $i$ as the positive index. This yields the system of linear inequalities in \eqref{eqn:IrV3}.

\eqref{eqn:IrV3} $\Rightarrow$ \eqref{eqn:IrV4}: This follows from $\lambda^C = \bar V \mu^C$ for any circuit and from our convention for hyperbolic circuits, see the paragraph before the theorem.

Finally, if $V$ is positively spanning, then by Theorem \ref{thm:CoV} the cone $\co(V)$ is pointed, and hence so is $\ir(V) \subset \co(V)$.
\end{proof}

\begin{rem}
Again, part \eqref{eqn:IrV3} of Theorem \ref{thm:IrV} can be interpreted in terms of the support function. Due to \eqref{eqn:ClIr} we have
$$
h \in \clir(V) \Leftrightarrow h_i = \widetilde{h}(v_i),
$$
where $\widetilde{h}$ is the support function of $P(h)$ and $\|v_i\| = 1$ is assumed. Since $\widetilde{h}$ is convex, equality $h_i = \widetilde{h}(v_i)$ implies the inequality $h_i \le \sum_{j \in C^-} \lambda_j^{C^-} h_j$ for every hyperbolic circuit $C$ with $p(C) = i$. In the opposite direction, if $h \notin \clir(V)$, then we have $h_i > \widetilde{h}(v_i)$ for some $i$. This yields a hyperbolic circuit $C$ (where $p(C) = i$ and $C^-$ are the indices of the extremal rays of the normal cone of $P(h)$ containing $v_i$ in its relative interior) that violates the linear inequality in the second line of \eqref{eqn:IrV3}.
\end{rem}

\subsubsection{Chambers and type cones}
\label{sec:Chambers}
Let $V = (v_1, \ldots, v_n)$ be a positively spanning configuration of $n$ different unit vectors in $\R^d$. By results of the previous sections, every polytope $P(h)$ with outer facet normals $v_i$ is represented by a point $y = \pi(h) \in \R^{n-d}$ lying in the interior of the $2$-core of the Gale dual configuration $\bar V = (\bar v_1, \ldots, \bar v_n)$. We now want to describe how the relative position of $y$ with respect to $\bar v_i$ determines the combinatorics of the polytope.

By slightly modifying the notation, we will distinguish between the \emph{geometric} normal fan $\cN(P(h))$ and the \emph{abstract} fan $\Delta$ isomorphic to $\cN(P(h))$. Here $\Delta$ is a collection of subsets of $[n]$ such that
$$
\sigma \in \Delta \Leftrightarrow \pos(V_\sigma) \in \cN(P(h)),
$$
where as usual $V_\sigma = \{v_i \mid i \in \sigma\}$.

\begin{lem}
\label{lem:CombGale}
Let $y = \pi(h) \in \int \core_2(\bar V)$. Then
\begin{equation}
\pos(V_\sigma) \in \cN(P(h)) \Leftrightarrow y \in \relint \pos (\bar V_{[n]\setminus\sigma})
\end{equation}
In other words, $\{v_i \mid i \in \sigma\}$ span a normal cone of $P(h)$ if and only if $\pi(h)$ lies in the relative interior of the cone spanned by $\{\bar v_i \mid i \notin \sigma\}$.
\end{lem}
\begin{proof}
Let $F_\sigma(h)$ denote the face of $P(h)$ with the normal cone $\pos(V_\sigma)$. We have
$$
x \in \relint F_\sigma(h) \Leftrightarrow \langle v_i, x \rangle
\begin{cases}
= h_i, &\text{if }i \in \sigma\\
< h_i, &\text{if }i \notin \sigma
\end{cases}
\Leftrightarrow h - Vx \in \relint \R_+^{[n] \setminus \sigma}
$$
By applying the projection $\pi$, we obtain
$$
F_\sigma(h) \ne \emptyset \Leftrightarrow y \in \relint \pos(V_{[n] \setminus \sigma})
$$
But $F_\sigma(h) \ne \emptyset$ is equivalent to $\pos(V_\sigma) \in \cN(P(h))$, and we are done.
\end{proof}

Lemma \ref{lem:CombGale} allows to give a very concise description of the type cones of the vector configuration $V$.

\begin{dfn}
\label{dfn:Cha}
Let $W = (w_1, \ldots, w_n) \subset \R^m$ be a vector configuration. Two vectors $y_1$ and $y_2$ are said to lie in the same \emph{relatively open chamber}, if for every $I \subset [n]$ we have
$$
y_1 \in \pos(W_I) \Leftrightarrow y_2 \in \pos(W_I)
$$
The closure of a relatively open chamber is called a \emph{chamber}, and the collection of all chambers is called the \emph{chamber fan} $\Ch(W)$ of $W$.
\end{dfn}
Equivalently, a relatively open chamber is an inclusion-minimal intersection of relative interiors of cones generated by $W$, and the chamber fan is the coarsest common refinement of all cones $\pos(W_I)$, $I \subset [n]$.

\begin{cor}
\label{cor:DeltaChamb}
Let $\Delta$ be a complete pointed fan with rays generated by $V$. Then $T(\Delta)$ is a relatively open chamber in the chamber fan $\Ch(\bar V)$.
\end{cor}
Indeed, two points $y_1, y_2 \in \R^{n-d}$ belong to the same type cone if and only if the corresponding polytopes have the same normal fan. By Lemma \ref{lem:CombGale} this is equivalent to $y_1$ and $y_2$ belonging to the same collections of relatively open cones $\pos(\bar V_I)$. Since relatively open and closed cones are related through the $\cup$ and $\setminus$ operations, this is equivalent to $y_1$ and $y_2$ belonging to the same collections of closed cones.

\begin{lem}
Let $\Delta$ be a complete pointed fan with rays generated by $V$. Then
\begin{equation}
\label{eqn:TChamber}
T(\Delta) = \bigcap_{\sigma \in \Delta} \relint \pos (\bar V_{[n]\setminus\sigma}) = \bigcap_{\rho \in \Delta^{(d)}} \relint \pos (\bar V_{[n]\setminus\rho})
\end{equation}
In particular, the fan $\Delta$ is polytopal if and only if the intersection on the right hand side is non-empty.
\end{lem}
\begin{proof}
By Lemma \ref{lem:CombGale}, the left hand side of \eqref{eqn:TChamber} is a subset of the middle. Clearly, the middle is a subset of the right hand side.

Let us show that the right hand side is a subset of the left hand side. Let $\pi(h) \in \relint \pos (\bar V_{[n]\setminus\rho})$ for all $\rho \in \Delta^{(d)}$. Since for every $i$ there is $\rho \in \Delta^{(d)}$ such that $i \in \rho$, we have
$$
\bigcap_{\rho \in \Delta^{(d)}} \relint \pos (\bar V_{[n]\setminus\rho}) \subset \bigcap_{i=1}^n \int \pos(\bar V_{[n] \setminus i}) = \int \core_2(\bar V) = \int \ir(V)
$$
(here $\rank \bar V_{[n] \setminus i} = n-d$ for all $i$ allows to replace $\relint$ by the absolute interior).
Hence by \eqref{eqn:IntIr} $P(h)$ is a $d$-polytope with $n$ facets having outer normals $v_1, \ldots, v_n$ and we are in a position to apply Lemma \ref{lem:CombGale}. It implies that $\pos(V_\rho) \in \cN(P(h))$ for all $\rho \in \Delta^{(d)}$. Since the cones $\pos(V_\rho)$ cover $\R^d$, the polytope $P(h)$ has no other full-dimensional normal cones, and thus $\cN(P(h)) = \Delta$.
\end{proof}

Note that if the fan $\Delta$ is simplicial, then we have $|\rho| = d$ for all $\rho \in \Delta^{(d)}$, and hence $|[n] \setminus \rho| = n-d$. According to Lemma \ref{lem:GaleProp}, the vector configuration $\bar V_{[n]\setminus\rho}$ is linearly independent, so that \eqref{eqn:TChamber} represents $T(\Delta)$ as a finite intersection of open simplicial $d$-cones.

\begin{rem}
Chambers of $\Ch(\bar V)$ contained in $\partial \clir(V) \cap \int \co(V)$ can be identified with type cones $T(\Delta)$ for complete pointed fans whose rays are generated by proper subsets of $V$. Chambers contained in $\int \co(V) \setminus \clir(V)$ are linearly isomorphic to $T(\Delta) \times \int \R_+^I$, where $I \subset [n]$ is the index set of redundant inequalities.
\end{rem}

Chambers outside $\int\ir(V)$ can be studied using the operations of \emph{contraction} and \emph{deletion} on vector configurations. It is easy to see that as a Gale dual of $V_{[n] \setminus i}$ (configuration obtained by deletion) one can take projections of vectors of $\bar V_{[n] \setminus i}$ along $\bar v_i$ (configuration obtained by contraction). According to that, the type cones $T(\Delta)$ with $\Delta^{(1)}$ generated by $V_{[n] \setminus i}$ are the chambers of the contracted Gale dual $\bar V_{[n]\setminus i}$. The contraction can be seen as a central projection in $\R^{n-d}$ from $\bar v_i$. This projection maps those chambers of $\Ch(\bar V) \cap \partial \clir(V)$ visible from $\bar v_i$ to the chambers of $\Ch(\bar V_{[n]\setminus i}) \cap \clir(V_{[n] \setminus i})$.

Chambers on the boundary of $\co(V)$ correspond to type cones of polytopes of dimension smaller than $d$. The facet normals of such a polytope $P$ are obtained from $V$ by contraction along vectors $v_i$ orthogonal to $\aff(P)$.

\begin{rem}
The chambers fan $\Ch(\bar V)$ is polytopal. Namely, Lemma \ref{lem:FanMink} and the observation after Definition \ref{dfn:Cha} implies that $\Ch(\bar V)$ is the normal fan of the Minkowski sum of representatives of all normal equivalence types with facet normals in $\bar V$. The fan $\Ch(\bar V)$ is called the \emph{secondary fan} of the vector configuration $V$, and any polyhedron that has $\Ch(\bar V)$ as its normal fan is called a \emph{secondary polyhedron} of $V$, \cite{BFS90, BGS93}.

If the cone $\pos(V)$ is pointed, then by Lemma \ref{lem:GaleProp} $\bar V$ is positively spanning; thus a secondary polyhedron is compact and is called a \emph{secondary polytope}. The interest in secondary polytopes was aroused by their applications in algebraic geometry, see the book \cite{GKZ94} by Gel'fand, Kapranov, and Zelevinsky.
\end{rem}

\begin{rem}
\label{rem:Monotypic}
Polytopes whose facet normals determine their normal equivalence class are called monotypic.
For such polytopes,
$$
\clir(V) = \cl T(\Delta),
$$
for a unique polytopal fan $\Delta$ with $\Delta^{(1)} = V$.
For example, polygons and polygonal prisms (including  parallelepipeds) are monotypic. Monotypic polytopes are studied in \cite{MSS74}. 
\end{rem}

\begin{rem}
\label{rem:AnyType}
One can show that any convex polyhedral cone is combinatorially (and even affinely) isomorphic to a chamber of some chamber fan $\Ch(W)$. By duplicating each vector in $W$ we obtain $\core_2(W) = \pos(W)$, so that all chambers lie in the $2$-core. It follows that every convex polyhedral cone can be realized as a type cone of some vector configuration.
\end{rem}

\begin{rem}
When the vector configuration $V$ is positively spanning, its Gale dual $\bar V$ can be represented by a point configuration in an affine hyperplane $A \subset \R^{n-d}$, see the paragraphs preceding Remark \ref{rem:AffGale}. Then the linear functionals $\mu$ on $\R^{n-d}$ are replaced by affine functions on $A$, in particular in the definition of the $k$-core.
\end{rem}

\subsubsection{Facets of $\clir(V)$ and truncated polytopes}
Recall that $C \subset [n]$ is a circuit of $V$ if and only if $[n] \setminus C$ is a cocircuit of $\bar V$, that is the index set of a maximal non-spanning subconfiguration. This subconfiguration spans a hyperplane which is the kernel of the linear functional $\mu^C$, see the beginning of Section \ref{sec:PosCirc}.

By \eqref{eqn:IrV4}, every facet of $\clir(V) = \core_2(\bar V)$ is determined by a positive or a hyperbolic circuit of $V$. The corresponding hyperplane in $\R^{n-d}$ spanned by $\bar V_{[n] \setminus C}$ has all of $\bar V_C$ on one of its sides if $C$ is positive, and separates $\bar v_{p(C)}$ from $\bar V_{C^-}$ if $C$ is hyperbolic.

Although every facet corresponds to a cocircuit, not every cocircuit (even not every hyperbolic one) corresponds to a facet of $\core_2(\bar V)$, as the following example shows.

\begin{exl}
The point configuration on Figure \ref{fig:5gonExl}, right has $3$ positive and $7$ hyperbolic cocircuits. None of the positive cocircuits and only $3$ of the hyperbolic cocircuits determine a facet of its 2-core.

\begin{figure}[ht]
\centering
\begin{picture}(0,0)%
\includegraphics{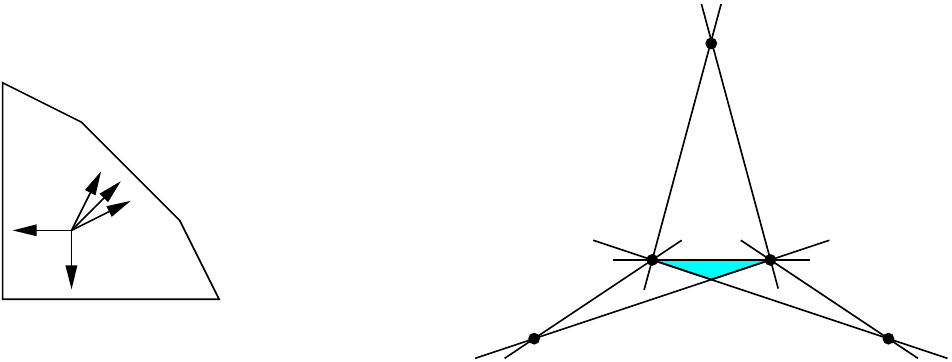}%
\end{picture}%
\setlength{\unitlength}{4144sp}%
\begingroup\makeatletter\ifx\SetFigFont\undefined%
\gdef\SetFigFont#1#2#3#4#5{%
  \reset@font\fontsize{#1}{#2pt}%
  \fontfamily{#3}\fontseries{#4}\fontshape{#5}%
  \selectfont}%
\fi\endgroup%
\begin{picture}(4344,1644)(-191,-613)
\end{picture}%
\caption{A configuration of five vectors in $\R^2$ and its affine Gale dual with a shaded $2$-core.}
\label{fig:5gonExl}
\end{figure}

This point configuration is the affine Gale dual of the five vectors in $\R^2$ shown on Figure \ref{fig:5gonExl}, left. Thus the points of the shaded triangle on the right parametrize the space of pentagons with normals as that on the left.
\end{exl}

In the next example every positive and every hyperbolic cocircuit determines a facet of the $2$-core.

\begin{exl}
The 2-core of the vertices of a dodecahedron is the truncated icosahedron. Every hyperbolic cocircuit gives rise to a hexagonal facet, and every positive cocircuit to a pentagonal facet, see Figure \ref{fig:DodecaCore}.

\begin{figure}[ht]
\centering
\includegraphics{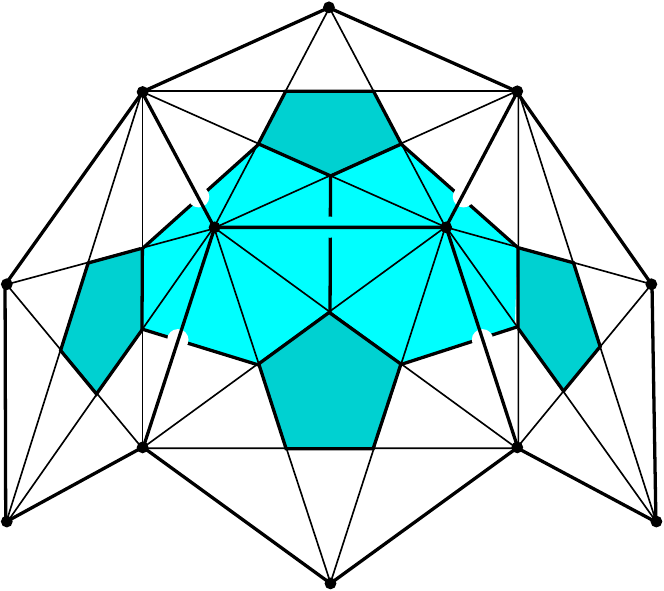}
\caption{A fragment of the $2$-core of the vertices of the dodecahedron.}
\label{fig:DodecaCore}
\end{figure}

Vertices of a dodecahedron correspond to a configuration of $12$ vectors in $\R^4$ which is Gale dual to a configuration of $12$ vectors in $\R^8$.
\end{exl}

\begin{dfn}
\label{dfn:PolTrunc}
Let $P \subset \R^d$ be a $d$-polytope, and $F$ be a proper face of $P$. We say that a polytope $P'$ is obtained from $P$ by \emph{truncating} the face $F$ if
$$
P' = P \cap \{x \in \R^d \mid \langle v, x \rangle \le h_P(v) - \epsilon\},
$$
where $v$ is any unit vector in $\relint N_F(P)$, and $\epsilon > 0$ is sufficiently small (so small that $\Vert(P) \setminus \Vert(F) \subset \Vert(P')$, $\Vert$ denoting the set of vertices).
\end{dfn}

Note that $P \cap \{x \in \R^d \mid \langle v, x \rangle = h_P(v)\} = F$. Thus geometrically truncation means pushing a supporting hyperplane of $P$ at $F$ inwards. The operation of truncating an edge of a $3$-polytope is shown on Figure \ref{fig:TruncEdge}.

\begin{figure}[ht]
\centering
\begin{picture}(0,0)%
\includegraphics{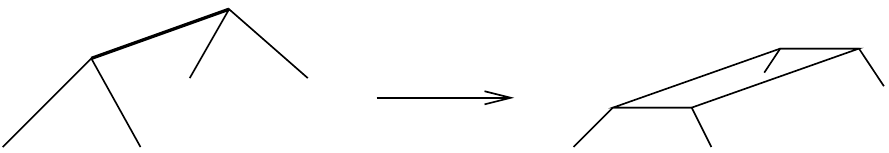}%
\end{picture}%
\setlength{\unitlength}{4144sp}%
\begingroup\makeatletter\ifx\SetFigFont\undefined%
\gdef\SetFigFont#1#2#3#4#5{%
  \reset@font\fontsize{#1}{#2pt}%
  \fontfamily{#3}\fontseries{#4}\fontshape{#5}%
  \selectfont}%
\fi\endgroup%
\begin{picture}(4053,689)(484,-537)
\put(1320,-473){\makebox(0,0)[lb]{\smash{{\SetFigFont{10}{12.0}{\rmdefault}{\mddefault}{\updefault}{\color[rgb]{0,0,0}$P$}%
}}}}
\put(4014,-482){\makebox(0,0)[lb]{\smash{{\SetFigFont{10}{12.0}{\rmdefault}{\mddefault}{\updefault}{\color[rgb]{0,0,0}$P'$}%
}}}}
\put(1081, 29){\makebox(0,0)[lb]{\smash{{\SetFigFont{10}{12.0}{\rmdefault}{\mddefault}{\updefault}{\color[rgb]{0,0,0}$F$}%
}}}}
\end{picture}%
\caption{Truncating an edge of a $3$-polytope.}
\label{fig:TruncEdge}
\end{figure}
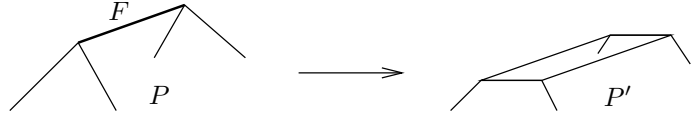

For any positive or hyperbolic circuit $C$ of $V$ denote
\begin{equation}
\label{eqn:ClirFacet}
\clir^C(V) = \clir(V) \cap \{y \in \R^{n-d} \mid \langle \mu^C, y \rangle = 0\}
\end{equation}

\begin{lem}
\label{lem:BdryIr}
Let $C = \{p(C)\} \cup C^-$ be a hyperbolic circuit of $V$ such that $\clir^C(V)$ is a facet of $\clir(V)$. Then for every $h \in \R^n$ such that $\pi(h) \in \relint \clir^C(V)$ the following holds.
\begin{enumerate}
\item
$P(h)$ is a $d$-dimensional polytope with outer facet normals $V_{[n] \setminus p(C)}$;
\item
$\pos(V_{C^-}) \in \cN(P(h))$;
\item
$P(h - \epsilon e_{p(C)})$ is obtained from $P(h)$ by truncating the face $F_{C^-}$, provided that $\epsilon > 0$ is sufficiently small.
\end{enumerate}
\end{lem}
\begin{proof}
It is easy to show that for a hyperbolic circuit $C$, $\relint \clir^C(V) \subset \int \co(V)$. Thus $h \in \relint \clir^C(V)$ implies $\dim P(h) = d$.

Next, if $\dim P(h) = d$, then $\dim F_i(h) = d-1$ is equivalent to the compatibility of the system
$$
\langle v_i, x \rangle = h_i, \quad \langle v_j, x \rangle < h_j \ \forall j \ne i
$$
which by a standard argument is equivalent to
\begin{equation}
\label{eqn:IntPos}
\pi(h) \in \int \pos (\bar V_{[n] \setminus i})
\end{equation}
The assumption $\pi(h) \in \relint \clir^C(V)$ implies that \eqref{eqn:IntPos} holds for all $i \ne p(C)$, bud doesn't hold for $i = p(C)$. Hence $F_i(h)$ is a facet of $P(h)$ iff $i \ne p(C)$. This finishes the proof of part 1).

For part 2), observe that $\clir^C(V) \subset \pos(\bar V_{[n] \setminus C})$, which implies
$$
\relint \clir^C(V) \subset \relint \pos(\bar V_{[n] \setminus C})
$$
Then the argument from the proof of Lemma \ref{lem:CombGale} implies that $\pos(V_{C^-})$ belongs to the normal fan.

Finally, the hyperbolic circuit relation \eqref{eqn:HypCircuit} implies $v_{p(C)} \in \relint \pos(V_{C^-})$. Thus by Definition \ref{dfn:PolTrunc} $P(h - \epsilon e_{p(C)})$ is obtained from $P(h)$ by truncating $F_{C^-}$.
\end{proof}

\begin{lem}
\label{lem:Bdry2Ir}
Let $C_1$ and $C_2$ be two hyperbolic circuits such that $\clir^{C_1}(V)$ and $\clir^{C_2}(V)$ are two facets of $\clir(V)$ intersecting along a codimension 2 face of $\clir(V)$. Then for every $h \in \R^n$ such that $\pi(h) \in \relint (\clir^{C_1}(V) \cap \clir^{C_2}(V))$ the following holds.
\begin{enumerate}
\item
$P(h)$ is a $d$-dimensional polytope with outer facet normals $V_{[n] \setminus \{p_1,p_2\}}$, where $p_i$ is the positive index of $C_i$, $i = 1, 2$;
\item
if $p_1 \notin C_2$ and $p_2 \notin C_1$ (in particular, $p_1 \ne p_2$), then
$\pos(V_{C_i^-}) \in \cN(P(h))$ for $i = 1, 2$;
\item
under the assumptions
$$
p_1 \notin C_2, p_2 \notin C_1, \quad \pos(V_{C_1^- \cup C_2^-}) \notin \cN(P(h))
$$
the polytopes $P(h - \epsilon_1 e_{p_1} - \epsilon_2 e_{p_2})$ for all sufficiently small $\epsilon_1, \epsilon_2 > 0$ are obtained from $P(h)$ by independentely truncating the faces $F_{C_1^-}$ and $F_{C_2^-}$.
\end{enumerate}
Here truncations of two faces are called independent if the truncated parts are disjoint.
\end{lem}
\begin{proof}
The first part is similar to that of Lemma \ref{lem:BdryIr}, if one notes that the relative interior of a codimension 2 face belongs to 2 faces only.

For the second part we have to prove
\begin{equation}
\label{eqn:Relint}
\relint (\clir^{C_1}(V) \cap \clir^{C_2}(V)) \subset \relint \pos(\bar V_{[n] \setminus C_i}) \quad \text{for }i = 1, 2
\end{equation}
In other words, for two cones $\clir^{C_1}(V) \subset \pos(\bar V_{[n] \setminus C_1})$ of the same dimension we have to show that the facet $\clir^{C_1}(V) \cap \clir^{C_2}(V)$ of the former is not contained in a facet of the latter (and the same with indices 1 and 2 exchanged). The ray $\R_+ \bar v_{p_2}$ is an extreme ray of $\pos \bar V$ and, since $p_2 \in [n] \setminus C_1$, also of $\pos(\bar V_{[n] \setminus C_1})$. Since $\R_+ v_{p_2}$ and $\clir^{C_1}(V)$ lie on different sides from the facet $\clir^{C_1}(V) \cap \clir^{C_2}(V)$, this facet cannot be contained in a facet of $\pos(\bar V_{[n] \setminus C_i})$, and we are done.

Finally, the third part is true because $\pos(V_{C_1^- \cup C_2^-}) \notin \cN(P(h))$ implies that the faces $F_{C_1^-}$ and $F_{C_2^-}$ of $P(h)$ are disjoint. Hence all small truncations of those faces are independent. The lemma is proved.
\end{proof}

\subsection{Examples}
\label{sec:Example}
In the following examples of vector configurations $V$ we analyze the closure $\clir(V)$ of the irredundancy domain and its decomposition into type cones. We are using both the direct approach (``what happens when the facets of a polytope are translated'') and the more formal one, through Gale diagram, circuits, and the chamber fan.

\subsubsection{Parallelepipeds with fixed face directions}
Let $(v_1, v_2, v_3)$ be a basis of $\R^3$, put $v_{i+3} = -v_i$, $i=1,2,3$, and consider the resulting configuration $V$ of six vectors in $\R^3$.

All polyhedra with facet normals $V$ are normally equivalent: they are parallelepipeds with face normals $\pm v_i$. The normal fan $\Delta$ is generated by a hyperplane arrangement spanned on the vectors $v_1, v_2, v_3$. 

The lifted type cone $\wT(\Delta)$ consists of $h \in \R^6$ that satisfy
\begin{equation}
\label{eqn:Par}
h_1 + h_4 > 0, \quad h_2 + h_5 > 0, \quad h_3 + h_6 > 0
\end{equation}
By identifying $\R^6/\im V$ with $\{h \in \R^6 \mid h_4 = h_5 = h_6 = 0\}$, we obtain
$$
\cl T(\Delta) \cong \{(h_1,h_2,h_3) \mid h_i \ge 0, i = 1, 2, 3\}
$$
The extreme rays of the cone $\cl T(\Delta)$ correspond to degeneration of a parallelepiped into a segment parallel to one of the three vectors $v_1 \times v_2$, $v_2 \times v_3$, $v_3 \times v_1$.

The Gale diagram of $V$ has the property $\bar v_{i+3} = \bar v_i$, $i = 1, 2, 3$. This fits together with the fact $\co(V) = \ir(V)$: the second core of a vector configuration where each vector is repeated twice is its convex hull, see Definition \ref{dfn:Core} and equation \eqref{eqn:IrV2}. The facets of the cone \eqref{eqn:Par} correspond to the three positive circuits $\{1,4\}$, $\{2,5\}$, and $\{3,6\}$.

\subsubsection{Polygons}\label{ex:poly1}

Let $\alpha:=(\alpha_1,\ldots,\alpha_n)$ with $0<\alpha_i<\pi$ be an $n$-tuple of real numbers such that $\sum_i \alpha_i=2\pi$.
Let $v_1, \ldots, v_n \in \R^2$ be unit vectors such that the angle from $v_i$ to $v_{i+1}$ equals $\alpha_{i+1}$ (of course, the indices are taken modulo $n$).
This determines a fan $\Delta$ in $\R^2$. The fan is polytopal since for every collection of positive numbers $a_1, \ldots, a_n$ such that $\sum_{i=1}^n a_i v_i = 0$ there is a convex polygon with edge lengths $a_i$ and edge normals $v_i$.
We will denote $T(\Delta)$ by $T(\alpha)$. Note that $\cl T(\alpha)$ is a pointed $(n-2)$-dimensional cone. To avoid dealing with trivial cases, below we assume $n \ge 5$.

By Lemma \ref{lem:TypeConeIneq}, every facet of $T(\alpha)$  corresponds to vanishing of an edge: $\ell_i=0$. This is possible without making any other edges to disappear if and only if $\alpha_i+\alpha_{i+1}\leq\pi$.
It follows that $\cl T(\alpha)$ has $n$, $n-1$ or $n-2$ facets.
Denote by $T_i$ the facet $\ell_i = 0$ of $\cl T(\alpha)$. If $j\notin\{i,i-1,i+1\}$ then $T_i$ and $T_j$ meet along a codimension $2$ facet. Otherwise, $T_i$ and $T_{i+1}$
meet if and only if $\alpha_i+\alpha_{i+1}+\alpha_{i+2}<\pi$. 

An extreme ray $e$ of $\cl T(\alpha)$ corresponds to triangles with fixed edge normals $v_{i_1}, v_{i_2}, v_{i_3}$. This means that $\ell_j = 0$ for all $j \notin \{i_1, i_2, i_3\}$ while $\ell_{i_\alpha} \ne 0$. Thus $e$ is contained in exactly $n-3$ facets, and hence $T(\alpha)$ is simple. It follows that
$\cl T(\alpha)$ is the cone over an $(n-3)$-dimensional polytope that is either a simplex, or truncated simplex, or doubly truncated simplex.
See  \cite{BG92} for more details
and the case of non-convex
polygons.

In the Gale diagrams language, the cone $\cl T(\alpha)$ is the 2-core of the Gale dual $\bar V$. Figure \ref{fig:5gonExl} shows an example for $n=5$, where this cone has three facets.

\subsubsection{Polygonal prisms}
\label{sec:PolPrisms1}
Embed $\R^2$ as a coordinate plane in $\R^3$ and add to the vectors $v_1, \ldots, v_n \in \R^2$ from the preceding example the third basis vector and its inverse: $v_{n+1} = e_3$ and $v_{n+2} = -e_3$. This new vector configuration $V^+$ determines only one pointed fan, namely the normal fan of a prism over an $n$-gon. Denote this fan by $\Delta^+$, and its type cone by $T^+(\alpha)$. Then we have
$$
(h, h_{n+1}, h_{n+2}) \in T^+(\alpha) \Leftrightarrow h \in T(\alpha) \text{ and } h_{n+1} + h_{n+2} > 0
$$
Thus $T^+(\alpha)$ is a product of $T(\alpha)$ with a half-space. It follows that
$$
T^+(\alpha) = T(\alpha) \times \R^+
$$
The new extreme ray $\{0\} \times \R^+$ corresponds to degeneration of the prism into a segment parallel to $e_3$.

The Gale diagram of $V^+$ lives in the space $\R^{n-1}$ which is one dimension higher than that for $\bar V$. It is easy to see that $\bar{V^+}$ is obtained from $\bar V$ by adding two equal vectors $\bar v_{n+1} = \bar v_{n+2} = e_{n-1}$. It follows that the 2-core of $\bar V \cup \{e_{n-1}, e_{n-1}\}$ is the pyramid over the 2-core of $\bar V$, which yields the same result as above.

%

\subsubsection{Triangular bipyramid}
\label{sub:bipyr1}
Let $u_1, u_2, u_3 \in \R^2$ be outward unit normals to the edges of a regular triangle. Choose $\lambda, \mu > 0$ such that $\lambda^2 + \mu^2 = 1$ and consider the following six vectors in $\R^6$:
\begin{equation}
\label{eqn:VectBipyr}
\begin{aligned}
&v_1 = \lambda u_1 + \mu e_3,  &&v_2 = \lambda u_2 + \mu e_3, &&v_3 = \lambda u_3 + \mu e_3 \\
&v_4 = \lambda u_1 - \mu e_3,  &&v_5 = \lambda u_2 - \mu e_3, &&v_6 = \lambda u_3 - \mu e_3
\end{aligned}
\end{equation}
The 3-polytope $Vx \le \one$ is a bipyramid over a triangle, see Figure \ref{fig:Bipyr}. The right half of the picture schematically shows the normal fan of the bipyramid (as a stereographic projection of the intersection of the normal fan with the sphere).

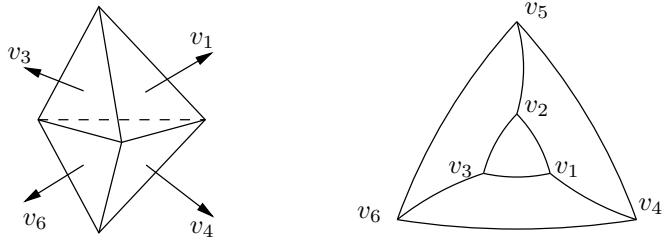
\begin{figure}[ht]
\centering
\begin{picture}(0,0)%
\includegraphics{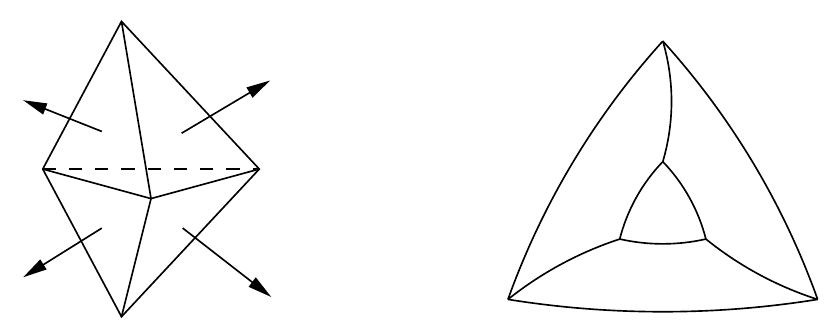}%
\end{picture}%
\setlength{\unitlength}{4144sp}%
\begingroup\makeatletter\ifx\SetFigFont\undefined%
\gdef\SetFigFont#1#2#3#4#5{%
  \reset@font\fontsize{#1}{#2pt}%
  \fontfamily{#3}\fontseries{#4}\fontshape{#5}%
  \selectfont}%
\fi\endgroup%
\begin{picture}(3765,1500)(346,-796)
\put(3421,-16){\makebox(0,0)[lb]{\smash{{\SetFigFont{10}{12.0}{\rmdefault}{\mddefault}{\updefault}{\color[rgb]{0,0,0}$v_2$}%
}}}}
\put(2971,-376){\makebox(0,0)[lb]{\smash{{\SetFigFont{10}{12.0}{\rmdefault}{\mddefault}{\updefault}{\color[rgb]{0,0,0}$v_3$}%
}}}}
\put(4096,-601){\makebox(0,0)[lb]{\smash{{\SetFigFont{10}{12.0}{\rmdefault}{\mddefault}{\updefault}{\color[rgb]{0,0,0}$v_4$}%
}}}}
\put(3421,569){\makebox(0,0)[lb]{\smash{{\SetFigFont{10}{12.0}{\rmdefault}{\mddefault}{\updefault}{\color[rgb]{0,0,0}$v_5$}%
}}}}
\put(1441,389){\makebox(0,0)[lb]{\smash{{\SetFigFont{10}{12.0}{\rmdefault}{\mddefault}{\updefault}{\color[rgb]{0,0,0}$v_1$}%
}}}}
\put(1441,-736){\makebox(0,0)[lb]{\smash{{\SetFigFont{10}{12.0}{\rmdefault}{\mddefault}{\updefault}{\color[rgb]{0,0,0}$v_4$}%
}}}}
\put(2431,-646){\makebox(0,0)[lb]{\smash{{\SetFigFont{10}{12.0}{\rmdefault}{\mddefault}{\updefault}{\color[rgb]{0,0,0}$v_6$}%
}}}}
\put(361,299){\makebox(0,0)[lb]{\smash{{\SetFigFont{10}{12.0}{\rmdefault}{\mddefault}{\updefault}{\color[rgb]{0,0,0}$v_3$}%
}}}}
\put(3601,-376){\makebox(0,0)[lb]{\smash{{\SetFigFont{10}{12.0}{\rmdefault}{\mddefault}{\updefault}{\color[rgb]{0,0,0}$v_1$}%
}}}}
\put(451,-691){\makebox(0,0)[lb]{\smash{{\SetFigFont{10}{12.0}{\rmdefault}{\mddefault}{\updefault}{\color[rgb]{0,0,0}$v_6$}%
}}}}
\end{picture}%
\caption{The bipyramid over triangle and its normal fan.}
\label{fig:Bipyr}
\end{figure}

By translating the faces of the bipyramid, we can split its four-valent vertices into pairs of three-valent vertices. For the spherical section of the normal fan this means subdividing quadrilaterals by their diagonals. Let us apply the Gale diagram technique to study the arrangement of the type cones that correspond to different combinatorial types.

Let $V$ be the $6\times 3$-matrix with rows $v_i$. We have to find a $6\times 3$-matrix $\bar V$ of rank $3$ whose columns are orthogonal to those of $V$. The matrix $\bar V$ is unique up to a multiplication from the right with an element of $\GL(\R, 3)$. Since the vectors $v_4, v_5, v_6$ form a basis of $\R^3$, the vectors $\bar v_1, \bar v_2, \bar v_3$ must do the same (see Lemma \ref{lem:GaleProp}). Thus we may assume
$$
\bar v_1 = (1, 0, 0), \quad \bar v_2 = (0, 1, 0), \quad \bar v_3 = (0, 0, 1)
$$
The remaining entries of $\bar V$ can be easily determined from the orthogonality condition between the columns of $V$ and $\bar V$:
$$
V =
\begin{pmatrix}
\lambda u_1 & \mu\\
\lambda u_2 & \mu\\
\lambda u_3 & \mu\\
\lambda u_1 & -\mu\\
\lambda u_2 & -\mu\\
\lambda u_3 & -\mu
\end{pmatrix}
\qquad
\bar V =
\begin{pmatrix}
1 & 0 & 0\\
0 & 1 & 0\\
0 & 0 & 1\\
-\frac13 & \frac23 & \frac23\\
\frac23 & -\frac13 & \frac23\\
\frac23 & \frac23 & -\frac13
\end{pmatrix}
$$

Recall that the map
$$
\bar V^\top \colon \R^n \to \R^{n-d} \cong \R^n/\im V
$$
projects the space of support vectors $h$ to its quotient by translations of the polytope $P(h)$ (see equation \eqref{eqn:PTransl} and Definition \ref{dfn:Gale}). We identify $\R^{n-d}$ with a subspace of $\R^n$ via a right inverse $\iota$ of $\bar V^\top$. In our case, we can put $\iota(e_i) = e_i$ for $i = 1, 2, 3$, so that $\R^{n-d} = \R^3$ is identified with the subspace $h_4 = h_5 = h_6 = 0$ of $\R^6$. Geometrically this corresponds to fixing the lower vertex of the bipyramid at the origin and varying only the heights $h_1, h_2, h_3$.

The rows of $\bar V$ are the coordinates of six vectors forming the Gale diagram of $V$. Since all of them lie in the subspace $h_1 + h_2 + h_3 > 0$, we can conveniently draw the \emph{affine Gale diagram} by intersecting the cones generated by $\bar V$ with the affine hyperplane $h_1 + h_2 + h_3 = 1$, see Figure \ref{fig:GaleCoIr}, left. (By a lucky coincidence, the vectors $\bar v_i$ lie in this plane.)

\begin{figure}[ht]
\centering
\begin{picture}(0,0)%
\includegraphics{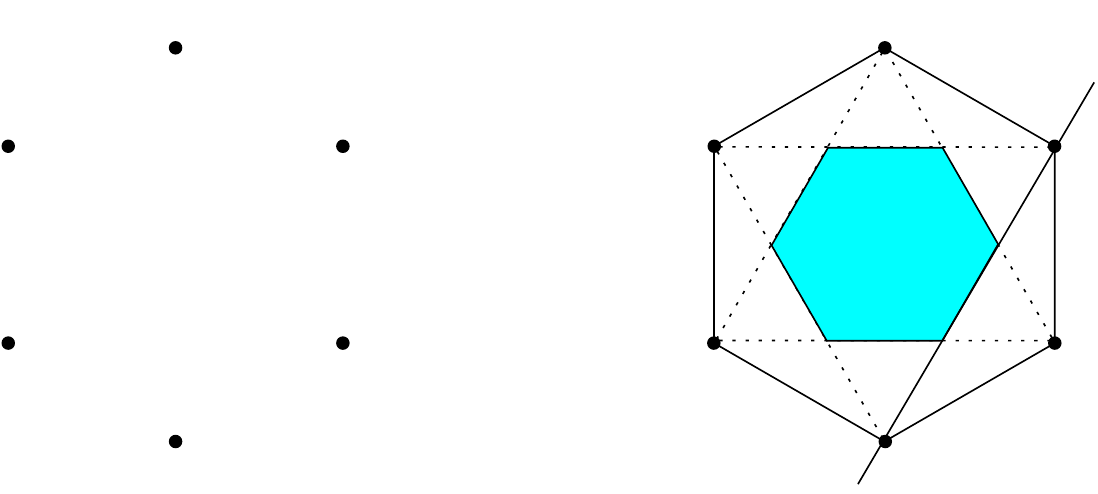}%
\end{picture}%
\setlength{\unitlength}{4144sp}%
\begingroup\makeatletter\ifx\SetFigFont\undefined%
\gdef\SetFigFont#1#2#3#4#5{%
  \reset@font\fontsize{#1}{#2pt}%
  \fontfamily{#3}\fontseries{#4}\fontshape{#5}%
  \selectfont}%
\fi\endgroup%
\begin{picture}(5015,2247)(323,-1422)
\put(1891,-916){\makebox(0,0)[lb]{\smash{{\SetFigFont{10}{12.0}{\rmdefault}{\mddefault}{\updefault}{\color[rgb]{0,0,0}$\bar v_1 = (1,0,0)$}%
}}}}
\put(1135,656){\makebox(0,0)[lb]{\smash{{\SetFigFont{10}{12.0}{\rmdefault}{\mddefault}{\updefault}{\color[rgb]{0,0,0}$\bar v_2 = (0,1,0)$}%
}}}}
\put(379,204){\makebox(0,0)[lb]{\smash{{\SetFigFont{10}{12.0}{\rmdefault}{\mddefault}{\updefault}{\color[rgb]{0,0,0}$\bar v_4 = (-\frac13, \frac23, \frac23)$}%
}}}}
\put(1104,-1358){\makebox(0,0)[lb]{\smash{{\SetFigFont{10}{12.0}{\rmdefault}{\mddefault}{\updefault}{\color[rgb]{0,0,0}$\bar v_5 = (\frac23, -\frac13, \frac23)$}%
}}}}
\put(338,-919){\makebox(0,0)[lb]{\smash{{\SetFigFont{10}{12.0}{\rmdefault}{\mddefault}{\updefault}{\color[rgb]{0,0,0}$\bar v_3 = (0,0,1)$}%
}}}}
\put(5151,-917){\makebox(0,0)[lb]{\smash{{\SetFigFont{10}{12.0}{\rmdefault}{\mddefault}{\updefault}{\color[rgb]{0,0,0}$\bar v_1$}%
}}}}
\put(4391,678){\makebox(0,0)[lb]{\smash{{\SetFigFont{10}{12.0}{\rmdefault}{\mddefault}{\updefault}{\color[rgb]{0,0,0}$\bar v_2$}%
}}}}
\put(4400,-1307){\makebox(0,0)[lb]{\smash{{\SetFigFont{10}{12.0}{\rmdefault}{\mddefault}{\updefault}{\color[rgb]{0,0,0}$\bar v_5$}%
}}}}
\put(5199,100){\makebox(0,0)[lb]{\smash{{\SetFigFont{10}{12.0}{\rmdefault}{\mddefault}{\updefault}{\color[rgb]{0,0,0}$\bar v_6$}%
}}}}
\put(3411,-890){\makebox(0,0)[lb]{\smash{{\SetFigFont{10}{12.0}{\rmdefault}{\mddefault}{\updefault}{\color[rgb]{0,0,0}$\bar v_3$}%
}}}}
\put(1847,213){\makebox(0,0)[lb]{\smash{{\SetFigFont{10}{12.0}{\rmdefault}{\mddefault}{\updefault}{\color[rgb]{0,0,0}$\bar v_6 = (\frac23, \frac23, -\frac13)$}%
}}}}
\put(3373,190){\makebox(0,0)[lb]{\smash{{\SetFigFont{10}{12.0}{\rmdefault}{\mddefault}{\updefault}{\color[rgb]{0,0,0}$\bar v_4$}%
}}}}
\end{picture}%
\caption{The Gale diagram of the face normals of a triangular bipyramid (left); the compatibility and the irredundancy domains (right).}
\label{fig:GaleCoIr}
\end{figure}
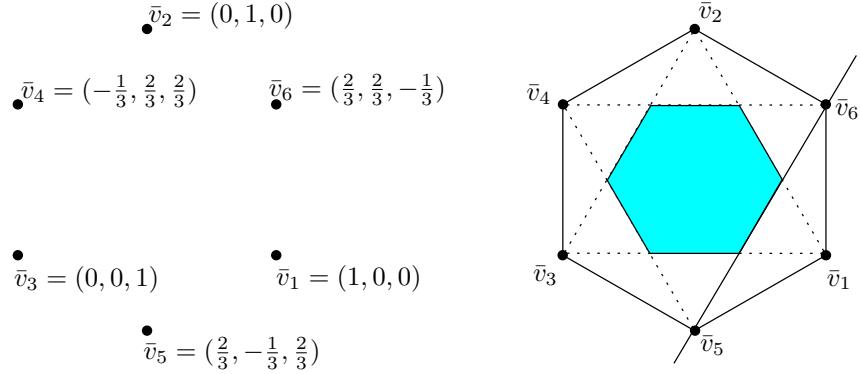

Figure \ref{fig:GaleCoIr}, right, shows the quotients $\co(V)$ and $\ir(V)$ of the compatibility domain and of the irredundancy domain. According to \eqref{eqn:CoV2}, $\co(V)$ is the positive hull of $\bar V$, which in the affine Gale diagram becomes the convex hull. According to \eqref{eqn:IrV2}, $\ir(V)$ is the 2-core of $\bar V$, which is shown as a shaded hexagon on Figure \ref{fig:GaleCoIr}.

It is also possible to interpret Figure \ref{fig:GaleCoIr}, right, in terms of positive and hyperbolic circuits, see equations \eqref{eqn:CoV3} and \eqref{eqn:IrV3}.
The six positive circuits of the vector configuration $V$ are obtained from
$$
v_1 + 2v_2 + v_4 + 2v_6 = 0
$$
by the action of the dihedral group. This particular circuit leads according to \eqref{eqn:CoV2} to the inequality $h_1 + 2h_2 \ge 0$ in the $h_4 = h_5 = h_6 = 0$ space, which determines the half-space containing vectors $\bar v_3$ and $\bar v_5$ in its boundary. The other five edges of the big hexagon on Figure \ref{fig:GaleCoIr} correspond to the other five positive circuits.

The lines bounding the irredundancy domain correspond to hyperbolic circuits (in this example, the inequalities in \eqref{eqn:IrV3} generated by the positive circuits turn out to be redundant). For example, the line highlighted on Figure \ref{fig:GaleCoIr} corresponds to the circuit
\begin{equation}\label{eq:rel bipyr v1}
v_1 = 2v_2 + 2v_3 + 3v_4,
\end{equation}
The principle ``evaluations on $\bar V$ correspond to dependencies in $V$'' (see \eqref{eqn:ValDep}) allows to read off the signature of the circuit from the position of the line. Since the line separates $\bar v_1$ from $\bar v_2$, $\bar v_3$, and $\bar v_4$, the coefficient at $v_1$ has the sign opposite to those at $\bar v_2$, $\bar v_3$, and $\bar v_4$; the points lying on the line correspond to zero coefficients.

In order to obtain the chamber fan of the vector configuration $\bar V$, one has to draw the diagonals $\bar v_1 \bar v_4$, $\bar v_2 \bar v_5$, and $\bar v_3 \bar v_6$, in addition to those drawn already on Figure \ref{fig:GaleCoIr}. The chambers in the interior of the 2-core are the type cones of $V$. Figure \ref{fig:TypeConesExl} shows the subdivision of $\clir(V)$ into chambers and describes the faces of one of the full-dimensional type cones.

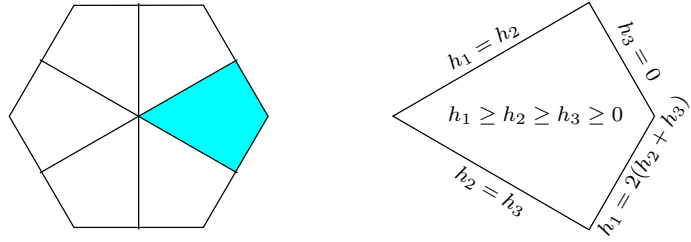
\begin{figure}[ht]
\centering
\begin{picture}(0,0)%
\includegraphics{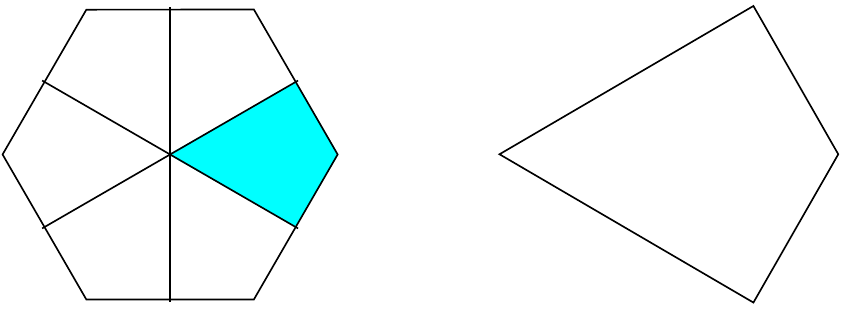}%
\end{picture}%
\setlength{\unitlength}{4144sp}%
\begingroup\makeatletter\ifx\SetFigFont\undefined%
\gdef\SetFigFont#1#2#3#4#5{%
  \reset@font\fontsize{#1}{#2pt}%
  \fontfamily{#3}\fontseries{#4}\fontshape{#5}%
  \selectfont}%
\fi\endgroup%
\begin{picture}(3845,1445)(123,-815)
\put(2759,216){\rotatebox{30.0}{\makebox(0,0)[lb]{\smash{{\SetFigFont{8}{9.6}{\rmdefault}{\mddefault}{\updefault}{\color[rgb]{0,0,0}$h_1 = h_2$}%
}}}}}
\put(2766,-423){\rotatebox{330.0}{\makebox(0,0)[lb]{\smash{{\SetFigFont{8}{9.6}{\rmdefault}{\mddefault}{\updefault}{\color[rgb]{0,0,0}$h_2=h_3$}%
}}}}}
\put(3683,-780){\rotatebox{60.0}{\makebox(0,0)[lb]{\smash{{\SetFigFont{8}{9.6}{\rmdefault}{\mddefault}{\updefault}{\color[rgb]{0,0,0}$h_1 = 2(h_2 + h_3)$}%
}}}}}
\put(3735,478){\rotatebox{300.0}{\makebox(0,0)[lb]{\smash{{\SetFigFont{8}{9.6}{\rmdefault}{\mddefault}{\updefault}{\color[rgb]{0,0,0}$h_3 = 0$}%
}}}}}
\put(2734,-91){\makebox(0,0)[lb]{\smash{{\SetFigFont{8}{9.6}{\rmdefault}{\mddefault}{\updefault}{\color[rgb]{0,0,0}$h_1 \ge h_2 \ge h_3 \ge 0$}%
}}}}
\end{picture}%
\caption{A type cone.}
\label{fig:TypeConesExl}
\end{figure}

For any point $h \in \clir(V)$ (recall that we identified $\R^3$ with a subspace of $\R^6$ by putting $h_4 = h_5 = h_6$), the combinatorics of the corresponding polytope can be read off from the diagram. By Lemma \ref{lem:CombGale}, the normal fan of $P(h)$ contains the cone $\pos(V_I)$ (equivalently, facets with normals $\{v_i \mid i \in I\}$ intersect along a face) if and only if $h$ lies in the relative interior of the positive hull of $\bar V_{[6]\setminus I}$. For example, since the type cone highlighted on Figure \ref{fig:TypeConesExl} lies in $\relint\pos\{\bar v_1, \bar v_2, \bar v_5\}$, the corresponding fan $\Delta$ contains the cone spanned by $v_3, v_4, v_6$. The normal fan corresponding to this type cone is shown in the center of the Figure \ref{fig:NormalFansExl}. This normal fan is simplicial, since the type cone is full-dimensional (compare Lemma \ref{lem:AffT}).

The other fans on Figure~\ref{fig:NormalFansExl} are associated with the faces of the depicted type cone. Note that the fans corresponding to boundary points of $\clir(V)$ aren't using all of the vertices $v_i$, compare Lemma \ref{lem:BdryIr}.

\begin{figure}[ht]
\centering
\begin{picture}(0,0)%
\includegraphics{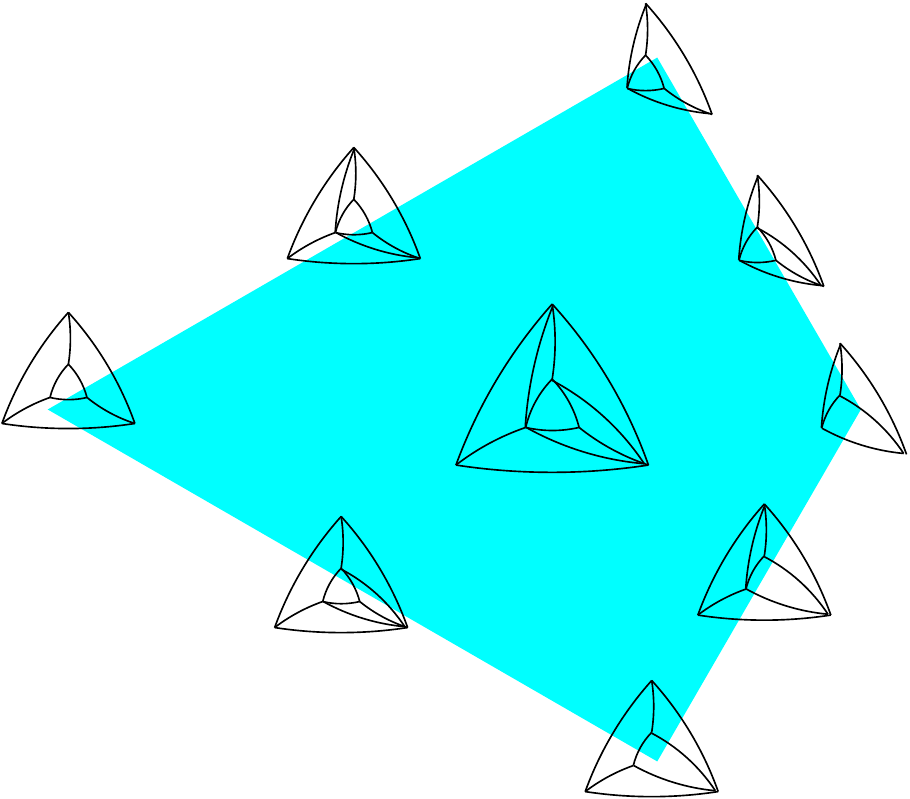}%
\end{picture}%
\setlength{\unitlength}{4144sp}%
\begingroup\makeatletter\ifx\SetFigFont\undefined%
\gdef\SetFigFont#1#2#3#4#5{%
  \reset@font\fontsize{#1}{#2pt}%
  \fontfamily{#3}\fontseries{#4}\fontshape{#5}%
  \selectfont}%
\fi\endgroup%
\begin{picture}(4152,3644)(-954,-2290)
\put(1688,-594){\makebox(0,0)[lb]{\smash{{\SetFigFont{8}{9.6}{\rmdefault}{\mddefault}{\updefault}{\color[rgb]{0,0,0}$1$}%
}}}}
\put(1595,-367){\makebox(0,0)[lb]{\smash{{\SetFigFont{8}{9.6}{\rmdefault}{\mddefault}{\updefault}{\color[rgb]{0,0,0}$2$}%
}}}}
\put(2033,-807){\makebox(0,0)[lb]{\smash{{\SetFigFont{8}{9.6}{\rmdefault}{\mddefault}{\updefault}{\color[rgb]{0,0,0}$4$}%
}}}}
\put(1529, -8){\makebox(0,0)[lb]{\smash{{\SetFigFont{8}{9.6}{\rmdefault}{\mddefault}{\updefault}{\color[rgb]{0,0,0}$5$}%
}}}}
\put(1024,-847){\makebox(0,0)[lb]{\smash{{\SetFigFont{8}{9.6}{\rmdefault}{\mddefault}{\updefault}{\color[rgb]{0,0,0}$6$}%
}}}}
\put(1390,-697){\makebox(0,0)[lb]{\smash{{\SetFigFont{8}{9.6}{\rmdefault}{\mddefault}{\updefault}{\color[rgb]{0,0,0}$3$}%
}}}}
\end{picture}%
\caption{The fans corresponding to the faces of a type cone.}
\label{fig:NormalFansExl}
\end{figure}
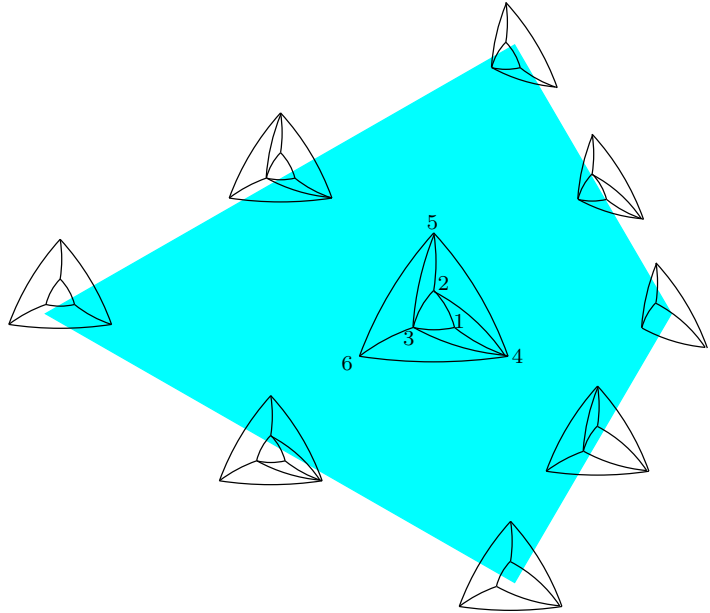

Crossing from one fully-dimensional type cone to an adjacent one
corresponds to a ``flip'', see Figure~\ref{fig:flip}. The edge $F_{35}$ becomes replaced by the edge $F_{26}$, compare the description of the faces of the type cones through vanishing edge lengths, Lemma \ref{lem:TypeConeIneq}. From a different point of view, boundaries between full-dimensional type cones correspond to (non-positive and non-hyperbolic) circuits of the vector configuration $V$. A circuit of this form corresponds to contracting an edge of the polytope, see equation \eqref{eqn:Circuit}. There are three such circuits:
$$
v_1 + v_5 = v_2 + v_4, \quad v_2 + v_6 = v_3 + v_5, \quad v_3 + v_4 = v_1 + v_6,
$$
and they correspond to the hyperplanes
$$
h_1 + h_5 = h_2 + h_4, \quad h_2 + h_6 = h_3 + h_5, \quad h_3 + h_4 = h_1 + h_6,
$$
or, in our picture to the lines
$$
h_1 = h_2, \quad h_2 = h_3, \quad h_3 = h_1
$$

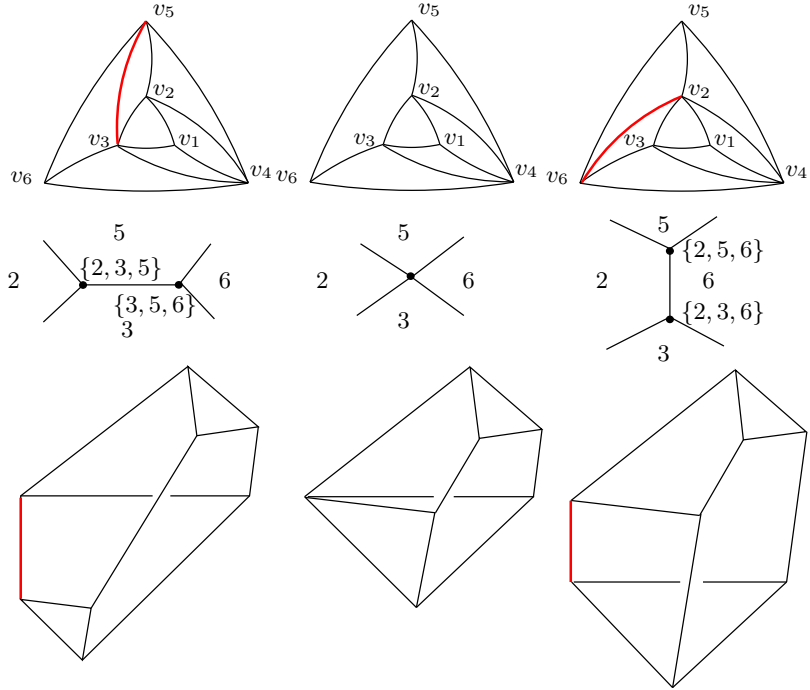
\begin{figure}[ht]
\centering
\begin{picture}(0,0)%
\includegraphics{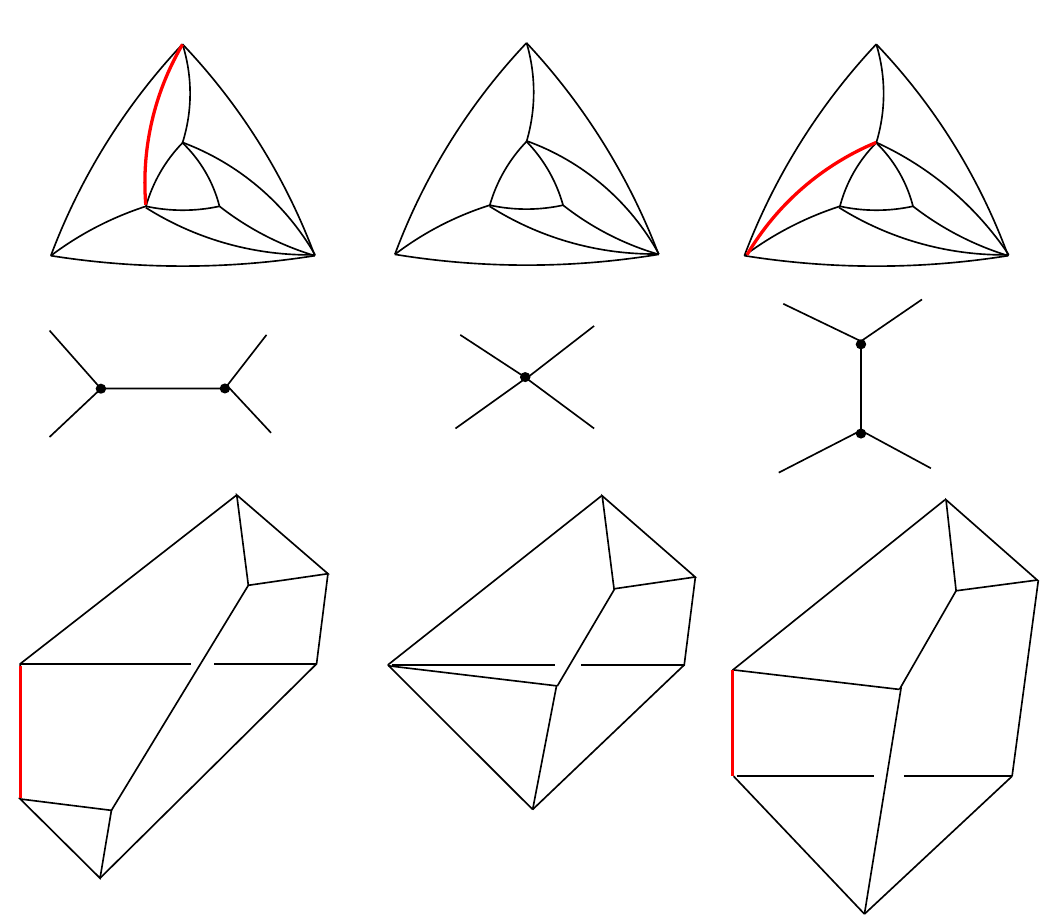}%
\end{picture}%
\setlength{\unitlength}{4144sp}%
\begingroup\makeatletter\ifx\SetFigFont\undefined%
\gdef\SetFigFont#1#2#3#4#5{%
  \reset@font\fontsize{#1}{#2pt}%
  \fontfamily{#3}\fontseries{#4}\fontshape{#5}%
  \selectfont}%
\fi\endgroup%
\begin{picture}(4759,4179)(7074,-3268)
\put(7948,279){\makebox(0,0)[lb]{\smash{{\SetFigFont{9}{10.8}{\rmdefault}{\mddefault}{\updefault}{\color[rgb]{0,0,0}$v_2$}%
}}}}
\put(7565,-16){\makebox(0,0)[lb]{\smash{{\SetFigFont{9}{10.8}{\rmdefault}{\mddefault}{\updefault}{\color[rgb]{0,0,0}$v_3$}%
}}}}
\put(8522,-201){\makebox(0,0)[lb]{\smash{{\SetFigFont{9}{10.8}{\rmdefault}{\mddefault}{\updefault}{\color[rgb]{0,0,0}$v_4$}%
}}}}
\put(7948,759){\makebox(0,0)[lb]{\smash{{\SetFigFont{9}{10.8}{\rmdefault}{\mddefault}{\updefault}{\color[rgb]{0,0,0}$v_5$}%
}}}}
\put(7105,-238){\makebox(0,0)[lb]{\smash{{\SetFigFont{9}{10.8}{\rmdefault}{\mddefault}{\updefault}{\color[rgb]{0,0,0}$v_6$}%
}}}}
\put(8101,-16){\makebox(0,0)[lb]{\smash{{\SetFigFont{9}{10.8}{\rmdefault}{\mddefault}{\updefault}{\color[rgb]{0,0,0}$v_1$}%
}}}}
\put(11119,279){\makebox(0,0)[lb]{\smash{{\SetFigFont{9}{10.8}{\rmdefault}{\mddefault}{\updefault}{\color[rgb]{0,0,0}$v_2$}%
}}}}
\put(10736,-16){\makebox(0,0)[lb]{\smash{{\SetFigFont{9}{10.8}{\rmdefault}{\mddefault}{\updefault}{\color[rgb]{0,0,0}$v_3$}%
}}}}
\put(11694,-201){\makebox(0,0)[lb]{\smash{{\SetFigFont{9}{10.8}{\rmdefault}{\mddefault}{\updefault}{\color[rgb]{0,0,0}$v_4$}%
}}}}
\put(11119,759){\makebox(0,0)[lb]{\smash{{\SetFigFont{9}{10.8}{\rmdefault}{\mddefault}{\updefault}{\color[rgb]{0,0,0}$v_5$}%
}}}}
\put(10277,-238){\makebox(0,0)[lb]{\smash{{\SetFigFont{9}{10.8}{\rmdefault}{\mddefault}{\updefault}{\color[rgb]{0,0,0}$v_6$}%
}}}}
\put(11272,-16){\makebox(0,0)[lb]{\smash{{\SetFigFont{9}{10.8}{\rmdefault}{\mddefault}{\updefault}{\color[rgb]{0,0,0}$v_1$}%
}}}}
\put(9519,284){\makebox(0,0)[lb]{\smash{{\SetFigFont{9}{10.8}{\rmdefault}{\mddefault}{\updefault}{\color[rgb]{0,0,0}$v_2$}%
}}}}
\put(9136,-10){\makebox(0,0)[lb]{\smash{{\SetFigFont{9}{10.8}{\rmdefault}{\mddefault}{\updefault}{\color[rgb]{0,0,0}$v_3$}%
}}}}
\put(10094,-195){\makebox(0,0)[lb]{\smash{{\SetFigFont{9}{10.8}{\rmdefault}{\mddefault}{\updefault}{\color[rgb]{0,0,0}$v_4$}%
}}}}
\put(9519,764){\makebox(0,0)[lb]{\smash{{\SetFigFont{9}{10.8}{\rmdefault}{\mddefault}{\updefault}{\color[rgb]{0,0,0}$v_5$}%
}}}}
\put(8677,-232){\makebox(0,0)[lb]{\smash{{\SetFigFont{9}{10.8}{\rmdefault}{\mddefault}{\updefault}{\color[rgb]{0,0,0}$v_6$}%
}}}}
\put(9673,-10){\makebox(0,0)[lb]{\smash{{\SetFigFont{9}{10.8}{\rmdefault}{\mddefault}{\updefault}{\color[rgb]{0,0,0}$v_1$}%
}}}}
\put(7511,-792){\makebox(0,0)[lb]{\smash{{\SetFigFont{9}{10.8}{\rmdefault}{\mddefault}{\updefault}{\color[rgb]{0,0,0}$\{2,3,5\}$}%
}}}}
\put(10571,-873){\makebox(0,0)[lb]{\smash{{\SetFigFont{9}{10.8}{\rmdefault}{\mddefault}{\updefault}{\color[rgb]{0,0,0}$2$}%
}}}}
\put(11204,-873){\makebox(0,0)[lb]{\smash{{\SetFigFont{9}{10.8}{\rmdefault}{\mddefault}{\updefault}{\color[rgb]{0,0,0}$6$}%
}}}}
\put(11077,-1056){\makebox(0,0)[lb]{\smash{{\SetFigFont{9}{10.8}{\rmdefault}{\mddefault}{\updefault}{\color[rgb]{0,0,0}$\{2,3,6\}$}%
}}}}
\put(11077,-689){\makebox(0,0)[lb]{\smash{{\SetFigFont{9}{10.8}{\rmdefault}{\mddefault}{\updefault}{\color[rgb]{0,0,0}$\{2,5,6\}$}%
}}}}
\put(7089,-873){\makebox(0,0)[lb]{\smash{{\SetFigFont{9}{10.8}{\rmdefault}{\mddefault}{\updefault}{\color[rgb]{0,0,0}$2$}%
}}}}
\put(8339,-873){\makebox(0,0)[lb]{\smash{{\SetFigFont{9}{10.8}{\rmdefault}{\mddefault}{\updefault}{\color[rgb]{0,0,0}$6$}%
}}}}
\put(7762,-1162){\makebox(0,0)[lb]{\smash{{\SetFigFont{9}{10.8}{\rmdefault}{\mddefault}{\updefault}{\color[rgb]{0,0,0}$3$}%
}}}}
\put(7714,-1017){\makebox(0,0)[lb]{\smash{{\SetFigFont{9}{10.8}{\rmdefault}{\mddefault}{\updefault}{\color[rgb]{0,0,0}$\{3,5,6\}$}%
}}}}
\put(9397,-585){\makebox(0,0)[lb]{\smash{{\SetFigFont{9}{10.8}{\rmdefault}{\mddefault}{\updefault}{\color[rgb]{0,0,0}$5$}%
}}}}
\put(7714,-585){\makebox(0,0)[lb]{\smash{{\SetFigFont{9}{10.8}{\rmdefault}{\mddefault}{\updefault}{\color[rgb]{0,0,0}$5$}%
}}}}
\put(9397,-1113){\makebox(0,0)[lb]{\smash{{\SetFigFont{9}{10.8}{\rmdefault}{\mddefault}{\updefault}{\color[rgb]{0,0,0}$3$}%
}}}}
\put(9782,-873){\makebox(0,0)[lb]{\smash{{\SetFigFont{9}{10.8}{\rmdefault}{\mddefault}{\updefault}{\color[rgb]{0,0,0}$6$}%
}}}}
\put(8916,-873){\makebox(0,0)[lb]{\smash{{\SetFigFont{9}{10.8}{\rmdefault}{\mddefault}{\updefault}{\color[rgb]{0,0,0}$2$}%
}}}}
\put(10936,-537){\makebox(0,0)[lb]{\smash{{\SetFigFont{9}{10.8}{\rmdefault}{\mddefault}{\updefault}{\color[rgb]{0,0,0}$5$}%
}}}}
\put(10936,-1306){\makebox(0,0)[lb]{\smash{{\SetFigFont{9}{10.8}{\rmdefault}{\mddefault}{\updefault}{\color[rgb]{0,0,0}$3$}%
}}}}
\end{picture}%
\caption{The flip $\left\{ \{2,3,5\}, \{3,5,6\} \right\} 
\rightsquigarrow \left\{ \{2,3,6\}, \{2,5,6\} \right\}$.}
\label{fig:flip}
\end{figure}

The fan on Figure \ref{fig:notfan}, left, is not polytopal. Indeed, since it contains $\pos\{v_1, v_6\}$, $\pos\{v_2, v_4\}$, and $\pos\{v_3, v_5\}$, the corresponding point $\pi(h)$ must lie in the intersection
$$
\relint\pos\{\bar v_2, \bar v_3, \bar v_4, \bar v_5\} \cap \relint\pos\{\bar v_1, \bar v_3, \bar v_5, \bar v_6\} \cap \relint\pos\{\bar v_1, \bar v_2, \bar v_4, \bar v_6\},
$$
which is empty. Similarly, the other two fans on Figure \ref{fig:notfan} are also non-polytopal. (One can also refer to Figure \ref{fig:NormalFansExl}, where all, up to symmetry, polytopal fans with the 1-skeleton in $V$ are shown.)

\begin{figure}[ht]
\centering
\includegraphics{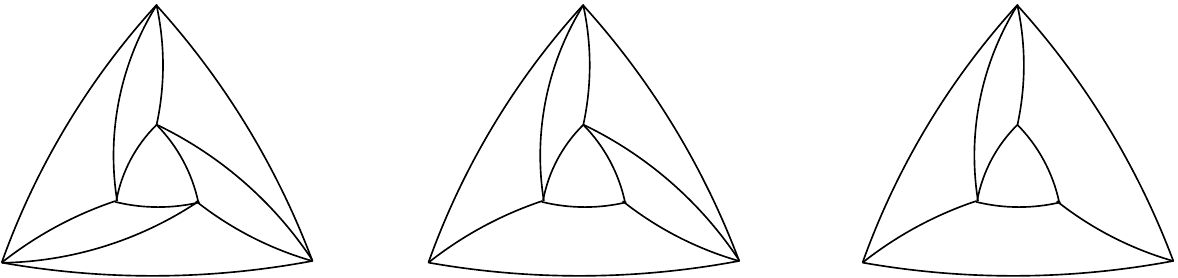}
\caption{Examples of non-polytopal fans.}
\label{fig:notfan}
\end{figure}

\section{Mixed volumes}
\label{sec:MixVol}
\subsection{The examples, continued}
\label{sec:ExlMV}
We continue the examples considered in Section \ref{sec:Example}. Motivated by the observation of Bavard and Ghys \cite{BG92} that the area of convex polygons as a function of their support numbers $h_i$ is a quadratic form of hyperbolic signature, we are looking at the surface area of 3-dimensional polytopes.

\subsubsection{Parallelepipeds with fixed face directions}
\label{sec:CubeShapes}
The surface area of the parallelepiped $\{x \mid 0 \le \langle x, v_i \rangle \le h_i, i = 1,2,3\}$ equals
$$
\area(h) = \frac{2}{D}(h_1h_2 + h_2h_3 + h_3h_1),
$$
where $D = |\det(v_1,v_2,v_3)|$. This is a quadratic form of signature $(+,-,-)$.

\subsubsection{Polygons}

The area of a polygon with support vector $h \in \wT(\alpha) \subset \R^n$ equals
$$\a(h)=\frac{1}{2}\sum_i h_i \ell_i(h), $$
where the edge length $\ell_i(h)$ is a linear function of $h$. Thus $\a(h)$ is a quadratic form.
The associated symmetric bilinear form is 
\begin{equation}\label{eq:ar polygon}
 \a(h,k)=\frac{1}{2}\sum_i h_i \ell_i(k) = \frac{1}{2}\sum_i k_i \ell_i(h),
\end{equation}
due to $\frac{\partial \a(h)}{\partial h_i} = \ell_i(h)$ that follows from a simple geometric argument.

The quadratic form $\a(h)$ has signature $(+,0,0,-,\ldots,-)$. This can be proved by induction on
the number of edges of the polygon as in \cite{BG92}, or using the Minkowski
inequality \cite[p.~321]{Scn93} in the plane. 
See also \cite[Lemma 2.14]{IzmHED}.

\subsubsection{Polygonal prisms}
Denote by $\a^+(h^+)$ the surface area of a prism $(h, h_{n+1}, h_{n+2})$ over the $n$-gon. Here $v_{n+1}$ and $v_{n+2}$ are as in Section \ref{sec:PolPrisms1}.
Then we have
\begin{multline*}
\a^+(h^+) = 2 \a(h) + (h_{n+1} + h_{n+2}) \per(h)\\
= 2 \a(h) + 2 (h_{n+1} + h_{n+2}) \a(\one, h)
\end{multline*}
with the associated symmetric bilinear form
\begin{multline*}
\a^+(h^+, k^+) = 2\a(h,k)\\
+(h_{n+1}+h_{n+2})\a(\one, k)
+ (k_{n+1}+k_{n+2})\a(\one, h)
\end{multline*}


The restriction of $\a^+$ to $\R^n$ is $2\a$ that by the above results has signature $(+,0,0,-,\ldots,-)$.
The vector $(\zero,-1,1)$ belongs to the kernel of $\a^+$ (it corresponds to vertical translation). The vector $(\one,-1,-1)$ is orthogonal to $\R^n\subset \R^{n+2}$ with respect to $\a^+$, and is negative.
This implies that  
$\a^+$ has signature $(+,0,0,0,-,\ldots,-)$.

\subsubsection{Triangular bipyramid}
\label{sec:TriBiForm}
Let $v_1, \ldots, v_6 \in \R^3$ be as in \ref{sub:bipyr1}. The combinatorics of a polyhedron $P(h)$ with face normals $(v_i)$ depends on $h$. Let us compute $\a(\partial P(h))$ for the type cone $\Delta$ shaded on Figure \ref{fig:TypeConesExl}.

The normal fan of $P(h)$ (the triangulation in the center of Figure \ref{fig:NormalFansExl}) shows that $P(h)$ is a doubly truncated tetrahedron, see Figure \ref{fig:TruncTetr}. We have
$$
P(h) = \overline{(\Sigma_1 \setminus \Sigma_2) \setminus \Sigma_3},
$$
where
\begin{equation*}
\begin{aligned}
&\Sigma_1 := \{x \in \R^3 \mid \langle v_i, x \rangle \le h_i, i \in \{2, 3, 4, 5\}\}\\
&\Sigma_2 := \{x \in \R^3 \mid \langle v_i, x \rangle \le h_i, i \in \{2, 3, 4\},\, \langle v_1, x \rangle \ge h_1\}\\
&\Sigma_3 := \{x \in \R^3 \mid \langle v_i, x \rangle \le h_i, i \in \{3, 4, 5\},\, \langle v_6, x \rangle \ge h_6\}
\end{aligned}
\end{equation*}
The surface area of a tetrahedron with fixed face normals is proportional to the squared length of any of its edges, and the edge length is a linear function of the support vector. Thus we have
$$
\area(\partial\Sigma_1) = f_1^2(h),
$$
for some linear function $f_1$. We have $f_1(h) = 0$ if and only if the hyperplanes $H_i$, $i \in \{2, 3, 4, 5\}$ pass through a common point, that is iff the system $\langle v_i, x \rangle = h_i$ has a solution. Since $v_2 + 2v_3 + 2v_4 + v_5 = 0$, a solution exists if and only if $h_2 + 2h_3 + 2h_4 + h_5 = 0$. By restricting to $h_4 = h_5 = h_6 = 0$ as we have done in Section \ref{sub:bipyr1}, we obtain
$$
\area(\partial\Sigma_1) = c_1 (h_2 + 2h_3)^2
$$
for some $c_1 > 0$.

\begin{figure}[ht]
\centering
\begin{picture}(0,0)%
\includegraphics{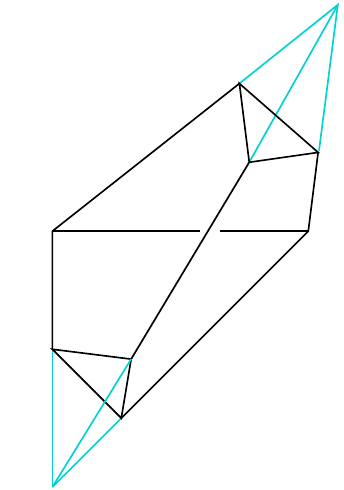}%
\end{picture}%
\setlength{\unitlength}{4144sp}%
\begingroup\makeatletter\ifx\SetFigFont\undefined%
\gdef\SetFigFont#1#2#3#4#5{%
  \reset@font\fontsize{#1}{#2pt}%
  \fontfamily{#3}\fontseries{#4}\fontshape{#5}%
  \selectfont}%
\fi\endgroup%
\begin{picture}(1560,2229)(661,-1873)
\put(1075,-1379){\makebox(0,0)[lb]{\smash{{\SetFigFont{8}{9.6}{\rmdefault}{\mddefault}{\updefault}{\color[rgb]{0,0,0}$F_6$}%
}}}}
\put(1811,-288){\makebox(0,0)[lb]{\smash{{\SetFigFont{8}{9.6}{\rmdefault}{\mddefault}{\updefault}{\color[rgb]{0,0,0}$F_1$}%
}}}}
\put(676,-1591){\makebox(0,0)[lb]{\smash{{\SetFigFont{10}{12.0}{\rmdefault}{\mddefault}{\updefault}{\color[rgb]{0,0,0}$\Sigma_3$}%
}}}}
\put(2206,-61){\makebox(0,0)[lb]{\smash{{\SetFigFont{10}{12.0}{\rmdefault}{\mddefault}{\updefault}{\color[rgb]{0,0,0}$\Sigma_2$}%
}}}}
\put(1591,-934){\makebox(0,0)[lb]{\smash{{\SetFigFont{8}{9.6}{\rmdefault}{\mddefault}{\updefault}{\color[rgb]{0,0,0}$F_4$}%
}}}}
\put(1351,-511){\makebox(0,0)[lb]{\smash{{\SetFigFont{8}{9.6}{\rmdefault}{\mddefault}{\updefault}{\color[rgb]{0,0,0}$F_3$}%
}}}}
\end{picture}%
\caption{Representing $P(h)$ as a truncated tetrahedron.}
\label{fig:TruncTetr}
\end{figure}
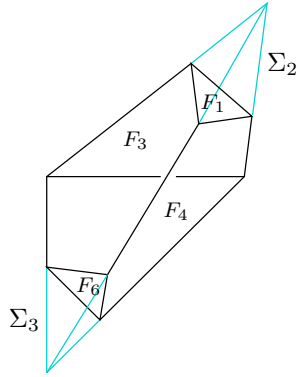

Next, observe that
\begin{multline*}
\area(\partial\overline{\Sigma_1 \setminus \Sigma_2}) = \area(\partial\Sigma_1)
- \area(F_2(\Sigma_2)) - \area(F_3(\Sigma_2)) \\ - \area(F_4(\Sigma_2)) + \area(F_1(\Sigma_2))
\end{multline*}
All quantities $\area(F_i(\Sigma_2))$ are proportional to the square of a linear function $f_2(h)$ that vanishes when the tetrahedron $\Sigma_2$ degenerates. Similarly to the previous paragraph, using \eqref{eq:rel bipyr v1}, we find $f_2(h) = h_1 - 2h_2 - 2h_3 - 3h_4$, and hence
$$
\area(\partial\overline{\Sigma_1 \setminus \Sigma_2}) = c_1 (h_2 + 2h_3)^2 - c_2 (h_1 - 2h_2 - 2h_3)^2
$$
Here $c_2 > 0$ because the sum of areas of three faces of a tetrahedron is bigger than the area of its fourth face.

Finally, cutting off the tetrahedron $\Sigma_3$ yields
$$
\area(\partial P(h)) = q_\Delta(h) := c_1 f_1^2 - c_2 f_2^2 - c_3 f_3^2, \quad c_1, c_2, c_3 > 0,
$$
where
$$
f_1(h) = h_2 + 2h_3, \quad f_2(h) = h_1 - 2h_2 - 2h_3, \quad f_3(h) = h_3.
$$
Since $f_1$, $f_2$, $f_3$ are linearly independent, quadratic form $q_\Delta(h)$ has signature $(+,-,-)$.

\subsection{Mixed volumes and quadratic forms}

\subsubsection{Definition and basic properties of mixed volumes}
\label{sec:MixVolDef}
Minkowski \cite{Min03} has shown that the volume behaves polylinearly with respect to the Minkowski addition and positive scaling. Namely, for any compact convex bodies $K_1, \ldots, K_m \subset \R^d$ there exist real numbers $c_{i_1\ldots i_d}$, $1 \le i_\alpha \le m$ such that
\begin{equation}
\label{eqn:MixVol}
\vol(\lambda_1 K_1 + \cdots + \lambda_m K_m) = \sum_{i_\alpha \in [m]} c_{i_1\ldots i_d} \lambda_{i_1} \cdots \lambda_{i_d}
\end{equation}
holds for all $\lambda_1, \ldots, \lambda_m \ge 0$.
The coefficients $c_{i_1\ldots i_d}$ are uniquely determined by the bodies $K_{i_1}, \ldots, K_{i_d}$ if we require that they are symmetric with respect to permutations of indices: $c_{\phi \circ I} = c_I$ for all $\phi \in S_m$.

\begin{dfn}
The coefficient $c_{i_1\ldots i_d}$ in \eqref{eqn:MixVol} is called a \emph{mixed volume} and denoted by $\vol(K_{i_1}, \ldots, K_{i_d})$.
\end{dfn}

Clearly, $\vol(K, \ldots, K) = K$. For more details on mixed volumes, see \cite[Chapter 5]{Scn93} and \cite[Chapter IV]{Ewald96}.

\begin{exl}
\label{exl:Quermass}
A special case of \eqref{eqn:MixVol} is the Steiner formula
$$
\vol(K + \rho B) = \sum_{i=0}^d \binom{d}{i} \rho^i W_i(K)
$$
where $B$ is the unit ball.
The coefficients $W_i(K) = \vol(\underbrace{K,\ldots,K}_{d-i}, \underbrace{B,\ldots,B}_i)$ are called \emph{quermassintegrals} of $K$. We have $W_0(K) = \vol(K)$, $W_1(K) = \frac{1}{d} \area(\partial K)$, $W_d(K) = \vol(B)$. For a polytope $P$ we have
$$
W_i(P) = c_{d,i} \sum_{F \in \cF^{d-i}(P)} \vol_{d-i}(F) \cdot |N_F(P)|
$$
for some constant $c_{d,i}$ independent of $P$, where the sum ranges over all $(d-i)$-faces of $P$ and $|N_F(P)|$ denotes the angular measure of the normal cone $N_F(P) \subset \R^d$.
\end{exl}

The following properties of mixed volumes will be needed in the sequel.

\begin{itemize}
\item
Mixed volume is multilinear with respect to the Minkowski addition:
$$
\vol(\lambda K + \mu L, \cC) = \lambda \vol(K, \cC) + \mu \vol(L, \cC) \quad \text{for } \lambda, \mu \ge 0,
$$
where $\cC = (K_1, \ldots, K_{d-1})$.
\item
Mixed volume is monotone under inclusion: $\vol(K, \cC) \ge \vol(L, \cC)$ if $K \supset L$. In particular,
\begin{equation}
\label{eqn:MixVolPos}
\vol(K_1, \ldots, K_d) \ge 0
\end{equation}
More precisely, we have the following \cite[Theorem 5.1.7]{Scn93}.
\end{itemize}

\begin{thm}\label{thm: segment}
The inequality in \eqref{eqn:MixVolPos} is strict if and only if there are segments $s_i\subset K_i$, $i=1,\ldots,d$ with linearly independent directions.
\end{thm}

In particular, the inequality \eqref{eqn:MixVolPos} is strict if $\dim K_i = d$ for all $i$.

\subsubsection{Alexandrov-Fenchel inequalities and signatures of quadratic forms}
\label{sec:AFSign}
Fix convex bodies $K_1, \ldots, K_{d-2} \subset \R^d$ and denote $\cC := (K_1, \ldots, K_{d-2})$.
Then the function
\begin{equation}
\label{eqn:fC}
\vol_{\cC} \colon K \mapsto \vol(K, K, \cC)
\end{equation}
on the set of convex bodies in $\R^d$ possesses the valuation property:
\begin{equation}
\label{eqn:ValProp}
\vol_{\cC}(K \cup L) + \vol_{\cC}(K \cap L) = \vol_{\cC}(K) + \vol_{\cC}(L),
\end{equation}
provided that $K \cup L$ is convex.
Besides, it is homogeneous of degree $2$:
\begin{equation}
\label{eqn:Hom2}
\vol_{\cC}(\lambda K) = \lambda^2 \vol_{\cC}(K)
\end{equation}
This follows quite easily from the corresponding properties of the volume, see \cite{McMScn83}.

Let $\Delta$ be a complete polytopal fan in $\R^d$. Due to the multilinearity of the mixed volume and compatibility of the Minkowski addition with the linear structure of $\wT(\Delta)$ (Corollary \ref{cor:MinkLin}), the function $\vol_{\cC}$ on $\wT(\Delta)$ is a restriction of a quadratic form.
\begin{dfn}
\label{dfn:qForm}
Given a collection $\cC = (K_1, \ldots, K_{d-2})$ of convex bodies and a complete polytopal fan $\Delta$, denote by $q_{\cC,\Delta}$ the unique quadratic form on $\span(\wT(\Delta))$ such that
$$
q_{\cC,\Delta}(h) = \vol(P(h), P(h), \cC) \quad \text{for } h \in \wT(\Delta)
$$
\end{dfn}


The following properties of $q_{\cC,\Delta}$ are immediate.
\begin{itemize}
\item
The kernel of $q_{\cC, \Delta}$ contains $\im V$ and thus has dimension at least $d$.
\item
If $\dim K_i = d$ for all $i = 1, \ldots, d-2$, then $q_{\cC,\Delta}(h) > 0$ for all $h \in \wT(\Delta)$.
\end{itemize}

The following theorem \cite{Ale37, Scn93} tells us more about the signature of $q_\Delta$.

\begin{thm}[Alexandrov-Fenchel inequalities]\label{th:af}
The inequality
\begin{equation}
\label{eqn:AFIneq}
\vol(K, L, \cC)^2 \ge \vol(K, K, \cC) \vol(L, L, \cC)
\end{equation}
holds for all compact convex bodies $K, L, K_1, \ldots, K_{d-2}$.
\end{thm}

Equation \eqref{eqn:AFIneq} basically says that the Gram matrix of the restriction of $q_{\cC,\Delta}$ to any 2-dimensional subspace spanned by two vectors from $\wT(\Delta)$ has a non-positive determinant. This is used in the proof of the following lemma (compare with \cite[Appendix A.3]{Izm10}).

\begin{lem}
Let $\cC = (K_1, \ldots, K_{d-2})$ consist of $d$-dimensional convex bodies. Then for every polytopal fan $\Delta$ the quadratic form $q_{\cC,\Delta}$ has the following properties.
\begin{enumerate}
\item
The positive index of $q_{\cC,\Delta}$ equals $1$.
\item
We have $\dim\ker q_{\cC,\Delta} = d$ if and only if the Alexandrov-Fenchel inequality for
$$
K = P(h), \quad L = P(h'), \quad 0 \ne h, h' \in \wT(\Delta)
$$
holds with equality only when $h' - \lambda h \in \im V$ for some $\lambda$.
\end{enumerate}
\end{lem}
\begin{proof}
Due to $q_{\cC,\Delta}(h) > 0$ for all $h \in \wT(\Delta)$, we have $\ind_+(q_{\cC,\Delta}) \ge 1$. If $\ind_+(q_{\cC,\Delta}) \ge 2$, then there exists a positive vector $x$ in the orthogonal complement to $h$ with respect to $q_{\cC,\Delta}$. For a sufficiently small $\epsilon$ we have
$$
h' := h + \epsilon x \in \wT(\Delta),
$$
since $\wT(\Delta)$ is relatively open, see Lemma \ref{lem:LinTypeCone}. It follows that the restriction of $q_{\cC,\Delta}$ to $\span(h,x) = \span(h,h')$ is positive definite, and hence the Gram matrix of $q_{\cC,\Delta}$ with respect to $h, h'$ has a positive determinant. Thus
$$
q_{\cC,\Delta}(h,h) q_{\cC,\Delta}(h',h') - q_{\cC,\Delta}(h,h')^2 > 0,
$$
which contradicts the Alexandrov-Fenchel inequality. Hence $\ind_+(q_{\cC,\Delta}) = 1$, and the first part of the lemma is proved.

For $0 \ne h, h' \in \wT(\Delta)$ the condition $h' - \lambda h \notin \im V$ for all $\lambda$ is equivalent to
\begin{equation}
\label{eqn:hhPrime}
\dim E = 2 \quad \text{and} \quad E \cap \im V = \{0\},
\end{equation}
where $E = \span(h, h')$. Assume that $\dim\ker q_{\cC,\Delta} = d$, that 
is $\ker q_{\cC,\Delta} = \im V$. Then \eqref{eqn:hhPrime} implies that the restriction of 
$q_{\cC,\Delta}$ to $E$ is non-degenerate, thus the determinant of its Gram matrix doesn't vanish. 
It follows that the Alexandrov-Fenchel inequality for $h,h'$ is strict.

In the opposite direction, if 
$\dim\ker q_{\cC,\Delta} \ge d+1$, then
$$
\dim \span\{h, \ker q_{\cC,\Delta}\} \ge d+2,
$$
which allows us to choose $E \subset \span\{h, \ker q_{\cC,\Delta}\}$ transversal to $\im V$. Then the restriction of $q_{\cC,\Delta}$ to $E$ is degenerate, and representing $E$ as $\span(h,h')$ for $h' \in \wT(\Delta)$, we see that the Alexandrov-Fenchel inequality for $h,h'$ holds with equality. The lemma is proved.
\end{proof}

\subsubsection{Quadratic forms of nullity $d$}
The condition $h' - \lambda h \in \im V$ means that $K = P(h)$ and $L = P(h')$ are homothetic. But equality in \eqref{eqn:AFIneq} can take place also for non-homothetic $K$ and $L$.

\begin{exl}
\label{exl:AFEq}
For $d=3$, let $K = \Sigma$ be a tetrahedron, and $L = \overline{\Sigma \setminus \Sigma_1}$ a truncated tetrahedron, where we assume that the common vertex of $\Sigma$ and $\Sigma_1$ lies at the coordinate origin. Then it can be shown (for example, with the help of the support function) that
$$
\lambda K + \mu L = \overline{(\lambda + \mu)\Sigma \setminus \mu \Sigma_1}
$$
See also Figure \ref{fig:LambdaMu} for the 2-dimensional case.

\begin{figure}[ht]
\centering
\includegraphics{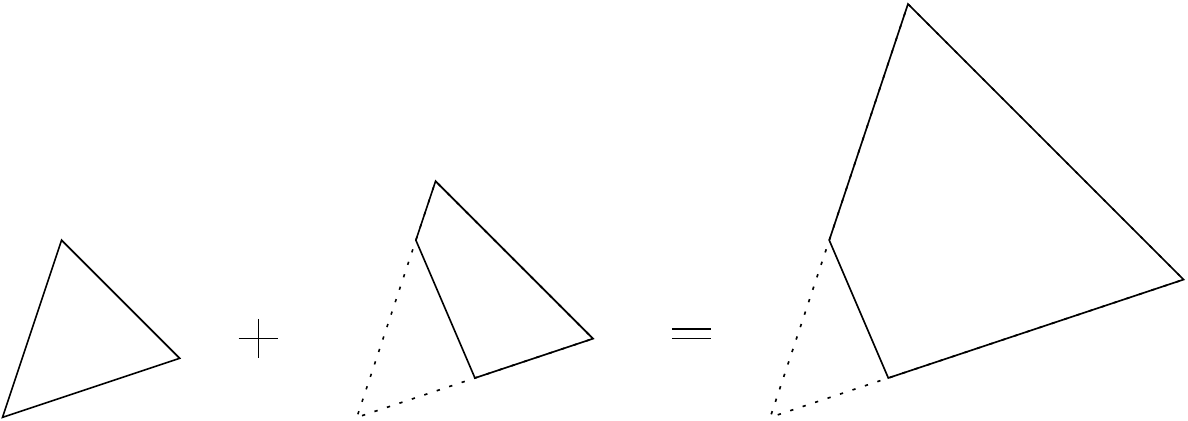}
\caption{Illustration for Example \ref{exl:AFEq}.}
\label{fig:LambdaMu}
\end{figure}

It follows that
\begin{multline*}
\vol(\lambda K + \mu L) = (\lambda + \mu)^3 \vol(\Sigma) - \mu^3 \vol(\Sigma_1)\\
= \lambda^3 \vol(\Sigma) + 3\lambda^2 \mu \vol(\Sigma) + 3 \lambda \mu^2 \vol(\Sigma) + \mu^3 \vol(\overline{\Sigma \setminus \Sigma_1}),
\end{multline*}
and hence
$$
\vol(L,L,K) = \vol(L,K,K) = \vol(K)
$$
Therefore $\vol(K,L,K)^2 = \vol(K,K,K) \vol(L,L,K)$, although $K$ and $L$ are not homothetic.
\end{exl}

A complete characterization of the equality case in the Alexandrov-Fenchel inequality is still missing, see \cite[Section 6.6]{Scn93}, but there are some partial results. Roughly speaking, if equality holds for some non-homothetic $K$ and $L$, then their normal fans are more complicated than those of $K_1, \ldots, K_{d-2}$.

%

\begin{thm}
\label{thm:Kubota}
If $\cC = (B, \ldots, B)$, where $B$ is the unit ball, and $\dim K = \dim L = d$, then the Alexandrov-Fenchel inequality \eqref{eqn:AFIneq} holds with equality only if $K$ and $L$ are homothetic.
\end{thm}
This was proved by Kubota \cite{Kub25}, see also \cite[Theorem 6.6.2]{Scn93}. More generally, $K_i$, $i = 1, \ldots, d-2$ may be any smooth convex bodies.


In view of Example \ref{exl:Quermass}, Theorem \ref{thm:Kubota} implies
\begin{cor}
\label{cor:Area}
For $d=3$, the function $h \mapsto \area(\partial P(h))$ restricted to any type cone $\wT(\Delta)$ is a quadratic form $q_\Delta$ of signature $(+, 0, 0, 0, -, \ldots, -)$.

For $d>3$, the $(d-2)$-nd quermassintegral
$$
h \mapsto \sum_{\sigma \in \Delta^{(d-2)}} |\sigma| \cdot \area(F_\sigma)
$$
restricted to any type cone $\wT(\Delta)$ is a quadratic form $q_\Delta$ of signature $(+, d \cdot 0, -, \ldots, -)$.
\end{cor}
The form $q_\Delta$ in this corollary is proportional to $q_{\cC,\Delta}$, where $\cC = (B, \ldots, B)$.

The next result is a reformulation of \cite[Theorem 6.6.20]{Scn93}.
\begin{thm}Let $\cC = \{K_1, \ldots, K_{d-2}\}$ be a collection of normally equivalent simple polytopes with the normal fan $\Delta_0$.
Then the Alexandrov-Fenchel inequality \eqref{eqn:AFIneq} holds with equality if and only if after applying a suitable homothety to $K$ or $L$ we have
$$
h_K(v) = h_L(v) \quad \text{for all } v \in \Delta_0^{(2)}
$$
\end{thm}
\begin{cor}
\label{cor:AFNormEq}
Let $\cC = \{K_1, \ldots, K_{d-2}\}$ be a collection of normally equivalent simple polytopes with the normal fan $\Delta_0$. Then the quadratic form $q_{\cC,\Delta}$ has signature $(+, d \cdot 0, -, \ldots, -)$ if $\Delta^{(1)} \subset \Delta_0^{(2)}$.
\end{cor}

In particular, the assumption of the Corollary \ref{cor:AFNormEq} is fulfilled when $\Delta^{(1)} = \Delta_0^{(1)}$, that is when $K_i$ and $K = P(h)$ have the same sets of the outward facet normals.
\begin{cor}
\label{cor:WeightArea}
Let $V$ be a vector configuration in $\R^3$, and let $h^0 \in \int\ir(V)$. Then the weighted sum of face areas
$$
h \mapsto \sum_{i=1}^n h^0_i \area(F_i(h))
$$
restricts on every type cone $\wT(\Delta)$ with $\Delta^{(1)} = V$ to a quadratic form of signature $(+, 0, 0, 0, -, \ldots, -)$.
\end{cor}
\begin{proof}
Since
$$
\vol(P(h), P(h), P(h^0)) = \frac13 \sum_{i=1}^n h^0_i \area(F_i(h)),
$$
(see \cite{Scn93,Ewald96}), the quadratic form in the corollary is proportional to $q_{P(h^0),\Delta}$. Polytopes $P(h^0)$ and $P(h)$ have the same sets of face normals: $\Delta_0^{(1)} = \Delta^{(1)} = V$. If $P(h^0)$ is simple, then the remark after Corollary~\ref{cor:AFNormEq} applies. If $P(h^0)$ is not simple, then it can be made simple by truncating vertices that have more than three adjacent edges. The quadratic form does not change (similarly to Example \ref{exl:AFEq}), while the normal fan of $K_1$ becomes only richer. Thus, by Corollary~\ref{cor:AFNormEq}, the form has the right signature.
\end{proof}

\begin{exl}
We know (Lemma \ref{lem:1inIr}) that $\one \in \int\ir(V)$. For $h^0 = \one$ the quadratic form in Corollary \ref{cor:WeightArea} is simply the surface area of $P$, that is coincides with the quadratic form from the first part of Corollary \ref{cor:Area}. Geometrically, $P(\one)$ is the circumscribed polytope. We thus have
$$
\vol(P(h),P(h),P(\one)) = \vol(P(h),P(h),B),
$$
which is another example of non-strict monotonicity of mixed volumes, compare Theorem \ref{thm: segment} and the paragraph before it.
\end{exl}

\begin{rem}\label{rem:hess}
For every simplicial polytopal fan $\Delta$ in $\R^d$ there exists a homogeneous polynomial $Z_\Delta$ of degree $d$ such that
$$
\vol(P(h)) = Z_\Delta(h) \quad \text{for all }h \in \wT(\Delta)
$$
The mixed volume of polytopes with the normal fan $\Delta$ is given by the polarization of the polynomial $Z_\Delta$. (In fact, the existence of development \eqref{eqn:MixVol} is usually proved through approximation of $K_i$ by normally equivalent simple polytopes.)

It follows that the quadratic form $q_{\cC,\Delta}$ for $\cC = (P(h^{0}), \ldots, P(h^{0}))$ with $h^{0} \in \Delta$ is proportional to the Hessian of $\vol(P(h))$ at $h = h^0$. By Corollary \ref{cor:AFNormEq}, this Hessian is non-degenerate modulo translations. This is related to the uniqueness part of the Minkowski problem for polytopes; besides, knowing the signature of the Hessian allows to prove the existence part. By duality, this is related to the infinitesimal rigidity of convex polytopes, \cite{Izm10,IzmHED}.
\end{rem}

\section{Hyperbolic geometry}
\label{sec:HypGeom}

\subsection{From type cones to hyperbolic polyhedra}
\label{sec:HypPol}
Let $\Delta$ be a simplicial polytopal fan with $\Delta^{(1)} = V$, and $T(\Delta)$ be the corresponding type cone. Let $\cC = (K_1, \ldots, K_{d-2})$ be a collection of convex bodies such that the quadratic form $q_{\cC,\Delta}$ from Definition \ref{dfn:qForm} has signature $(+, d \cdot 0, -, \ldots, -)$ (examples of such $\cC$ are given in Corollaries \ref{cor:Area}, \ref{cor:AFNormEq}, and \ref{cor:WeightArea}). Since $\ker q_{\cC,\Delta} = \im V$, the form $q_{\cC,\Delta}$ descends from $\R^n$ to $\R^n/\im V$. By an abuse of notation, this form will also be denoted by $q_{\cC,\Delta}$; its signature is $(+, -, \ldots, -)$.

Thus $q_{\cC,\Delta}$ is a Minkowski scalar product on $\R^n/\im V$. The upper half of the hyperboloid $\{\pi(h) \mid q_{\cC,\Delta}(h) = 1\}$ becomes a model of the hyperbolic space $\H^{n-d-1}$. On the convex polyhedral cone $\cl T(\Delta)$, the form $q_{\cC,\Delta}$ takes non-negative values, due to the non-negativity of mixed volumes. Thus
$$
H_{\cC}(\Delta) := \cl T(\Delta) \cap \{\pi(h) \mid q_{\cC,\Delta}(h) = 1\}
$$
(recall that $\pi \colon \R^n \to \R^n/\im V$ is the projection map)
becomes a convex hyperbolic polyhedron. More exactly, $H_{\cC}(\Delta)$ is the convex hull of finitely many points, some of which can be ideal.

An ideal vertex of $H_{\cC}(\Delta)$ corresponds to $P(h)$ degenerating into a segment (at least in the situations of Corollaries \ref{cor:Area} and \ref{cor:WeightArea}). Such a degeneration is possible if and only if the orthogonal complement of the segment is positively spanned by a subset of~$V$.

A non-simplicial fan $\Delta$ gives rise to a hyperbolic polyhedron $H_{\cC}(\Delta)$ of dimension smaller than $n-d-1$. If $\Delta' \preccurlyeq \Delta$, then $H_{\cC}(\Delta)$ is a face of $H_{\cC}(\Delta')$.


%

In general, the irreducibility domain $\ir(V)$ is composed of several type cones, quadratic forms on which have different extensions to $\R^n$: $q_{\cC,\Delta} \ne q_{\cC,\Delta'}$. However, at the common boundary points of the type cones these forms coincide:
$$
q_{\cC,\Delta}(h) = q_{\cC,\Delta'}(h) \quad \text{for } h \in \cl \wT(\Delta) \cap \cl \wT(\Delta'),
$$
which follows, for example, from the continuity of the mixed volumes with respect to the Hausdorff metric.
Thus the closure of the irredundancy domain $\clir(V)$ becomes equipped with a piecewise hyperbolic metric. Denote by
\begin{equation}
\label{eqn:MC}
M_{\cC}(V) := \bigcup_{\Delta} H_{\cC}(\Delta)
\end{equation}
the corresponding metric space.

Recall that the combinatorial structure of $M_{\cC}(V)$ is that of the chamber fan $\Ch(\bar V)$ intersected with $\clir(V)$, and that $\clir(V) = \core_2(V)$, see Section \ref{sec:Chambers}. (Strictly speaking, $M_{\cC}(V)$ is combinatorially isomorphic to the chamber \emph{complex} of the \emph{affine} Gale diagram, intersected with the \emph{affine} 2-core.) Besides, each facet of $\clir(V)$ corresponds to a positive or hyperbolic circuit of $V$, see \eqref{eqn:IrV4} and \eqref{eqn:ClirFacet}. Let
$$
M_{\cC}^C(V) \cong \clir^C(V) \cap A
$$
be the subset of $M_{\cC}(V)$ that corresponds to the facet $\clir^C(V)$ under the isomorphism of polyhedral complexes
$$
M_{\cC}(V) \cong \clir(V) \cap A,
$$
where $A$ is an affine hyperplane in $\R^{n-d}$ whose intersection with $\clir(V)$ is bounded.

Let us now look at the polytopes $H_{\cC}(\Delta)$ and the metric space $M_{\cC}(V)$ in our standard examples.

\subsection{The examples, continued}
%

In each of the examples below, the quadratic form $q(h)$ (or $q_\Delta(h)$) is the area of a polygon or the surface area of a 3-dimensional polytope, already studied in Section \ref{sec:ExlMV}. It is proportional to $q_{\cC,\Delta}(h)$ with $\cC = \emptyset$ for polygons and $\cC = B$ for 3-polytopes.

In the first three examples the vector configurations are monotypic (there is only one full-dimensional type cone $T(\Delta)$), so that it suffices to study the corresponding hyperbolic polytope $H(\Delta)$. In the last example we have several $H(\Delta)$, and we are also studying the space \eqref{eqn:MC} obtained by gluing them together.

\subsubsection{Parallelepipeds with fixed face directions form an ideal hyperbolic triangle}
\label{sec:Par3}
We identified $\R^6/\im V$ with $\R^3 = \{h_4 = h_5 = h_6 = 0\}$.
The cone $\cl T(\Delta) = \{h \in \R^3 \mid h_i \ge 0,\, i = 1, 2, 3\}$ is spanned by three vectors $e_1$, $e_2$, $e_3$, which are light-like with respect to the quadratic form $q(h) = \frac{2}{D}(h_1h_2 + h_2h_3 + h_3h_1)$. Thus $H(\Delta)$ is an ideal hyperbolic triangle.

\subsubsection{Polygons}
We have $\dim T(\alpha) = n-2$, so that $H(\alpha)$ is an $(n-3)$-dimensional hyperbolic polyhedron.
Recall from Section \ref{ex:poly1} that the facet $T_i$ of the cone $T(\alpha)$ corresponds to vanishing of the $i$-th edge of the polygon: $\ell_i(h) = 0$. At the same time, \eqref{eq:ar polygon} implies $\a(e_i, h) = \frac12 \ell_i(h)$. It follows that $\aff(T_i)$ is orthogonal to $e_i$ with respect to the quadratic form $q$. In other words, the point corresponding to $e_i$ is polar dual to the corresponding facet $H_i$ of $H(\alpha)$. (This point lies in the de Sitter space, if $T_i$ intersects the interior of the light cone.) This leads to the following formula for the dihedral angle $\Theta_{i-1i}$ between
$H_{i-1}$ and $H_{i}$:
$$\cos^2(\Theta_{i-1i})=\frac{\sin(\alpha_{i-1})\sin(\alpha_{i+1})}{\sin(\alpha_{i-1}+\alpha_i)\sin(\alpha_i+\alpha_{i+1}) }. $$
If $j\notin\{i-1,i,i+1\}$, then this formula implies that
$H_i$ and $H_j$ meet orthogonally. Thus $H(\alpha)$ is an \emph{orthoscheme}.
All hyperbolic orthoschemes (included truncated and doubly truncated) can be constructed this way.
For example, $H\left(\frac{2\pi}{5},\frac{2\pi}{5},\frac{2\pi}{5},\frac{2\pi}{5},\frac{2\pi}{5}\right)$ is the regular right-angled hyperbolic pentagon, see Figure \ref{fig:Pentagon}. For more details, see \cite{BG92,FillEM}.

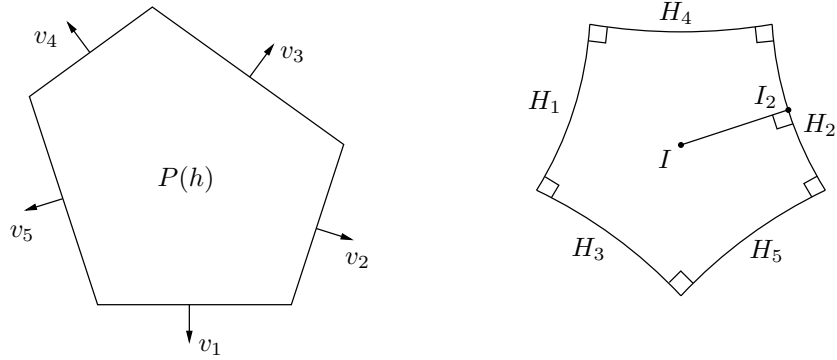
\begin{figure}[ht]
\centering
\begin{picture}(0,0)%
\includegraphics{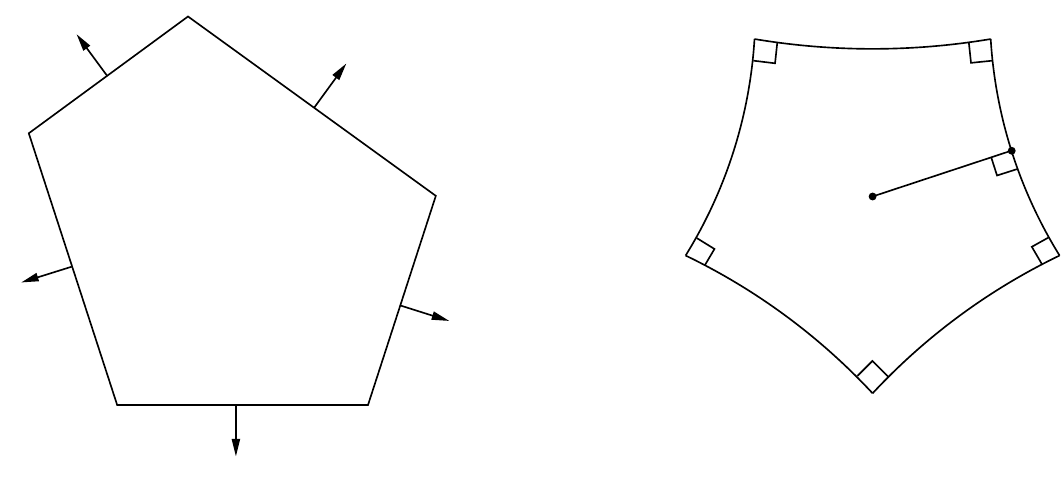}%
\end{picture}%
\setlength{\unitlength}{4144sp}%
\begingroup\makeatletter\ifx\SetFigFont\undefined%
\gdef\SetFigFont#1#2#3#4#5{%
  \reset@font\fontsize{#1}{#2pt}%
  \fontfamily{#3}\fontseries{#4}\fontshape{#5}%
  \selectfont}%
\fi\endgroup%
\begin{picture}(4853,2184)(61,-1353)
\put(951,-308){\makebox(0,0)[lb]{\smash{{\SetFigFont{10}{12.0}{\rmdefault}{\mddefault}{\updefault}{\color[rgb]{0,0,0}$P(h)$}%
}}}}
\put(1193,-1289){\makebox(0,0)[lb]{\smash{{\SetFigFont{10}{12.0}{\rmdefault}{\mddefault}{\updefault}{\color[rgb]{0,0,0}$v_1$}%
}}}}
\put(2063,-771){\makebox(0,0)[lb]{\smash{{\SetFigFont{10}{12.0}{\rmdefault}{\mddefault}{\updefault}{\color[rgb]{0,0,0}$v_2$}%
}}}}
\put(1678,459){\makebox(0,0)[lb]{\smash{{\SetFigFont{10}{12.0}{\rmdefault}{\mddefault}{\updefault}{\color[rgb]{0,0,0}$v_3$}%
}}}}
\put(4773, 15){\makebox(0,0)[lb]{\smash{{\SetFigFont{10}{12.0}{\rmdefault}{\mddefault}{\updefault}{\color[rgb]{0,0,0}$H_2$}%
}}}}
\put(3403,-724){\makebox(0,0)[lb]{\smash{{\SetFigFont{10}{12.0}{\rmdefault}{\mddefault}{\updefault}{\color[rgb]{0,0,0}$H_3$}%
}}}}
\put(3143,153){\makebox(0,0)[lb]{\smash{{\SetFigFont{10}{12.0}{\rmdefault}{\mddefault}{\updefault}{\color[rgb]{0,0,0}$H_1$}%
}}}}
\put( 76,-610){\makebox(0,0)[lb]{\smash{{\SetFigFont{10}{12.0}{\rmdefault}{\mddefault}{\updefault}{\color[rgb]{0,0,0}$v_5$}%
}}}}
\put(218,550){\makebox(0,0)[lb]{\smash{{\SetFigFont{10}{12.0}{\rmdefault}{\mddefault}{\updefault}{\color[rgb]{0,0,0}$v_4$}%
}}}}
\put(4457,-736){\makebox(0,0)[lb]{\smash{{\SetFigFont{10}{12.0}{\rmdefault}{\mddefault}{\updefault}{\color[rgb]{0,0,0}$H_5$}%
}}}}
\put(3920,684){\makebox(0,0)[lb]{\smash{{\SetFigFont{10}{12.0}{\rmdefault}{\mddefault}{\updefault}{\color[rgb]{0,0,0}$H_4$}%
}}}}
\put(3916,-196){\makebox(0,0)[lb]{\smash{{\SetFigFont{10}{12.0}{\rmdefault}{\mddefault}{\updefault}{\color[rgb]{0,0,0}$I$}%
}}}}
\put(4481,184){\makebox(0,0)[lb]{\smash{{\SetFigFont{10}{12.0}{\rmdefault}{\mddefault}{\updefault}{\color[rgb]{0,0,0}$I_2$}%
}}}}
\end{picture}%
\caption{The space of equiangular pentagons is the regular right-angled hyperbolic pentagon.}
\label{fig:Pentagon}
\end{figure}

The polyhedron $H(\alpha)$ contains a distinguished point $I$ that corresponds to a circumscribed polygon. In other words, $I := (\a(\one))^{-1/2} \one$.
For further reference,  we want to compute the 
hyperbolic distance from $I$ to the facet $H_i$ of $H(\alpha)$.
Let $I_i$ be the orthogonal projection of 
$I$ onto $H_i$. As $e_i$ is orthogonal to 
$H_i$, the polygon corresponding to $I_i$ is
obtained from a circumscribed polygon by moving its $i$-th edge until it disappears. Thus $I_i$ corresponds to a circumscribed polygon with one edge less. In other words, its support vector is proportional to
$$
\hat{\one}_i := (1, \ldots, 0, \ldots, 1),
$$
with the only zero at the $i$-th place.
%
%

\begin{figure}[ht]
\centering
\includegraphics{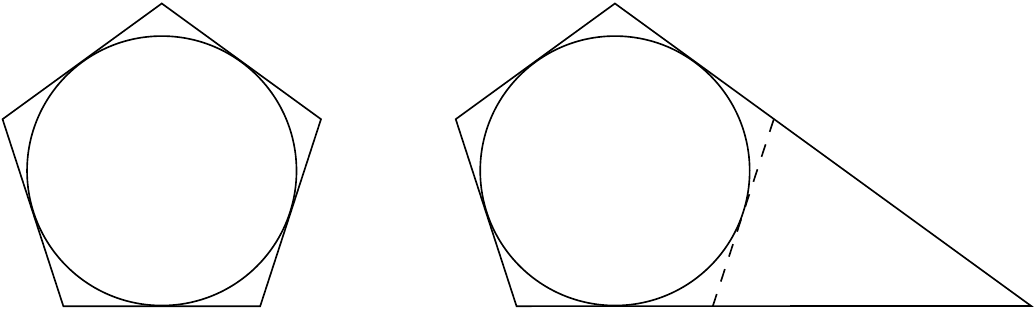}
\caption{The polygons $P(\one)$ and $P(\hat{\one}_i)$.}
\label{fig:P1i}
\end{figure}

By standard hyperbolic geometry, we have

$$\cosh \dist(I, I_i) = \frac{\area (\one,\hat{\one}_i)}{\sqrt{\area(\one)\area(\hat\one_i)}} $$

By \eqref{eq:ar polygon} $\area (\one, \hat\one_i)$ is half of the perimeter of $P(\hat\one_i)$. 
Since $P(\hat\one_i)$ is circumscribed about a unit circle, its perimeter equals $2\area(\hat\one_i)$, so 

\begin{equation}\label{eq: gammai}\cosh \dist(I, I_i)= \sqrt{\frac{\a(\hat\one_i)}{\a(\one)}}.\end{equation}

\subsubsection{Polygonal prisms}

Denote by $H^+(\alpha)$ the hyperbolic polyhedron
that is the section of $T^+(\alpha)$. Then $H^+(\alpha)$ is a pyramid over $H(\alpha)$ with apex $A$ that corresponds to prisms degenerating into a segment. Since this degeneration nullifies the surface area, $A$ is an ideal point.
In coordinates, $A$ is given by
 the light-like vector
$(\zero,1,1)$ (proportional to $(\zero,0,1)$ and $(\zero, 1, 0)$ modulo $\im V$).
The orthogonal projection of the apex $A$ onto the hyperplane $h_{n+1} = h_{n+2} = 0$ that contains the basis
$H(\alpha)$ is a linear combination of $(\zero, 1, 1)$ and the normal $(\one, -1, -1)$ to that hyperplane. We obtain the vector $(\one, 0, 0)$. In other words, the orthogonal projection of the apex $A$ is the point $I \in H(\alpha)$ that corresponds to a circumscribed polygon.

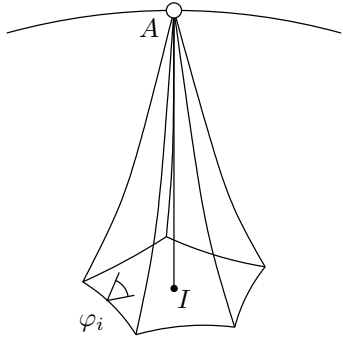
\begin{figure}[ht]
\centering
\begin{picture}(0,0)%
\includegraphics{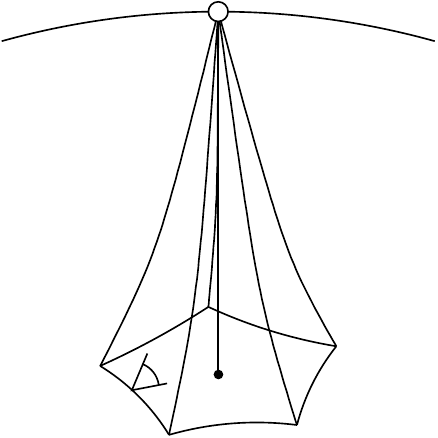}%
\end{picture}%
\setlength{\unitlength}{4144sp}%
\begingroup\makeatletter\ifx\SetFigFont\undefined%
\gdef\SetFigFont#1#2#3#4#5{%
  \reset@font\fontsize{#1}{#2pt}%
  \fontfamily{#3}\fontseries{#4}\fontshape{#5}%
  \selectfont}%
\fi\endgroup%
\begin{picture}(1996,2000)(893,-1828)
\put(1909,-1655){\makebox(0,0)[lb]{\smash{{\SetFigFont{10}{12.0}{\rmdefault}{\mddefault}{\updefault}{\color[rgb]{0,0,0}$I$}%
}}}}
\put(1327,-1762){\makebox(0,0)[lb]{\smash{{\SetFigFont{10}{12.0}{\rmdefault}{\mddefault}{\updefault}{\color[rgb]{0,0,0}$\phi_i$}%
}}}}
\put(1684,-40){\makebox(0,0)[lb]{\smash{{\SetFigFont{10}{12.0}{\rmdefault}{\mddefault}{\updefault}{\color[rgb]{0,0,0}$A$}%
}}}}
\end{picture}%
\caption{The space of prisms is the pyramid over the space of polygons.}
\label{fig:IdPyr}
\end{figure}

Let us compute the dihedral angle $\phi_i$ between the facets $\conv\{A, H_i\}$ and $H(\alpha)$ of the pyramid $H^+(\alpha)$.
Applying the hyperbolic Pythagorean theorem \cite{Thu97} to the triangle $AII_i$, we obtain
from \eqref{eq: gammai}
$$ \sin \phi_i= \sqrt{\frac{\a(\one)}{\a(\hat\one_i)}}.$$

The formulas for the other dihedral angles of $H^+(\alpha)$ don't look nice.


\subsubsection{Triangular bipyramid}
\label{sec:TrBiHyp}
The space $M_{\cC}(V)$ is glued from six quadrilaterals as shown on Figure \ref{fig:TypeConesExl}. Each of these quadrilaterals is equipped with a hyperbolic metric that arises from the corresponding quadratic form $q_\Delta$.
A priori, the hexagon on Figure \ref{fig:TypeConesExl} can become a hyperbolic 12-gon with a conic singularity at the point where all six quadrilaterals meet. Let us understand what is its shape in reality.


Consider the hyperbolic quadrilateral $H(\Delta)$, where $\Delta$ is the cone over the quadrilateral shaded on Figure \ref{fig:TypeConesExl}.
The equations of the boundary hyperplanes on Figure \ref{fig:TypeConesExl}, right, can be rewritten in coordinates $f_1, f_2, f_3$ as shown on Figure \ref{fig:RightAnglesExl}. Since
\begin{equation}
\label{eqn:AreaPh}
q_\Delta(h) = c_1 f_1^2 - c_2 f_2^2 - c_3 f_3^2,
\end{equation}
as was proved in Section \ref{sec:TriBiForm}, it follows that the quadrilateral $H(\Delta)$ has three right angles. The same holds for all of the other type cones. For two adjacent type cones, the two adjacent right angles sum up to $\pi$. It follows that the union of all six quadrilaterals is a right-angled hyperbolic hexagon, possibly with a conic singularity in the interior.

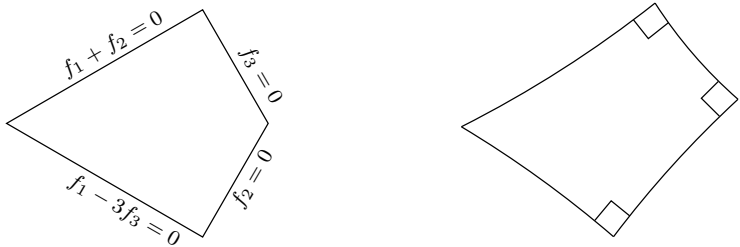
\begin{figure}[ht]
\centering
\begin{picture}(0,0)%
\includegraphics{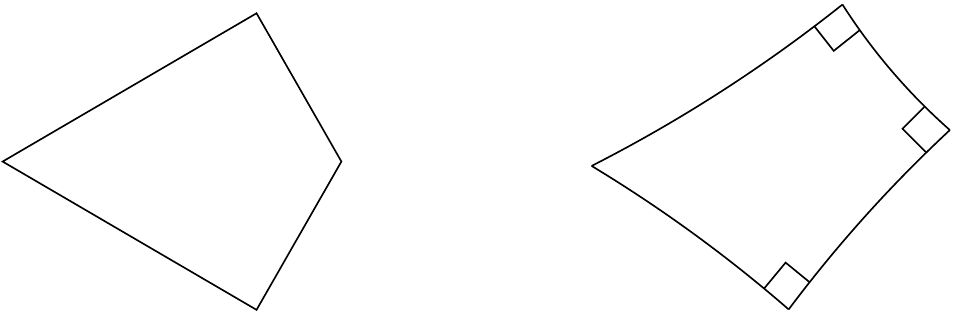}%
\end{picture}%
\setlength{\unitlength}{4144sp}%
\begingroup\makeatletter\ifx\SetFigFont\undefined%
\gdef\SetFigFont#1#2#3#4#5{%
  \reset@font\fontsize{#1}{#2pt}%
  \fontfamily{#3}\fontseries{#4}\fontshape{#5}%
  \selectfont}%
\fi\endgroup%
\begin{picture}(4351,1416)(2395,-750)
\put(2759,216){\rotatebox{30.0}{\makebox(0,0)[lb]{\smash{{\SetFigFont{8}{9.6}{\rmdefault}{\mddefault}{\updefault}{\color[rgb]{0,0,0}$f_1 + f_2 = 0$}%
}}}}}
\put(3798,-570){\rotatebox{60.0}{\makebox(0,0)[lb]{\smash{{\SetFigFont{8}{9.6}{\rmdefault}{\mddefault}{\updefault}{\color[rgb]{0,0,0}$f_2 = 0$}%
}}}}}
\put(2766,-423){\rotatebox{330.0}{\makebox(0,0)[lb]{\smash{{\SetFigFont{8}{9.6}{\rmdefault}{\mddefault}{\updefault}{\color[rgb]{0,0,0}$f_1 - 3f_3 = 0$}%
}}}}}
\put(3791,377){\rotatebox{300.0}{\makebox(0,0)[lb]{\smash{{\SetFigFont{8}{9.6}{\rmdefault}{\mddefault}{\updefault}{\color[rgb]{0,0,0}$f_3 = 0$}%
}}}}}
\end{picture}%
\caption{The hyperbolic quadrilateral corresponding to the type cone.}
\label{fig:RightAnglesExl}
\end{figure}

To proceed further, one needs to know the coefficients $c_1, c_2, c_3$ in \eqref{eqn:AreaPh}. This can be done with the help of formulas from Appendix \ref{sec:App2}, but in our case we can exploit the symmetry.

Let $\Delta'$ be the type cone $h_1 \ge h_3 \ge h_2$ sharing with $\Delta$ the facet $h_2 = h_3$. A priori we have $q_\Delta \ne q_{\Delta'}$: indeed, the surface areas of polytopes with pairwise parallel faces but different combinatorial structure is in general given by different quadratic forms. However, we claim that in our case $q_{\Delta'} = q_\Delta$.

The fans $\Delta$ and $\Delta'$ are related by a flip,
see Figure~\ref{fig:flip}.
It follows that
\begin{multline*}
q_{\Delta'}(h) = \area(\partial P(h)) + \area(F_3(\Sigma)) + \area(F_5(\Sigma))\\
- \area(F_2(\Sigma)) - \area(F_6(\Sigma)),
\end{multline*}
where $\Sigma$ is the tetrahedron bounded by the planes $H_2, H_3, H_5, H_6$, see Figure~\ref{fig:FlipExl} and \ref{fig:but-move}.
On the other hand, we have $v_3 + v_5 - v_2 - v_6 = 0$, and $-v_3, -v_5, v_2, v_6$ are the outer unit normals to the faces of $\Sigma$. Together with the Minkowski identity $\sum \area(F_i) v_i = 0$ this implies that $\area(F_3(\Sigma)) + \area(F_5(\Sigma)) - \area(F_2(\Sigma)) - \area(F_6(\Sigma)) = 0$, and thus $q_{\Delta'} = q_\Delta$.

\begin{figure}[ht]
\centering
\begin{picture}(0,0)%
\includegraphics{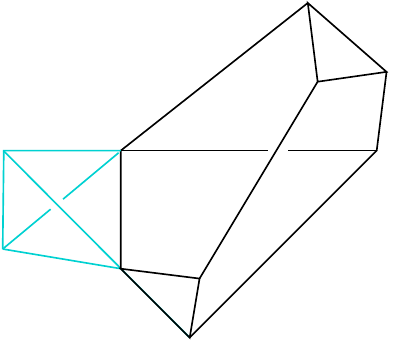}%
\end{picture}%
\setlength{\unitlength}{4144sp}%
\begingroup\makeatletter\ifx\SetFigFont\undefined%
\gdef\SetFigFont#1#2#3#4#5{%
  \reset@font\fontsize{#1}{#2pt}%
  \fontfamily{#3}\fontseries{#4}\fontshape{#5}%
  \selectfont}%
\fi\endgroup%
\begin{picture}(1779,1554)(349,-1558)
\put(446,-1361){\makebox(0,0)[lb]{\smash{{\SetFigFont{10}{12.0}{\rmdefault}{\mddefault}{\updefault}{\color[rgb]{0,0,0}$\Sigma$}%
}}}}
\end{picture}%
\caption{The difference $q_{\Delta'} - q_\Delta$ for fans related by a flip.}
\label{fig:FlipExl}
\end{figure}
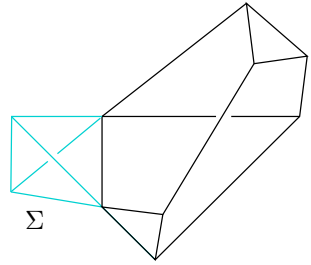

\begin{figure}[ht]
\centering
\begin{picture}(0,0)%
\includegraphics{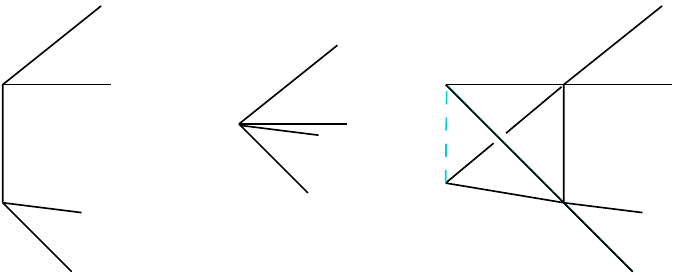}%
\end{picture}%
\setlength{\unitlength}{4144sp}%
\begingroup\makeatletter\ifx\SetFigFont\undefined%
\gdef\SetFigFont#1#2#3#4#5{%
  \reset@font\fontsize{#1}{#2pt}%
  \fontfamily{#3}\fontseries{#4}\fontshape{#5}%
  \selectfont}%
\fi\endgroup%
\begin{picture}(3084,1239)(889,-1558)
\end{picture}%
\caption{Changing a height of a convex polytope with combinatorics $ \Delta$ until going outside $\tilde{T}(\Delta)$, but keeping the combinatorics $\Delta$.
The shaded edge has negative length (in the sense of Lemma~\ref{lem:EllLin}). Compare ``butterfly moves'' \cite{Thu98}.}
\label{fig:but-move}
\end{figure}

The same holds for any other pair of adjacent type cones. Thus for all $h \in \ir(V)$ the surface area of $P(h)$ is given by the same quadratic form $q(h)$, independent of the combinatorics of $P(h)$.

It follows that the right-angled hyperbolic hexagon $\cup_\Delta H(\Delta)$ has no cone singularity in the interior.
By exploiting the symmetry of the vector configuration further, it can be shown that the hexagon is equilateral and that the lines subdividing it into the type polygons are its axes of symmetry. Indeed, the maps
$$
(h_1, h_2, h_3) \mapsto (h_1, h_3, h_2), \qquad (h_1, h_2, h_3) \mapsto \left(\frac23 - h_1, \frac23 - h_2, \frac23 - h_3\right)
$$
send the hexagon to itself and preserve the value of $q(h)$: the first map corresponds to reflection of $P(h)$ in the coordinate $xz$-plane, the second one to reflection in the $xy$-plane. Thus both maps are hyperbolic isometries. It follows that all of the type quadrilaterals are congruent and symmetric.

The invariance of the form $q(h)$ under permutations of coordinates allows us to quickly compute the coefficients $c_i$ in \eqref{eqn:AreaPh} up to a common factor. Namely, we have $c_1 = 3c_2$, $c_3 = 3c_1$, and
$$
q(h) = c(-h_1^2 - h_2^2 - h_3^2 + 4h_1h_2 + 4h_2h_3 + 4h_3h_1)
$$
The factor $c$ depends on the parameter $\lambda$ in \eqref{eqn:VectBipyr}.


\subsection{Dihedral angles at the boundary}
\label{sec:DihAngles}
Let $C_1$ and $C_2$ be positive or hyperbolic circuits determining two facets of $\clir(V)$, and let $M_{\cC}^{C_1}(V)$ and $M_{\cC}^{C_2}(V)$ be the corresponding subcomplexes of $\partial M_{\cC}(V)$, see Section \ref{sec:HypPol}. If $M_{\cC}^{C_1}(V)$ and $M_{\cC}^{C_2}(V)$ intersect along a codimension 2 (with respect to $M_{\cC}(V)$) subcomplex, then one may ask what are the dihedral angles of $M_{\cC}(V)$ along this intersection. Note that the dihedral angles may be different at different cells $H_{\cC}(\Delta)$ of $M_{\cC}^{C_1}(V) \cap M_{\cC}^{C_2}(V)$.

The following theorem describes a special case when $M_{\cC}^{C_1}(V)$ and $M_{\cC}^{C_2}(V)$ intersect at a right angle. The proof generalizes an argument from \cite{BG92}. Example in Section \ref{sec:TrBiHyp} can serve as an illustration.


\begin{thm}
Let $C_1, C_2 \subset [n]$ be hyperbolic circuits with the positive indices $p_1$ and $p_2$ such that
$$
|C_1| = |C_2| = d+1 \quad \text{and} \quad p_1 \notin C_2, p_2 \notin C_1
$$
(in particular $p_1 \ne p_2$)
and such that $\clir^{C_1}(V)$ and $\clir^{C_2}(V)$ are facets of $\clir(V)$ intersecting along a codimension 2 face. Then the dihedral angle between the corresponding subsets of $\partial M_{\cC}(V)$ equals $\frac{\pi}2$ at every point.
\end{thm}
\begin{proof}
Let $h^0$ be such that $\pi(h^0) \in \relint(\clir^{C_1}(V) \cap \clir^{C_2}(V))$.
By Lemma~\ref{lem:Bdry2Ir}, the polytope $P(h)$ has dimension $d$ and outer normals $V_{[n] \setminus \{p_1, p_2\}}$. For a generic choice of $h^0$, the polytope $P(h^0)$ is simple. Since $|C_1^-| = |C_2^-| = d$, the faces $F_{C_1^-}$ and $F_{C_2^-}$ are two different vertices of $P(h^0)$. Also by Lemma \ref{lem:Bdry2Ir}, for all $h \in \wir(V)$ sufficiently close to $h^0$ the polytopes $P(h)$ have the same combinatorics. Namely, $P(h)$ is obtained by an independent truncation of the vertices $F_{C_1^-}$ and $F_{C_2^-}$ of the polytope $P(h^0)$, where $h - h^0 = \epsilon_1 \bar v_{p_1} + \epsilon_2 \bar v_{p_2}$. Let $\Delta$ be the corresponding complete simplicial fan. It suffices to compute the dihedral angle of the hyperbolic polytope $H_{\cC}(\Delta)$ at $\clir^{C_1}(V) \cap \clir^{C_2}(V)$.

We have
$$
P(h) = \overline{P(h^0) \setminus (\Sigma_1 \cup \Sigma_2)},
$$
where $\Sigma_i$ is a simplex with outward facet normals $V_{C_i^-} \cup \{-v_{p_i}\}$, $i = 1,2$, $\Sigma_1 \cap \Sigma_2 = \emptyset$.
The valuation property \eqref{eqn:ValProp} implies
\begin{multline}
\vol_{\cC}(P(h^0)) = \vol_{\cC}(P(h)) - \vol_{\cC}(\Sigma_1) - \vol_{\cC}(\Sigma_2)\\
+ \vol_{\cC}(F_{p_1}(h^0)) + \vol_{\cC}(F_{p_2}(h^0)),
\end{multline}
where $F_{p_i}$ is a common facet of $P(h)$ and $\Sigma_i$. It follows that
\begin{equation}
\label{eqn:OrthDecomp}
q_{\cC,\Delta}(h^0) = q(h_{[n]\setminus\{p_1,p_2\}}) + c_1 f_1^2(h) + c_2 f_2^2(h),
\end{equation}
where the quadratic form $q$ on the right hand side doesn't depend on $h_{p_1}$ and $h_{p_2}$, and
$$
f_i(h) = \langle \lambda^{C_i}, h \rangle, \quad i = 1,2
$$
is the linear function corresponding to the circuit $C_i$.
Indeed, the linear measurements of both $\Sigma_1(h)$ and $F_{p_1}(h)$ are proportional to $f_1(h)$, and thus by homogeneity \eqref{eqn:Hom2} both $\vol_{\cC}(\Sigma_1(h))$ and $\vol_{\cC}(F_{p_1}(h))$ are proportional to $f_1^2(h)$.

As the decomposition \eqref{eqn:OrthDecomp} is orthogonal, the hyperplanes
$$
\{h \mid f_1(h) = 0\} \quad \text{and} \quad \{h \mid f_2(h) = 0\}
$$
are orthogonal with respect to the quadratic form $q_{\cC,\Delta}$. It follows that the corresponding facets of $\cl T(\Delta')$ are orthogonal, and hence the dihedral angle of $M_{\cC}(V)$ at $\pi(h)$ equals $\frac{\pi}2$.
\end{proof}

\begin{rem}
The monotonicity of mixed volumes under inclusion implies that $c_1, c_2 \le 0$ in \eqref{eqn:OrthDecomp}. Similarly to Section \ref{sec:TriBiForm}, this can be used to determine the signature of the quadratic form $q_{\cC,\Delta}$ if $\Delta$ is the normal fan of a (multiply) truncated simplex.
\end{rem}

\subsection{Cone angles in the interior}
\label{sec:IntAngles}
In this section we provide an evidence that the metric space $M_{\cC}(V)$ is in general a cone-manifold, that is some of its interior codimension 2 strata can have total angles different from $2\pi$ around them.

Let $\Delta$, $\Delta^{(1)} = V$, be a polytopal fan in $\R^d$ all of whose cones are simplicial except for two, each of which is spanned by $d+1$ vectors in general position (non-positive circuits of full rank). Such a cone can be triangulated in two ways, and perturbations of $h \in T(\Delta)$ allow to obtain any of the four combinations of these two pairs of triangulations. This results in four simplicial fans that we denote by $\Delta_{00}$, $\Delta_{01}$, $\Delta_{10}$, and $\Delta_{11}$. For $d = 3$ we have two quadrilaterals and subdivide each of them by a diagonal; it is reflected in the face structure of $P(h)$ by flipping two edges, see Figure \ref{fig:Walls}, left.

\begin{figure}[ht]
\centering
\begin{picture}(0,0)%
\includegraphics{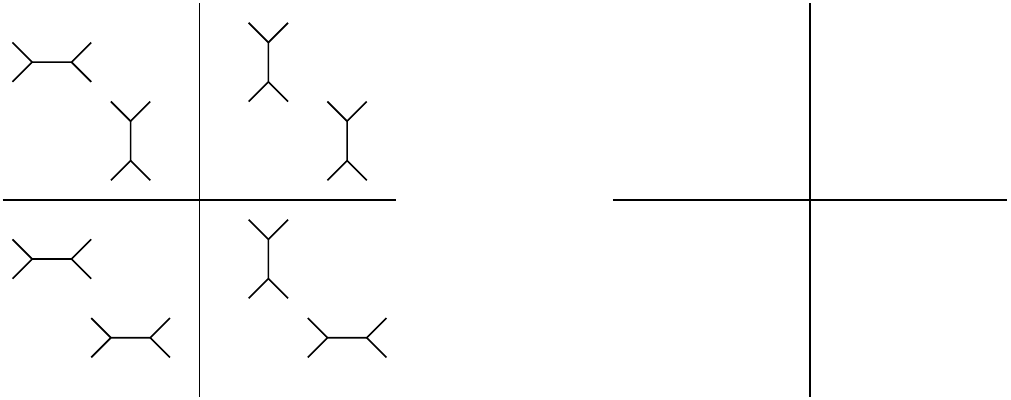}%
\end{picture}%
\setlength{\unitlength}{4144sp}%
\begingroup\makeatletter\ifx\SetFigFont\undefined%
\gdef\SetFigFont#1#2#3#4#5{%
  \reset@font\fontsize{#1}{#2pt}%
  \fontfamily{#3}\fontseries{#4}\fontshape{#5}%
  \selectfont}%
\fi\endgroup%
\begin{picture}(4614,1824)(-2801,-973)
\put(1036,299){\makebox(0,0)[lb]{\smash{{\SetFigFont{10}{12.0}{\rmdefault}{\mddefault}{\updefault}{\color[rgb]{0,0,0}$q+c_1f_1^2$}%
}}}}
\put(  1,-466){\makebox(0,0)[lb]{\smash{{\SetFigFont{10}{12.0}{\rmdefault}{\mddefault}{\updefault}{\color[rgb]{0,0,0}$q+c_2f_2^2$}%
}}}}
\put(991,-466){\makebox(0,0)[lb]{\smash{{\SetFigFont{10}{12.0}{\rmdefault}{\mddefault}{\updefault}{\color[rgb]{0,0,0}$q+c_1f_1^2+c_2f_2^2$}%
}}}}
\put(316,299){\makebox(0,0)[lb]{\smash{{\SetFigFont{10}{12.0}{\rmdefault}{\mddefault}{\updefault}{\color[rgb]{0,0,0}$q$}%
}}}}
\end{picture}%
\caption{Two independent flips.}
\label{fig:Walls}
\end{figure}
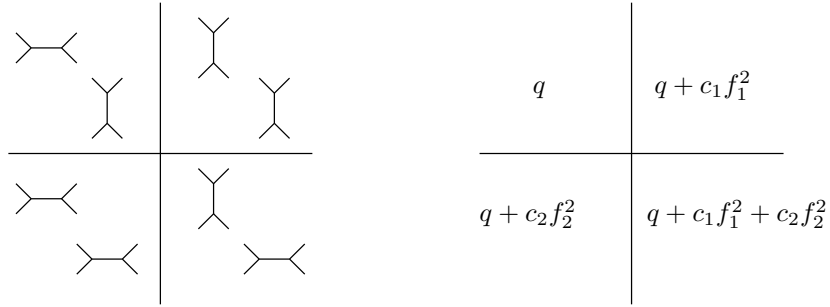

Locally, the arrangement of the type cones $T(\Delta_{ij})$ around $T(\Delta)$ is that of the intersections of half-spaces determined by linear functionals $f_1$ and $f_2$ on $\R^n$. Here $f_i$ is proportional to the length of any edge in a triangulation of the $i$-th circuit, $i = 1,2$. Denote by $q$ the quadratic form $q_{\cC,\Delta_{00}}$. Crossing the hyperplane $f_i(h) = 0$ changes the quadratic form by a multiple of $f_i^2$, so that we have the situation on Figure \ref{fig:Walls}, right.

The total angle around $H_{\cC}(\Delta)$ in the metric space $M_{\cC}(V)$ equals $2\pi$ if one of the following conditions is fulfilled:
\begin{enumerate}
\item
One of the coefficients $c_1$ or $c_2$ equals zero.
\item
The dihedral angle of $H_{\cC}(\Delta_{00})$ at $H_{\cC}(\Delta)$ equals $\frac{\pi}2$.
\end{enumerate}
Figure \ref{fig:Singular} shows how the geometry around $T(\Delta)$ changes if we change the quadratic form $q$ in two steps: first adding a summand in one half-space, then in the other. It indicates that if neither of the above conditions is satisfied, then the total angle around $H_{\cC}(\Delta)$ is different from $2\pi$.

\begin{figure}[ht]
\centering
\begin{picture}(0,0)%
\includegraphics{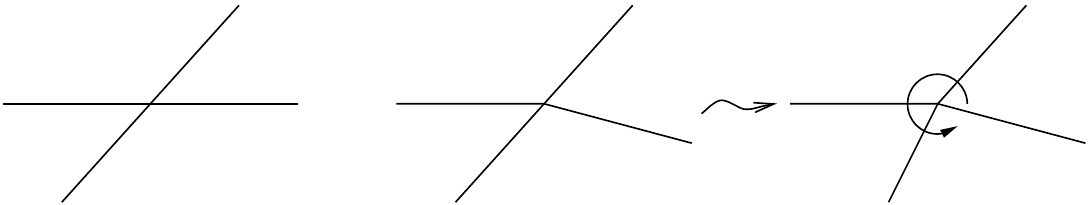}%
\end{picture}%
\setlength{\unitlength}{4144sp}%
\begingroup\makeatletter\ifx\SetFigFont\undefined%
\gdef\SetFigFont#1#2#3#4#5{%
  \reset@font\fontsize{#1}{#2pt}%
  \fontfamily{#3}\fontseries{#4}\fontshape{#5}%
  \selectfont}%
\fi\endgroup%
\begin{picture}(4974,924)(-11,-523)
\put(676,-331){\makebox(0,0)[lb]{\smash{{\SetFigFont{10}{12.0}{\rmdefault}{\mddefault}{\updefault}{\color[rgb]{0,0,0}$q+f_1^2$}%
}}}}
\put(1531,-106){\makebox(0,0)[lb]{\smash{{\SetFigFont{10}{12.0}{\rmdefault}{\mddefault}{\updefault}{\color[rgb]{0,0,0}$\simeq$}%
}}}}
\put(4322,-300){\makebox(0,0)[lb]{\smash{{\SetFigFont{10}{12.0}{\rmdefault}{\mddefault}{\updefault}{\color[rgb]{0,0,0}$\ne 2\pi$}%
}}}}
\put(901, 74){\makebox(0,0)[lb]{\smash{{\SetFigFont{10}{12.0}{\rmdefault}{\mddefault}{\updefault}{\color[rgb]{0,0,0}$q+f_1^2$}%
}}}}
\put(361, 74){\makebox(0,0)[lb]{\smash{{\SetFigFont{10}{12.0}{\rmdefault}{\mddefault}{\updefault}{\color[rgb]{0,0,0}$q$}%
}}}}
\put( 91,-331){\makebox(0,0)[lb]{\smash{{\SetFigFont{10}{12.0}{\rmdefault}{\mddefault}{\updefault}{\color[rgb]{0,0,0}$q$}%
}}}}
\put(2431,-376){\makebox(0,0)[lb]{\smash{{\SetFigFont{10}{12.0}{\rmdefault}{\mddefault}{\updefault}{\color[rgb]{0,0,0}$q$}%
}}}}
\put(3601,-376){\makebox(0,0)[lb]{\smash{{\SetFigFont{10}{12.0}{\rmdefault}{\mddefault}{\updefault}{\color[rgb]{0,0,0}$q+f_2^2$}%
}}}}
\put(3916,164){\makebox(0,0)[lb]{\smash{{\SetFigFont{10}{12.0}{\rmdefault}{\mddefault}{\updefault}{\color[rgb]{0,0,0}$q$}%
}}}}
\put(2251, 74){\makebox(0,0)[lb]{\smash{{\SetFigFont{10}{12.0}{\rmdefault}{\mddefault}{\updefault}{\color[rgb]{0,0,0}$q$}%
}}}}
\put(1891,-331){\makebox(0,0)[lb]{\smash{{\SetFigFont{10}{12.0}{\rmdefault}{\mddefault}{\updefault}{\color[rgb]{0,0,0}$q$}%
}}}}
\put(2791, 29){\makebox(0,0)[lb]{\smash{{\SetFigFont{10}{12.0}{\rmdefault}{\mddefault}{\updefault}{\color[rgb]{0,0,0}$q$}%
}}}}
\end{picture}%
\caption{Angles around a generic codimension 2 stratum.}
\label{fig:Singular}
\end{figure}
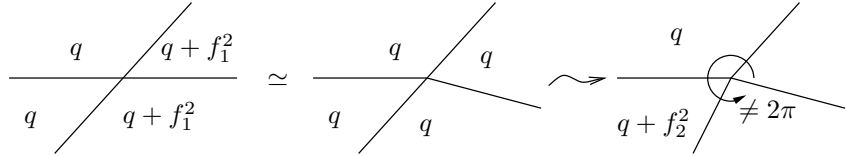

\begin{exl}
The normals of the triangular bipyramid produce six type cones on Figure \ref{fig:TypeConesExl}. There is only one codimension 2 stratum, namely the point in the center, and it corresponds to vanishing of three edges rather than two. The symmetry can be broken by perturbing the Gale diagram $\bar V$ on Figure \ref{fig:GaleCoIr} (which is equivalent to perturbing the normal vectors of the bipyramid). Figure \ref{fig:BipyrPerturb} shows the type cones arising from a generic perturbation.

\begin{figure}[ht]
\centering
\includegraphics{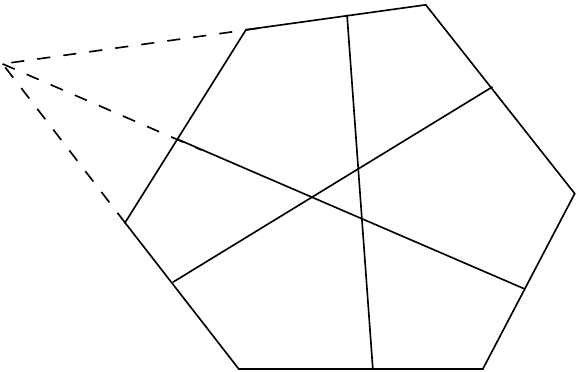}
\caption{Decomposition of $\ir(V)$, with $V$ obtained by perturbing the normals of a triangular bipyramid.}
\label{fig:BipyrPerturb}
\end{figure}

In the metric space $M_{\cC}(V)$, the three intersection points on Figure \ref{fig:BipyrPerturb} are singular in general. Indeed, if the perturbation is small, then the angles of the polygons $H(\Delta)$ at these points are close to $\frac{\pi}3$ or $\frac{2\pi}3$. Thus the second of the conditions of non-singularity above is not fulfilled. The first condition means that crossing one of the 1-dimensional strata in the interior doesn't change the quadratic form $q$. As we have seen in Section \ref{sec:TrBiHyp}, this is equivalent to the coefficient sum of the corresponding circuit being zero. For a generic perturbation this is not the case.
\end{exl}

\subsection{Questions}
There is definitely more to say about the metric space $M_{\cC}(V)$ in general.

\begin{conj}
The subsets $M_{\cC}^C$ of $\partial M_{\cC}(V)$ defined in Section \ref{sec:DihAngles} are orthogonal to all singular strata. This would allow to call them facets of $M_{\cC}(V)$.
\end{conj}

Let $C_1, C_2 \subset [n]$ be two circuits of the vector configuration $V$, and let $\Delta$ be a polytopal fan that contains $\pos(V_{C_1})$ and $\pos(V_{C_2})$ and is otherwise simplicial. In Section \ref{sec:IntAngles}, we have shown that the angle around $H_{\cC}(\Delta)$ may be different from $2\pi$. The following conjecture suggests that the situation is even worse.

\begin{conj}
Even if the circuits $C_1$ and $C_2$ are disjoint (and even if they are ``sufficiently far apart''), the four dihedral angles at $H_{\cC}(\Delta)$ may be different from $\frac{\pi}2$. Moreover, these angles are not determined by local data. That is, changing a vector $v_i$ with $i \notin C_1 \cup C_2$ can change the values of these angles.
\end{conj}


A generalization of the triangular bipyramid example is the configuration of the $2d$ vectors
\begin{equation}
\label{eqn:D1Dd}
v_i^\pm = u_i \pm e_d, i = 1, \ldots, d,
\end{equation}
where $u_i$ are normals to a regular simplex in $\R^{d-1}$. The corresponding affine Gale diagram has the properties $\bar v_i^- = -\bar v_i^+$ and $\sum_i \bar v_i = 0$. It follows that
$$
\clir(V) = \pos(S \cap -S),
$$
where $S$ is a $(d-1)$-dimensional simplex with the barycenter at the coordinate origin. The case $d=3$ was dealt with in Section \ref{sub:bipyr1}. For $d = 4$ the Gale diagram is formed by the vertices of a $3$-cube, and $\clir(V)$ is the cone over the octahedron. The corresponding secondary polytope is related to the permutahedron, see \cite[Section 6.2.1]{LRS10} where a similar point configuration is analyzed.

\begin{pro}
For the vector configuration \eqref{eqn:D1Dd}, describe the metric structure of the space $M_{\cC}(V)$ in the case when $\cC = (B,\ldots, B)$, that is the quadratic form is given by the $(d-2)$-nd quermassintegral.
\end{pro}
For example, $d=4$ should yield the right-angled ideal octahedron.

The hyperbolic polyhedron $H_{\cC}(\Delta)$ has an especially nice structure, if the fan $\Delta$ is obtained by a sequence of stellar subdivisions from the normal fan of a tetrahedron. The corresponding polyhedron $P(h)$ is a multiply truncated tetrahedron, see Section \ref{sec:TriBiForm}.

\begin{pro}
Analyze the combinatorics and geometry of $H_{\cC}(\Delta)$, where $\Delta$ is a multiple stellar subdivision of the normal fan of a tetrahedron.
\end{pro}

%

In the following two problems we assume $\cC = (B,\ldots, B)$ and $d=3$, thus the quadratic form is the surface area of a 3-dimensional polytope. 

\begin{pro}
What is $M(V)$ when $V$ consists of normals to an octahedron? (The quadratic form seems to be independent of the choice of the type cone, thus $M(V)$ should be a 4-dimensional hyperbolic polytope with 8 facets.)

The same question when $V$ consists of normals to a regular bipyramid over an $n$-gon.
\end{pro}

\begin{pro}
Let $\Delta$ be the normal fan of the dodecahedron. Describe the corresponding hyperbolic $8$-dimensional polyhedron $H(\Delta)$ and compute its dihedral angles.
\end{pro}

\begin{pro}
Describe all vector configurations in $\R^3$ for which the surface area is given by a same quadratic form independent of a choice of a type cone.
\end{pro}
A trivial class of examples are monotypic polyhedra, see Remark \ref{rem:Monotypic}. For vectors in general position, the necessary and sufficient condition is probably that for all $(2,2)$-circuits the sum of coefficients is zero.

Finally, inspired by the results of \cite{BG92}, we pose the following problem.
\begin{pro}
Describe some of the known or construct new examples of hyperbolic Coxeter polyhedra that appear as $H_{\cC}(\Delta)$ or $M_{\cC}(V)$.
\end{pro}

Note that both $H_{\cC}(\Delta)$ and $M_{\cC}(V)$ can belong to any combinatorial type, see Remark \ref{rem:AnyType}. It is more hard to determine the possible values of their dihedral angles.

\section{Related work}
\label{sec:RelWork}
\subsection{The first weight space and the discrete Christoffel problem}
Here we explain the relation of Section \ref{sec:ContrEdge} with the first weight space of McMullen \cite{McM96}. Definitions and propositions below are taken from \cite{McM96}.

Let $\Delta$ be a complete polytopal fan in $\R^d$.
Pick a cone $\tau \in \Delta^{(d-2)}$ and consider all $\sigma \in \Delta^{(d-1)}$ such that $\tau \subset \sigma$. Denote by $v_{\sigma/\tau} \in \aff(\sigma)$ the inner unit normal to the facet $\tau$ of the cone $\sigma$, see Figure \ref{fig:ConeFace}, left.

\begin{lem}
For every $h \in \wT(\Delta)$ and every $\tau \in \Delta^{(d-2)}$ the following equality holds:
\begin{equation}
\label{eqn:MinkCond}
\sum_{\sigma: \sigma \supset \tau} \ell_\sigma(h) v_{\sigma/\tau} = 0
\end{equation}
\end{lem}
\begin{proof}
Indeed, $v_{\sigma/\tau}$ is the outer unit normal to the edge $F_\sigma(h)$ of the $2$-face $F_\tau(h)$. Thus equation \eqref{eqn:MinkCond} follows from
$$
\sum_{\sigma : \sigma \supset \tau} \ell_\sigma(h) e_\sigma = \sum_{i=1}^n (p_{i+1}(h) - p_i(h)) = 0,
$$
where $p_1, \ldots, p_n, p_{n+1} = p_1$ are the vertices of the polygon~$F_\tau(h)$ in a cyclic order. Figure \ref{fig:ConeFace} illustrates the case $d=3$.
\end{proof}

\begin{figure}[ht]
\centering
\begin{picture}(0,0)%
\includegraphics{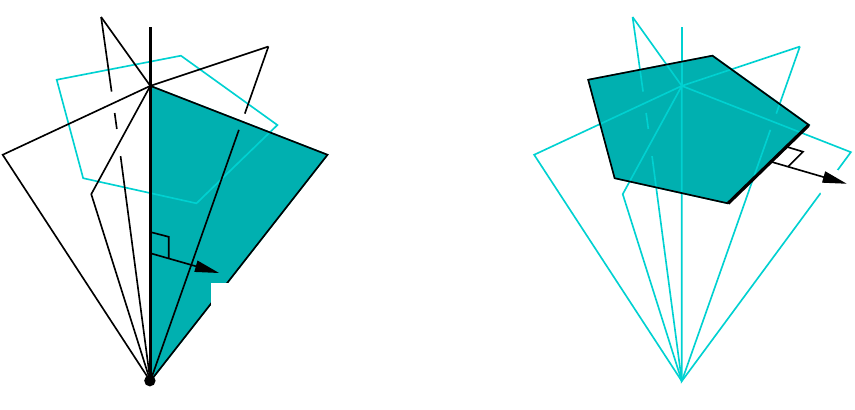}%
\end{picture}%
\setlength{\unitlength}{4144sp}%
\begingroup\makeatletter\ifx\SetFigFont\undefined%
\gdef\SetFigFont#1#2#3#4#5{%
  \reset@font\fontsize{#1}{#2pt}%
  \fontfamily{#3}\fontseries{#4}\fontshape{#5}%
  \selectfont}%
\fi\endgroup%
\begin{picture}(3902,1873)(214,-1991)
\put(1351,-1006){\makebox(0,0)[lb]{\smash{{\SetFigFont{10}{12.0}{\rmdefault}{\mddefault}{\updefault}{\color[rgb]{0,0,0}$\sigma$}%
}}}}
\put(3391,-725){\makebox(0,0)[lb]{\smash{{\SetFigFont{10}{12.0}{\rmdefault}{\mddefault}{\updefault}{\color[rgb]{0,0,0}$F_\tau$}%
}}}}
\put(946,-241){\makebox(0,0)[lb]{\smash{{\SetFigFont{10}{12.0}{\rmdefault}{\mddefault}{\updefault}{\color[rgb]{0,0,0}$\tau$}%
}}}}
\put(3634,-1073){\makebox(0,0)[lb]{\smash{{\SetFigFont{10}{12.0}{\rmdefault}{\mddefault}{\updefault}{\color[rgb]{0,0,0}$F_\sigma$}%
}}}}
\put(757,-1936){\makebox(0,0)[lb]{\smash{{\SetFigFont{10}{12.0}{\rmdefault}{\mddefault}{\updefault}{\color[rgb]{0,0,0}$0$}%
}}}}
\put(3976,-1098){\makebox(0,0)[lb]{\smash{{\SetFigFont{10}{12.0}{\rmdefault}{\mddefault}{\updefault}{\color[rgb]{0,0,0}$v_{\sigma/\tau}$}%
}}}}
\put(1171,-1501){\makebox(0,0)[lb]{\smash{{\SetFigFont{10}{12.0}{\rmdefault}{\mddefault}{\updefault}{\color[rgb]{0,0,0}$v_{\sigma/\tau}$}%
}}}}
\end{picture}%
\caption{The star of a $(d-2)$-cone $\tau \in \cN(P)$ corresponds to a $2$-face $F_\tau \subset P$.}
\label{fig:ConeFace}
\end{figure}
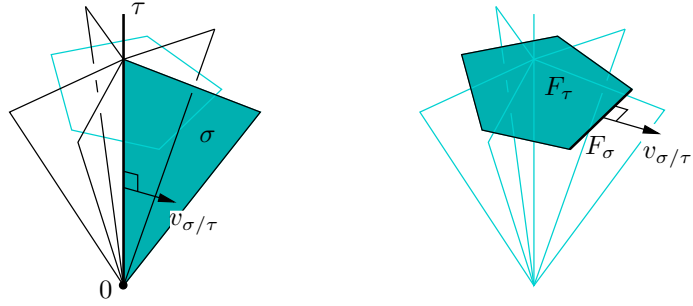

\begin{dfn}
A \emph{$1$-weight} on $\Delta$ is a map $\sigma \mapsto a_\sigma$ on the set of $(d-1)$-cones of $\Delta$ such that
$$
\sum_{\sigma : \sigma \supset \tau} a_\sigma v_{\sigma/\tau} = 0 \quad \text{for all } \tau \in \Delta^{(d-2)}
$$
The set of all $1$-weights is denoted by $\Omega_1(\Delta)$. Denote by
$$
\Omega_1^+(\Delta) := \{a \in \Omega_1(\Delta) \mid a_\sigma > 0 \quad \text{for all }\sigma\}
$$
the set of positive $1$-weights.
\end{dfn}

Since $\Delta$ is polytopal, we have $\Omega_\Delta^+ \ne \emptyset$. Besides, since $\Omega_\Delta^+$ is open in $\Omega_\Delta$, we have $\aff(\Omega_\Delta^+) = \Omega_\Delta$.

\begin{thm}
\label{lem christ}
 For any $a\in \Omega_1^+(\Delta)$, there exists a convex polytope $P\in T(\Delta) $ with $\ell_\sigma(P)=a_\sigma$.
 Moreover, $P$ is unique up to translation.
\end{thm}
\begin{proof}[Sketch of proof]
Choose an arbitrary point in $\R^d$ as a vertex of $P$ and follow the graph of $P$ to construct the other vertices. The fan $\Delta$ gives us the directions of the edges, and the weight $a$ gives their lengths. The weight condition ensures that this construction is well defined. One needs to check that the result is a convex polytope with the normal fan $\Delta$. This can be done with the help of Lemma \ref{lem:EllLin}: the construction yields a conewise linear function $h$ with the gradient jump across $\sigma$ equal to $a_\sigma$, the positivity of $a$ is equivalent to the convexity of $h$, which thus becomes a support function of a polytope with the normal fan $\Delta$.
\end{proof}

In terms of weights Lemma \ref{lem:TypeConeIneq} has the following reformulation.

\begin{lem}
\label{lem:HW}
The linear map
$$
\ell^\Delta \colon \span(T(\Delta)) \to \R^{\Delta^{(d-1)}}
$$
is injective and has image $\Omega_1(\Delta)$. Besides,
$$
\ell^\Delta(T(\Delta)) = \Omega_1^+(\Delta)
$$
\end{lem}

\begin{rem}\label{rem weight}
Theorem~\ref{lem christ} is \cite[Lemma 8.1]{McM96}.
It first appeared in \cite{she63} (see also  \cite[Theorem 15.1.2]{gru03}), in the context of decomposition of convex polytopes into Minkowski sum.

An element of $\Omega_1^+(\Delta)$ can be viewed as a Borelian measure on
$\mathbb{S}^{d-1}$. In this context, Theorem~\ref{lem christ} solves the polyhedral version of the Christoffel problem (find a convex body with a given first area measure),
and was proved in \cite{Sch77}, see also \cite{Scn93}.

Theorem~\ref{lem christ} appears as a part of \cite[Theorem 15.5]{PRW08} that contains also other characterizations of $T(\Delta)$.

The space $\Omega_1(\Delta)$ is the linear space of virtual polytopes with the ``normal fan'' $\Delta$, consult Figure \ref{fig:but-move}.
\end{rem}

\begin{rem}
More generally, an $r$-weight on $\Delta$ is a map $a \colon \Delta^{(d-r)} \to \R$ such that
$$
\sum_{\sigma : \sigma \supset \tau} a_\sigma v_{\sigma/\tau} = 0 \quad \text{for all} \quad \tau \in \Delta^{(d-r-1)}
$$
McMullen \cite{McM96} defined product of weights (modelled on mixed volumes) that gives rise to a graded algebra
$$
\Omega(\Delta) := \bigoplus_{r=0}^d \Omega_r(\Delta)
$$
He defined quadratic forms generalizing the form $q_{\cC,\Delta}$ (see Definition \ref{dfn:qForm}) and proved a generalization of the signature theorem \ref{cor:AFNormEq}. This led him to a proof of the so called $g$-theorem for simple polytopes (previously proved in \cite{Sta80} using heavy machinery from algebraic geometry).
\end{rem}

\subsection{Regular subdivisions of constant curvature surfaces with cone singularities}
\label{sec:Variations}
A configuration $V$ of unit vectors in $\R^d$ is a finite set of points on the unit sphere $\Sph^{d-1}$, and a polytopal fan $\Delta$ with $\Delta^{(1)} = V$ yields a subdivision of $\Sph^{d-1}$ with the vertex set $V$. If $\Delta$ is a polytopal fan, then the corresponding subdivision is called regular. For every $V$ there is a distinguished fan $\Delta_V$, the normal fan of the circumscribed polytope with $v_i$ as tangent points between the facets and the sphere. (Equivalently, this is the central fan of the convex hull of $V$.) It is easy to show that the subdivision corresponding to $\Delta_V$ is the \emph{Delaunay subdivision} of the sphere with the vertex set $V$. Here a Delaunay subdivision is one where every cell is an inscribed polygon and the circumcircle of every cell contains no vertices in its interior.

Delaunay subdivisions can be constructed for (Euclidean, spherical, or hyperbolic) surfaces with cone singularities. (In the spherical case there is a restriction on the metric, \cite[Lemma 2.11]{FI11}.) As a vertex set $V$, one can choose any finite set containing the cone points. In the Euclidean case the proof was sketched in \cite[Proposition 3.1]{Thu98}, and a more detailed treatment was given in \cite{ILTC01,BS07}.

Assigning to every point $v_i \in V$ a weight $w_i \in \R$ allows to define a \emph{weighted Delaunay subdivision} $\Delta_V(w)$ by requiring that the extension of the map $v_i \mapsto w_i$, piecewise linear with respect to $\Delta_V(w)$, is convex. Compare Corollary \ref{cor:PolFanChar} and Remark \ref{rem:RegSubdiv}, where the role of weights is played by the support numbers $h_i$. The weighted Delaunay subdivision with equal weights is the usual Delaunay subdivision. 

Weighted Delaunay triangulations of Euclidean cone-surfaces were introduced in \cite{BI08}, and those of hyperbolic and spherical cone-surfaces in \cite{FI09,FI11}.
In the non-Euclidean case instead of piecewise linear extension one uses functions of a different sort.

\subsection{Relation to Thurston's space of shapes of polyhedra}\label{sub:th}
Let $P\subset\R^3$ be a simple polytope with $n$ facets. 
From the Euler formula, it has $s:=2n-4$ vertices.
Then the metric on the boundary of $P$ 
induced from the ambient Euclidean space is a Euclidean metric on the sphere $\Sph^2$ with conical singularities
of positive curvature. The cone angle around a vertex of $P$ 
is the sum of the adjacent face angles.

Let us denote by $C(\alpha)$, where $\alpha=(\alpha_1, \ldots, \alpha_s)$, 
the set of Euclidean metrics  on
the sphere with (marked) cone singularities of angles $\alpha_i$,
up to orientation-preserving similarity.
If we fix a polytopal simplicial fan $\Delta$, then all polytopes from $T(\Delta)$ have the same cone-angles.
Thus we have a map
\begin{equation}
\label{eqn:TtoC}
T(\Delta) \to C(\alpha)
\end{equation}
Together with the combinatorics $\Delta$, the induced metric determines the edge lengths of a polytope $P(h) \in T(\Delta)$, and this determines the polytope according to Theorem \ref{lem christ}. Therefore the map \eqref{eqn:TtoC} is injective on any affine slice of the type cone $T(\Delta)$.

The space $C(\alpha)$ can be endowed 
 with a structure of a
complex manifold of dimension $s-3$ as follows.
Any $m\in C(\alpha)$ can be 
geodesically triangulated so that the singularities are exactly
the vertices of the triangulation. After fixing the position of a vertex and the direction of one of its adjacent edges, the triangulation can be developed in $\mathbb{C}$.
In this way, to each edge a complex number is associated, and it can be shown that certain $s-2$ of these numbers suffice  
to recover the triangulation. Modulo scaling we have $s-3$ complex parameters.
This gives a local chart for $C(\alpha)$ around $m$. Changes of charts
corresponds to flip of the triangulation and are linear maps in the coordinates.
Any chart can be endowed with the restriction of a Hermitian form on $\mathbb{C}^{s-2}$, which
is given by the area of the Euclidean metrics (the sum of the area of each triangle).
This makes $C(\alpha)$ to a complex hyperbolic manifold of dimension $s-3$. See \cite{Thu98} for details.

It follows that the interior of the hyperbolic polyhedron $H(\Delta)$ associated with $T(\Delta)$ (see Section \ref{sec:HypPol}) embeds isometrically in $C(\alpha)$:
$$
\int H(\Delta) \subset C(\alpha)
$$
Note that $\int H(\Delta)$ has real dimension $\frac{s}2 - 2$, while $C(\alpha)$ has complex dimension $s-3$.

Also note that changing the type cone $\Delta$ while preserving the set of facet normals changes in general the collection of angles $\alpha$, see Figure~\ref{fig:flip-angle}.

\begin{figure}
\centering
\includegraphics[scale=0.4]{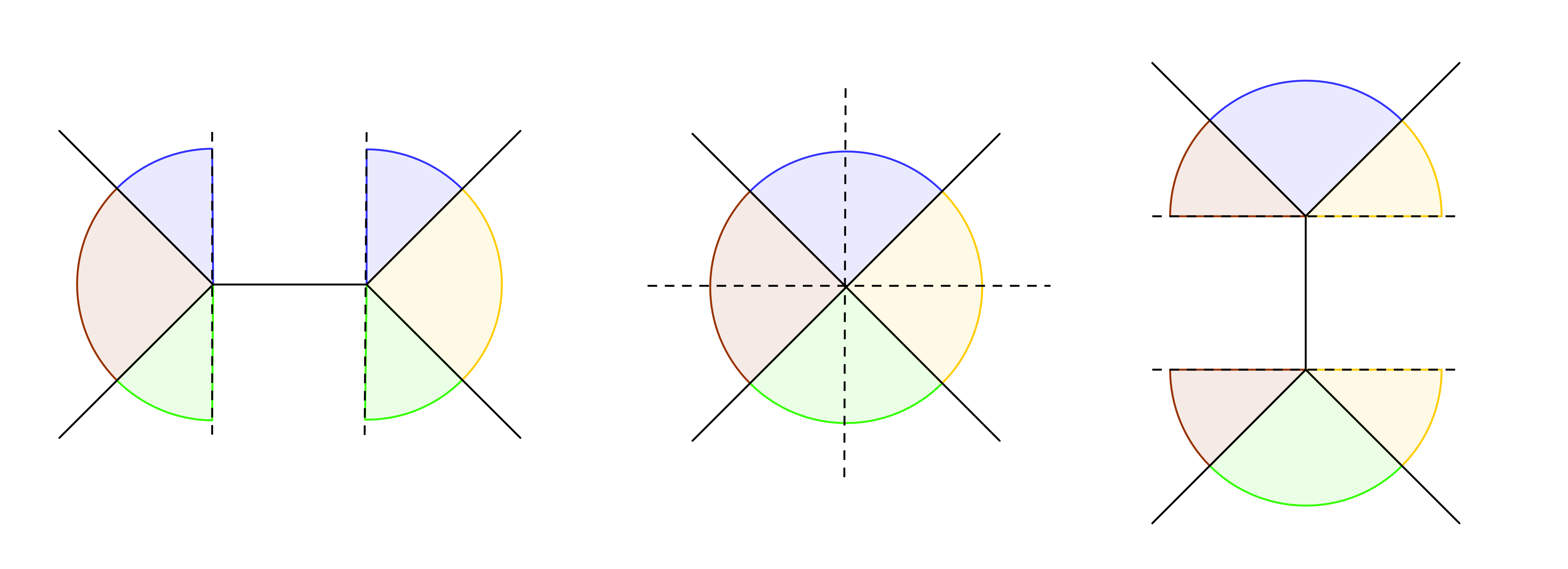}  
\caption{Cone angles of a simple polytope depend not only
on the facet normals but also on the combinatorics.
\label{fig:flip-angle}
} 
  \end{figure}



\subsection{Shape of co-convex polyhedra, mixed covolumes and spherical geometry}

The following definitions and results are from \cite{KT13}, 
restricted to the polyhedral case.

Let $C$ be a simple pointed convex polyhedral cone.
A $C$-convex (or \emph{co-convex}) polyhedron $P$ is a convex polyhedral subset of $C$ such that $x+C\subset P$ for all $x\in P$.
If the volume
of $C\setminus P$ is finite, then it is called the \emph{covolume} of $P$.
Multiplication by positive scalars and Minkowski addition are well-defined for
$C$-convex polyhedra.
This leads to the notion of the mixed covolume for simple normally equivalent $C$-convex polyhedra. As shown in \cite{KT13}, the mixed covolume satisfies the inverse Alexandrov-Fenchel inequalities.

It follows that certain quadratic forms associated with the covolume in the way described in Section \ref{sec:AFSign} are positive definite. Thus, in the co-convex case the types cones give rise to 
convex spherical polyhedra.

A different point of view on $C$-convex sets is presented in \cite{Fil13}.
View $\R^d$  as the Lorentzian Minkowski space (that does not change notions of convexity,
polyhedrality or volume).
Take a convex polyhedral cone $C$ that is a fundamental 
domain for a cocompact action of a group of linear isometries (in other terms of a subgroup of the isometry group of the hyperbolic space $\mathbb{H}^{d-1}$).
The case $d=2$, that is a Lorentzian analog of the Bavard-Ghys construction \cite{BG92}  was analyzed in  \cite{FillEdM}. 
The case $d=3$ is of a special interest, since the boundary of a $C$-convex polyhedron quotiented by the group action is a closed surface of genus $>1$, and the induced metric on it is a euclidean metric with cone singularities of negative curvature.
Spaces of such metrics were studied for example in  \cite{Vee93, Tro07, NO12}.
The construction of the present paper would yield
spherical convex polyhedra isometrically embedded in spaces of flat metrics on compact surfaces.


\begin{appendix}
\section{Computing the surface area of a 3-dimensional polytope}
\subsection{Face areas of a tetrahedron}
\label{sec:CompDecTetr}
Let $P \subset \R^3$ be a polytope whose facet normals belong to a set $V$. Then $\area(\partial P)$ can be represented as a sum of face areas of certain tetrahedra with facet normals in the set $V$. Indeed, this is true for a truncated tetrahedron; any other combinatorial type can be obtained by a sequence of flips, and a flip adds two faces of a tetrahedron and subtracts other two. For an illustration see Section \ref{sec:TriBiForm}. Lemma \ref{lem:FaceAreasFromV} below expresses the face areas of a tetrahedron in terms of its face normals and support numbers.

Let $v_0, v_1, v_2, v_3$ be positively spanning vectors in $\R^3$. (We are not assuming them to have length $1$.) There is a unique up to scaling positive linear dependency
$$
\lambda_0 v_0 + \lambda_1 v_1 + \lambda_2 v_2 + \lambda_3 v_3 = 0, \quad \lambda_i > 0\, \forall i
$$
Denote by
$$
\begin{aligned}
\Delta_0 := \{x \in \R^3 \mid &\langle v_0, x \rangle \le \|v_0\|,\\
&\langle v_i, x\rangle  \le 0 \text{ for } i = 1,2,3\}
\end{aligned}
$$
the tetrahedron with facet normals $v_0, v_1, v_2, v_3$ such that its altitude with respect to the face $F_0$ has length $1$.

\begin{lem}
\label{lem:FaceAreasFromV}
The face areas of $\Delta_0$ satisfy the following relations.
\begin{multline}
\frac{\area(F_0)}{\|v_0\|} : \frac{\area(F_1)}{\|v_1\|} : \frac{\area(F_2)}{\|v_2\|} : \frac{\area(F_3)}{\|v_3\|} = \lambda_0 : \lambda_1 : \lambda_2 : \lambda_3 \label{eqn:Proport} \\
= \det(v_1, v_2, v_3) : -\det(v_0, v_2, v_3) : \det(v_0, v_1, v_3) : -\det(v_0, v_1, v_2)
\end{multline}
\begin{equation}
\label{eqn:F0}
\area(F_0) = \frac{\|v_0\|^3 (\det(v_1, v_2, v_3))^2}{2 |\det(v_0,v_1,v_2) \det(v_0,v_2,v_3) \det(v_0,v_1,v_3)|}
\end{equation}
\end{lem}
\begin{proof}
Equation \eqref{eqn:Proport} follows from the Minkowski formula
$$
\area(F_0) \frac{v_0}{\|v_0\|} + \area(F_1) \frac{v_1}{\|v_1\|} + \area(F_2) \frac{v_2}{\|v_2\|} + \area(F_3) \frac{v_3}{\|v_3\|} = 0
$$
and from the elementary linear algebra.

To prove \eqref{eqn:F0}, let $e_{12}$ denote the vector along the edge $F_{12}$ of $\Delta_0$, directed towards the face $F_0$. Then we have
$$
e_{12} = \frac{\|v_0\| (v_1 \times v_2)}{\det(v_0, v_1, v_2)}
$$
It is easy to show that $(v_1 \times v_2) \times (v_1 \times v_3) = \det(v_1, v_2, v_3) v_1$. This implies
$$
\det(v_1 \times v_2, v_2 \times v_3, v_3 \times v_1) = (\det(v_1,v_2,v_3))^2
$$
Therefore
$$
\vol(\Delta_0) = \frac16 |\det(e_{12}, e_{23}, e_{31})| = \frac{\|v_0\|^3 (\det(v_1, v_2, v_3))^2}{6 |\det(v_0,v_1,v_2) \det(v_0,v_2,v_3) \det(v_0,v_1,v_3)|}
$$
The formula for the area of $F_0$ follows from $\vol(\Delta_0) = \frac13 h_0 \area(F_0)$.
\end{proof}

Formulas of Lemma \ref{lem:FaceAreasFromV} imply the following formula for the area of the face $F_1$:
$$
\area(F_1)|_{h_0=1} = \frac{\|v_0\|^2 \|v_1\| \det(v_1, v_2, v_3)}{2 \det(v_2, v_1, v_0) \det(v_1, v_3, v_0)}
$$
In practice, if the coefficients $\lambda_i$ of the linear dependency are known, it is easier to compute one of the areas by the above formulas for $F_0$ and $F_1$, and then the others using the proportions from the first part of the Lemma.

\subsection{Formula in terms of the angles in the normal fan}
\label{sec:App2}
Denote by $q_\Delta$ the quadratic form in variables $h_1, \ldots, h_n$ that computes the area of a polyhedron with normal fan $\Delta$.
We have 
$$
q_\Delta(h) = \sum_{i,j} a_{ij} h_i h_j,
$$
where $a_{ij} \ne 0$ if and only if $i = j$ or faces $F_i$ and $F_j$ are adjacent (in other terms $\pos\{v_i,v_j\} \in \Delta$). To determine the coefficients $a_{ij}$, compute the partial derivatives of $q_\Delta(h)$. We are using notations on Figure \ref{fig:Nothijk}.

\begin{figure}
\begin{center}
\begin{picture}(0,0)%
\includegraphics{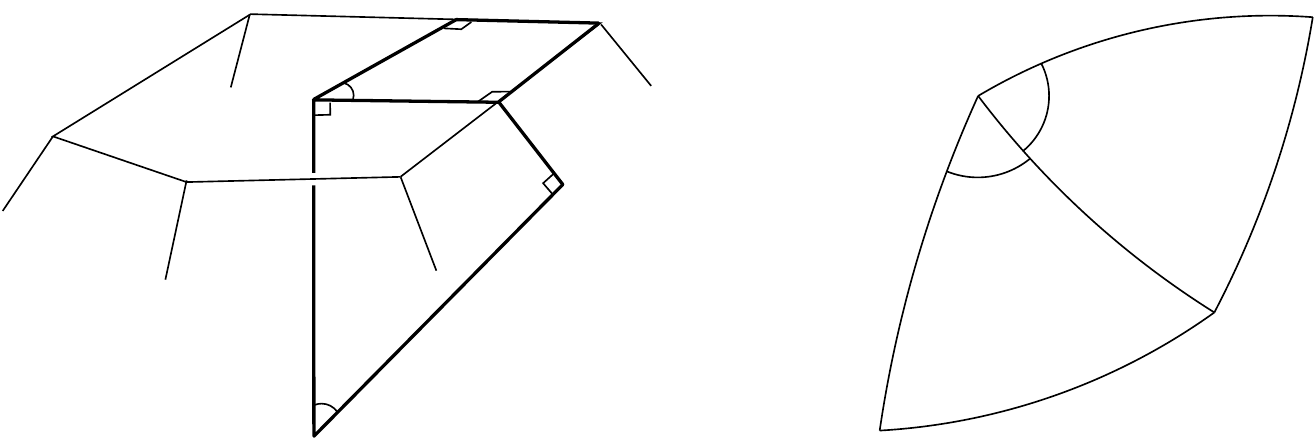}%
\end{picture}%
\setlength{\unitlength}{4144sp}%
\begingroup\makeatletter\ifx\SetFigFont\undefined%
\gdef\SetFigFont#1#2#3#4#5{%
  \reset@font\fontsize{#1}{#2pt}%
  \fontfamily{#3}\fontseries{#4}\fontshape{#5}%
  \selectfont}%
\fi\endgroup%
\begin{picture}(6010,2016)(2908,-1819)
\put(7830,-691){\makebox(0,0)[lb]{\smash{{\SetFigFont{8}{9.6}{\rmdefault}{\mddefault}{\updefault}{\color[rgb]{0,0,0}$\phi_{ij}$}%
}}}}
\put(7605, 74){\makebox(0,0)[lb]{\smash{{\SetFigFont{8}{9.6}{\rmdefault}{\mddefault}{\updefault}{\color[rgb]{0,0,0}$\phi_{ik}$}%
}}}}
\put(6750,-1006){\makebox(0,0)[lb]{\smash{{\SetFigFont{8}{9.6}{\rmdefault}{\mddefault}{\updefault}{\color[rgb]{0,0,0}$\phi_{il}$}%
}}}}
\put(7695,-376){\makebox(0,0)[lb]{\smash{{\SetFigFont{8}{9.6}{\rmdefault}{\mddefault}{\updefault}{\color[rgb]{0,0,0}$\phi_{ij,ik}$}%
}}}}
\put(7245,-781){\makebox(0,0)[lb]{\smash{{\SetFigFont{8}{9.6}{\rmdefault}{\mddefault}{\updefault}{\color[rgb]{0,0,0}$\phi_{ij,il}$}%
}}}}
\put(4190,-1126){\makebox(0,0)[lb]{\smash{{\SetFigFont{8}{9.6}{\familydefault}{\mddefault}{\updefault}{\color[rgb]{0,0,0}$h_i$}%
}}}}
\put(5074,-1197){\makebox(0,0)[lb]{\smash{{\SetFigFont{8}{9.6}{\familydefault}{\mddefault}{\updefault}{\color[rgb]{0,0,0}$h_j$}%
}}}}
\put(4927,-628){\rotatebox{40.0}{\makebox(0,0)[lb]{\smash{{\SetFigFont{8}{9.6}{\familydefault}{\mddefault}{\updefault}{\color[rgb]{0,0,0}$h_{ijl}$}%
}}}}}
\put(5432,-241){\rotatebox{40.0}{\makebox(0,0)[lb]{\smash{{\SetFigFont{8}{9.6}{\familydefault}{\mddefault}{\updefault}{\color[rgb]{0,0,0}$h_{ijk}$}%
}}}}}
\put(4677,-387){\makebox(0,0)[lb]{\smash{{\SetFigFont{8}{9.6}{\familydefault}{\mddefault}{\updefault}{\color[rgb]{0,0,0}$h_{ij}$}%
}}}}
\put(4368,-1551){\makebox(0,0)[lb]{\smash{{\SetFigFont{8}{9.6}{\familydefault}{\mddefault}{\updefault}{\color[rgb]{0,0,0}$\phi_{ij}$}%
}}}}
\put(4188,-179){\makebox(0,0)[lb]{\smash{{\SetFigFont{8}{9.6}{\familydefault}{\mddefault}{\updefault}{\color[rgb]{0,0,0}$\phi_{ij,ik}$}%
}}}}
\end{picture}%
\end{center}
\caption{Lengths and angles in a 3-dimensional polytope.}
\label{fig:Nothijk}
\end{figure}
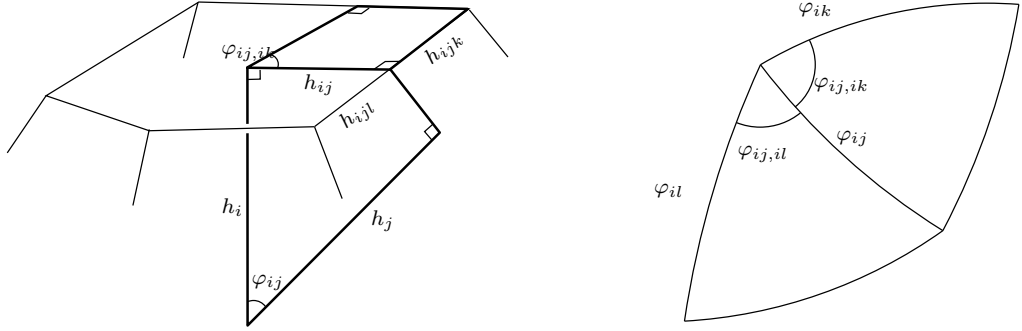

Let $q_i$ be the area of the $i$-th face. Then we have \cite{BI08}
\begin{gather*}
\frac{\partial q_j}{\partial h_i} = \frac{\ell_{ij}}{\sin\phi_{ij}} \text{ for }i\ne j,\\
\frac{\partial q_i}{\partial h_i} = - \sum_{j \ne i} \ell_{ij} \cot\phi_{ij},
\end{gather*}
where $\ell_{ij} = h_{ijk} + h_{ijl}$ is the length of the $ij$-edge.
Hence
$$
\frac{\partial q}{\partial h_i} = \sum_{j \ne i} \ell_{ij} \frac{1-\cos\phi_{ij}}{\sin\phi_{ij}} = \sum_{j \ne i} \ell_{ij} \tan \frac{\phi_{ij}}2
$$
Differentiating again, we obtain
$$
\frac{\partial^2 q}{\partial h_i \partial h_j} = \frac{\partial \ell_{ik}}{\partial h_j} \tan\frac{\phi_{ik}}2 + \frac{\partial \ell_{il}}{\partial h_j} \tan\frac{\phi_{il}}2 + \frac{\partial \ell_{ij}}{\partial h_j} \tan\frac{\phi_{ij}}2
$$
Using
$$
h_{ijk} = h_{ik} \csc\phi_{ij,ik} - h_{ij} \cot\phi_{ij,ik}, \quad h_{ij} = h_j \csc\phi_{ij} - h_i \cot\phi_{ij}
$$
we compute
\begin{gather*}
\frac{\partial \ell_{ik}}{\partial h_j} = \csc\phi_{ij} \csc\phi_{ij,ik}\\
\frac{\partial \ell_{ij}}{\partial h_j} = - \csc\phi_{ij} (\cot\phi_{ij,ik} + \cot\phi_{ij,il})
\end{gather*}
Therefore
\begin{multline*}
2a_{ij} = \frac{\partial^2 q}{\partial h_i \partial h_j} = \tan\frac{\phi_{ik}}2 \csc\phi_{ij} \csc\phi_{ij,ik} + \tan\frac{\phi_{il}}2 \csc\phi_{ij} \csc\phi_{ij,il}\\
- \tan\frac{\phi_{ij}}2 \csc\phi_{ij} (\cot\phi_{ij,ik} + \cot\phi_{ij,il})
\end{multline*}
$$
2a_{ii} = \frac{\partial^2 q}{\partial h_i^2} = - \sum_{j \ne i} \tan\frac{\phi_{ij}}2 \csc\phi_{ij} (\cot\phi_{ji,jk} + \cot\phi_{ji,jl})
$$

\begin{exl}
For a dodecahedron we have $\cos\phi_{ij} = \frac1{\sqrt{5}}$, $\cos\phi_{ij,ik} = \frac{\sqrt{5}-1}4$. The coefficients of the quadratic form are
$$
a_{ij} = \tan\frac{\phi_{ij}}2 \csc\phi_{ij} (\csc\phi_{ij,ik} - \cot\phi_{ij,ik}), \quad a_{ii} = -5 \tan\frac{\phi_{ij}}2 \csc\phi_{ij} \cot\phi_{ij,ik}
$$
Hence $a_{ij} = C (1 - \cos\phi_{ij,ik}), a_{ii} = - 5C \cos\phi_{ij,ik}$ and $\frac{a_{ij}}{a_{ii}} = -\frac1{\sqrt{5}}$. Up to scaling, the quadratic form is thus equal
$$
q(h) = - \sum_i h_i^2 + \frac{2}{\sqrt{5}} \sum_{\{i,j\}} h_i h_j
$$
\end{exl}

\end{appendix}

\def\cprime{$'$} \def\cprime{$'$}


\end{document}